\documentclass{amsart}
\usepackage{amssymb,amstext,amsmath,amscd,amsthm,amsfonts,enumerate,latexsym, comment}
\usepackage{color}
\usepackage[dvipdfmx]{graphicx}
\usepackage[all]{xy}

\theoremstyle{plain}
\newtheorem{thm}{Theorem}[section]

\newtheorem*{thm*}{Theorem}
\newtheorem*{cor*}{Corollary}

\newtheorem{prop}[thm]{Proposition}
\newtheorem{proposition}[thm]{Proposition}
\newtheorem{lemma}[thm]{Lemma}
\newtheorem{lem}[thm]{Lemma}
\newtheorem{cor}[thm]{Corollary}
\newtheorem{corollary}[thm]{Corollary}

\newtheorem*{claim*}{Claim}
\newtheorem*{condition}{Condition}

\theoremstyle{definition}
\newtheorem{defn}[thm]{Definition}

\newtheorem{ex}[thm]{Example}
\newtheorem{Example}[thm]{Example}
\newtheorem{rem}[thm]{Remark}
\newtheorem{remark}[thm]{Remark}
\newtheorem{fact}[thm]{Fact}

\newtheorem{setup}[thm]{Setup}
\newtheorem{setting}[thm]{Setup}

\newtheorem{conv}[thm]{Convention}
\newtheorem*{acknowledgments}{Acknowledgments}

\theoremstyle{remark}

\newtheorem*{proof of claim}{{\sl Proof of Claim}}

\numberwithin{equation}{thm}

\def\Ext{\operatorname{Ext}}

\def\Im{\operatorname{Im}}

\def\Ker{\operatorname{Ker}}

\def\Hom{\operatorname{Hom}}

\def\Max{\operatorname{Max}}

\def\Soc{\mathrm{Soc}}

\def\mod{\mathrm{mod}}

\def\Ker{\mathrm{Ker}}
\def\Im{\mathrm{Im}}

\def\tr{\mathrm{tr}}

\def\rank{\mathrm{rank}}

\def\e{\mathrm{e}}
\def\m{\mathfrak m}

\newcommand{\rma}{\mathrm{a}}

\newcommand{\rmc}{\mathrm{c}}

\newcommand{\rme}{\mathrm{e}}
\newcommand{\rmf}{\mathrm{f}}

\newcommand{\rmr}{\mathrm{r}}

\newcommand{\rmv}{\mathrm{v}}

\newcommand{\rmI}{\mathrm{I}}
\newcommand{\rmJ}{\mathrm{J}}
\newcommand{\rmK}{\mathrm{K}}

\newcommand{\rmQ}{\mathrm{Q}}

\newcommand{\calR}{\mathcal{R}}

\newcommand{\calX}{\mathcal{X}}

\newcommand{\fka}{\mathfrak{a}}
\newcommand{\fkb}{\mathfrak{b}}
\newcommand{\fkc}{\mathfrak{c}}

\newcommand{\fkm}{\mathfrak{m}}
\newcommand{\fkn}{\mathfrak{n}}

\newcommand{\fkp}{\mathfrak{p}}
\newcommand{\fkq}{\mathfrak{q}}

\newcommand{\fkM}{\mathfrak{M}}
\newcommand{\fkN}{\mathfrak{N}}

\newcommand{\mapright}[1]{%
\smash{\mathop{%
\hbox to 1cm{\rightarrowfill}}\limits^{#1}}}

\newcommand{\mapleft}[1]{%
\smash{\mathop{%
\hbox to 1cm{\leftarrowfill}}\limits_{#1}}}

\def\depth{\operatorname{depth}}
\def\AGL{\operatorname{AGL}}
\def\GGL{\operatorname{GGL}}
\def\NGL{\operatorname{NGL}}

\def\Ass{\operatorname{Ass}}
\def\height{\mathrm{ht}}
\def\grade{\mathrm{grade}}
\def\Spec{\operatorname{Spec}}

\def\ol{\overline}

\title[On $\GGL$ rings]{On generalized Gorenstein local rings}
\author[S. Goto]{Shiro Goto}
\address{Department of Mathematics, School of Science and Technology, Meiji University, 1-1-1 Higashi-mita, Tama-ku, Kawasaki 214-8571, Japan}
\email{shirogoto@gmail.com}

\author[S. Kumashiro]{Shinya Kumashiro}
\address{Department of Mathematics, Osaka Institute of Technology, 5-16-1 Omiya, asahi-ku, Osaka, 535-8585, Japan}
\email{shinya.kumashiro@oit.ac.jp}
\email{shinyakumashiro@gmail.com}


\thanks{The second author was supported by JSPS KAKENHI Grant Number JP24K16909 and by Grant for Basic Science Research Projects from the Sumitomo Foundation (Grant number 2200259).}

\begin{document}

\begin{abstract}
In this paper, we introduce generalized Gorenstein local (GGL) rings. The notion of GGL rings is a natural generalization of the notion of almost Gorenstein rings, which can thus be treated as part of the theory of GGL rings. For a Cohen-Macaulay local ring $R$, we explore the endomorphism algebra of the maximal ideal, the trace ideal of the canonical module, Ulrich ideals, and Rees algebras of parameter ideals in connection with the GGL property. We also give numerous examples of numerical semigroup rings, idealizations, and determinantal rings of certain matrices.
\end{abstract}

\maketitle

\section{Introduction}\label{section0}

The purpose of this paper is to develop the study of non-Gorenstein Cohen-Macaulay rings. 
The notion of Gorenstein rings is introduced in a famous paper \cite{Bass} of H. Bass. After his work, Gorenstein rings have been studied from various perspectives. On the other hand, one can expect some results of Gorenstein rings can be generalized to these of Cohen-Macaulay rings with additional conditions. For instance, the theory of the canonical module is a shining result of its success. 

Almost Gorenstein rings are one of the most interesting candidates for the framework to study non-Gorenstein Cohen-Macaulay rings. The notion of almost Gorenstein rings originates from V. Barucci and R. Fr\"{o}berg \cite{BF}, who studied  one-dimensional analytically unramified Cohen-Macaulay local rings. After the work of \cite{BF}, the first author, N. Matsuoka, and T. T. Phuong \cite{GMP} developed the theory of almost Gorenstein rings of dimension one. Currently, the notion of almost Gorenstein rings can be defined in arbitrary Cohen-Macaulay local/graded rings by the work \cite{GTT} of the first author, R. Takahashi, and N.~Taniguchi through the notion of Ulrich modules. Let us recall the definition of almost Gorenstein local rings.  

\begin{defn}{\rm (\cite[Definition 1.1]{GTT})}\label{zzz2.1}
Let $(R, \fkm)$ be a Cohen-Macaulay local ring of dimension $d$ possessing the canonical module $\rmK_R$. Then, we say that $R$ is an {\it almost Gorenstein local ring} ($\AGL$ ring) if there exists an exact sequence
$$
0 \to R \to \rmK_R \to C \to 0
$$
of $R$-modules such that either $C = (0)$ or $C$ is an Ulrich $R$-module with respect to $\fkm$.
\end{defn}

Here, for a finitely generated $R$-module $C$ and an $\fkm$-primary ideal $\fka$, $C$ is called an {\it Ulrich $R$-module with respect to} $\fka$ if the following three conditions are satisfied (cf. \cite[Definition 1.2]{GOTWY}):
\begin{enumerate}[{\rm (1)}]
\item $C$ is a nonzero Cohen-Macaulay $R$-module (i.e., $\depth_R C=\dim_R C$),
\item $\rme_{\fka}^0 (C)=\ell_R (C/\fka C)$, and
\item $C/\fka C$ is an $R/\fka$-free module, 
\end{enumerate}
where $\ell_R(*)$ denotes the length and $\e^0_\fka(C)$ denotes the multiplicity of $C$ with respect to $\fka$.

The condition of Definition of \ref{zzz2.1} requires that $R$ is embedded into $\rmK_R$ even though $R\ne \rmK_R$, the difference $C=\rmK_R/R$ between $\rmK_R$ and $R$ is an Ulrich $R$-module with respect to $\fkm$ and behaves well.
Indeed, it is known that $\AGL$ rings enjoy interesting properties. Let $R$ be a non-Gorenstein $\AGL$ ring, and let $M$ be a finitely generated $R$-module.
Then, $\Ext_R^i (M, R)=0$ for all $i>0$ implies that $M$ is a free $R$-module (\cite[Corollary 4.5]{GTT} and \cite[Corollary 4.6]{AI}). In particular, $R$ is $G$-regular in the sense of \cite{Tak}, that is, every totally reflexive module is free. Moreover, the $\AGL$ property is related to the Gorensteinness of the blow-up at the maximal ideal in dimension one (\cite[Proposition 25]{BF} and \cite[Theorem 5.1]{GMP}), and all the known Cohen-Macaulay local rings of finite representation type are $\AGL$ rings (\cite[Sections 11 and 12]{GTT}).
The conditions under which (numerical) semigroup rings, idealizations, Rees algebras of certain ideals/modules, determinantal rings are almost Gorenstein rings have also been studied (\cite{BF, GK, GTT, HJS, T} and so on). 
In contrast to the helpful properties of $\AGL$ rings, it is also known that the class of $\AGL$ rings is still tight as a framework to study non-Gorenstein Cohen-Macaulay rings. For instance, if $R$ is a non-Gorenstein $\AGL$ ring, then $R[x]/(x^n)$, where $R[x]$ is a polynomial ring over $R$ and $n\ge 2$, is no longer an $\AGL$ ring (\cite[Proposition 3.12]{GTT}). 
In addition, recently, another candidate for the framework of non-Gorenstein Cohen-Macaulay rings, say {\it nearly Gorenstein local rings} ($\NGL$ rings) are introduced by J. Herzog, T. Hibi, and D. I. Stamate (\cite{HHS}). With this background, a theory improving the theory of $\AGL$ rings and unifying these works is expected.

In this paper, we attempt to unify these theories based on the $\AGL$ theory. That is, 
we introduce the notion of {\it generalized Gorenstein rings} ($\GGL$ rings), and we deal with results on $\AGL$ rings as part of results on $\GGL$ rings. In conclusion, useful properties of $\AGL$ rings are naturally restated (Theorems \ref{def1.3}, \ref{def1.3.1} and Corollary \ref{a5.5}). Furthermore, we establish results that cannot be obtained in the framework of $\AGL$ rings (Theorems \ref{xxx15}, \ref{a7.5}, and \ref{a7.14}). 
Let us define $\GGL$ rings to express our results precisely.
In what follows, throughout this section, let $R$ be a Cohen-Macaulay local ring of positive dimension with the maximal ideal $\fkm$. Suppose that $R$ has the canonical module $\rmK_R$. To state our results simply, we here assume that the residue field $R/\fkm$ is infinite.

\begin{defn}(Definition \ref{a3.2}) 
Let $\fka$ be an $\fkm$-primary ideal of $R$. We say that $R$ is a {\it generalized Gorenstein local ring with respect to $\fka$} ($\GGL$ ring with respect to $\fka$) if there exists an exact sequence
$$0 \to R \xrightarrow{} \rmK_R \to C \to 0$$
of $R$-modules such that 
\begin{enumerate}
\item[$\mathrm{(i)}$] $C$ is an Ulrich $R$-module with respect to $\fka$ and
\item[$\mathrm{(ii)}$] $\rmK_R/\fka \rmK_R$ is $R/\fka$-free.
\end{enumerate}

We say that $R$ is a {\it $\GGL$ ring} if $R$ is a $\GGL$ ring with respect to some $\fkm$-primary ideal of $R$.
\end{defn}

Note that the notion of non-Gorenstein $\AGL$ rings is the same as that of non-regular $\GGL$ rings with respect to $\fkm$. 
If $R$ is a Gorenstein ring of positive dimension, then we can regard an exact sequence $0 \to R \xrightarrow{a_1} R \to R/(a_1) \to 0$ as a defining exact sequence of $\GGL$ rings with respect to $\fkq$, where $\fkq=(a_1, a_2, \dots, a_d)$ is a parameter ideal of $R$ (Remark \ref{rem3.4}). Hence, for Cohen-Macaulay local rings of positive dimension, we have the following implications:
\begin{center}
Gorenstein rings $\Rightarrow$ $\AGL$ rings $\Rightarrow$ $\GGL$ rings. 
\end{center}

The following assert the $\GGL$ property is inherited by the canonical change of rings.

\begin{thm}[Theorem \ref{a3.5}]\label{def1.3}
The following assertions hold true.
\begin{enumerate}[{\rm (1)}]
\item Let $R$ be a $\GGL$ ring with respect to $\fka$ of dimension $\ge 2$. Then, we can choose a non-zerodivisor $f\in \fka$ of $R$ such that $R/(f)$ is a $\GGL$ ring with respect to $\fka/(f)$.
\item Let $f\in \fkm$ be a non-zerodivisor of $R$ and suppose that $R/(f)$ is a $\GGL$ ring with respect to $[\fka+(f)]/(f)$. Then, $R$ is a $\GGL$ ring with respect to $\fka+(f)$ and $f\not\in \fkm \fka$.
\end{enumerate}
\end{thm}

\begin{thm}[Theorem \ref{aaaa72}]\label{def1.3.1}
Let $\psi : R \to S$ be a flat local homomorphism of Noetherian local rings such that $S/\fkm S$ is a Cohen-Macaulay local ring. Let $J\subseteq S$ be a parameter ideal in $S/\fkm S$. Then the following are equivalent.
\begin{enumerate}[{\rm (1)}]
\item $R$ is a $\GGL$ ring with respect to $\fka$, and $S/\fkm S$ is a Gorenstein ring.
\item $S$ is a $\GGL$ ring with respect to $\fka S+ J$.
\end{enumerate}
\end{thm}

In particular, if $R$ is a $\GGL$ ring, then $R[x]/(x^n)$ is also a $\GGL$ ring for every $n>0$ by Theorem \ref{def1.3.1}. Note that this is already over the framework of $\AGL$ rings. Indeed, we cannot replace a $\GGL$ ring with an $\AGL$ ring in the assertion, see \cite[Proposition 3.12]{GTT}. 

Special attention is given in dimension one. By using Theorem \ref{def1.3} or other reasons, the heart of this paper is in dimension one. For instance, to prove Theorem \ref{def1.3.1}(2)$\Rightarrow$(1), we need the argument in dimension one (that's why Theorem \ref{def1.3.1} is proven in Section \ref{sec7} in contrast to Theorem \ref{def1.3} is proven in Section \ref{section2}). After illustrating our results in dimension one (Sections \ref{section3} and \ref{sec5}), we note the results in higher dimension.
Let $\dim R=1$. The first important result in dimension one is the defining $\fkm$-primary ideal of a $\GGL$ ring $R$ is to be the trace ideal of the canonical module 
\[
\tr_R(\rmK_R):=\sum_{f\in \Hom_R(\rmK_R, R)} \Im f
\]
(Corollary \ref{a7.1.1}). In particular, the defining ideal of a $\GGL$ ring is unique. We note that the trace ideal of the canonical module defines the non-Gorenstein locus of $R$ (e.g., see \cite[Lemma 2.1]{HHS}) and is also used to define the notion of $\NGL$ rings. Thus, it is an interesting problem to represent the trace ideal of the canonical module. 
On the other hand, as another problem, the property of the endomorphism algebra $B:=\Hom_R(\fkm, \fkm)$ of the maximal ideal of $R$ has been studied (for examples, \cite{Bass, L}). One of motivations of the study of $B$ is arising from a simple fact: if a finitely generated reflexive $R$-module $M$ has no free summands, then $M$ is a $B$-module (\cite[(7.2) Proposition]{Bass}). With these two problems, we obtain the following theorem which describes the structures of $B$ and $\tr_R(\rmK_R)$ for $\GGL$ rings $R$ having minimal multiplicity. 
Let $\rmv(R)$ and $\rme(R)=\rme_\fkm^0(R)$ denote the embedding dimension of $R$ and the multiplicity of $R$, respectively. 

\begin{thm}{\rm (Theorems \ref{a4.24} and \ref{a7.12})} \label{xxx15}
Suppose that there is an element $\alpha \in \fkm$ such that $\fkm^2 = \alpha \fkm$. 
The following conditions are equivalent:
\begin{enumerate}[{\rm (1)}]
\item $R$ is a $\GGL$ ring but not an $\AGL$ ring.
\item $B:=\Hom_R(\fkm, \fkm)$ is a $\GGL$ ring with $\rmv(B)=\rme(B)=\rme(R)$ but not a Gorenstein ring.
\end{enumerate}
In particular, $B$ is a local ring. When the equivalent conditions hold, by setting $\fkn$ the maximal ideal of $B$, we further have the following:
\begin{enumerate}[{\rm (i)}]
\item $R/\fkm \cong B/\fkn$.
\item $\ell_B(B/\tr_B(\rmK_B))=\ell_R(R/\tr_R(\rmK_R))-1$.
\item There exist elements $x_2, x_3, \dots, x_{\rmv(R)} \in \fkm$ satisfying the following conditions. Set $n=\ell_R(R/\tr_R(\rmK_R))$. 
\begin{enumerate}[{\rm (a)}]
\item $\fkm=(\alpha, x_2, x_3, \dots, x_{\rmv(R)})$,
\item $\tr_R(\rmK_R)=(\alpha^n, x_2, x_3, \dots, x_{\rmv(R)})$, and $(\alpha^n)$ is a minimal reduction of $\tr_R(\rmK_R)$.
\end{enumerate}
\end{enumerate}
\end{thm}

We should emphasize that we assume that $R$ is not an $\AGL$ ring in (1) of Theorem \ref{xxx15}, and (1) $\Rightarrow$ (2) of Theorem \ref{xxx15} does not hold if  $R$ is an $\AGL$ ring (Example \ref{a4.22.5}).
On the other hand, it is known that $R$ is an $\AGL$ ring having minimal multiplicity if and only if $B$ is a (not necessarily local) Gorenstein ring (\cite[Theorem 5.1.]{GMP}). By combining this result with Theorem \ref{xxx15}, a $\GGL$ ring having minimal multiplicity reaches a Gorenstein ring by taking the endomorphism algebra of the maximal ideal recursively. 

We can also find a relation between $\GGL$ rings and Ulrich ideals. The notion of Ulrich ideals was introduced by the first author, K. Ozeki, R. Takahashi, K. Watanabe, and K. Yoshida \cite{GOTWY}. It is shown that Ulrich ideals present interesting properties, for examples, see \cite{GOTWY, GTT2}. 
We summarize the definition and the fundamental properties of Ulrich ideals in Subsection \ref{subsection22}. In addition, roughly to say, we prove that most of Ulrich ideals contain $\tr_R(\rmK_R)$ (Proposition \ref{a7.2}). Thus, we pose a natural question of when $\tr_R(\rmK_R)$ is an Ulrich ideal. With the question, we obtain that $\tr_R(\rmK_R)$ is an Ulrich ideal if and only if $R$ is a $\GGL$ ring and the endomorphism algebra $\Hom_R(\tr_R(\rmK_R), \tr_R(\rmK_R))$ is Gorenstein (Theorem \ref{a7.5}). We further determine the set of Ulrich ideals of $\GGL$ rings having minimal multiplicity (Theorem \ref{a7.14}). 
In addition, we explore the $\GGL$ property in idealizations and numerical semigroup rings (Subsection \ref{subsection4.4} and Section \ref{sec5}). We give criteria for being a $\GGL$ ring in several cases (Proposition \ref{a4.15} and Theorems \ref{a4.29} and \ref{a4.33}). As an application, numerous examples of $\GGL$ rings are obtained.

After the arguments in dimension one (Sections \ref{section3} and \ref{sec5}), the remainder of this paper is organized as follows. 
In Section \ref{section5}, we study the following condition. Let $(R, \fkm)$ be a Cohen-Macaulay local ring of dimension $>0$ possessing the canonical module $\rmK_R$. Let $r=\rmr(R)$ denote the Cohen-Macaulay type of $R$.

\begin{condition}
There exists an exact sequence 
\[
0 \to R \to \rmK_R \to \bigoplus_{i=2}^r R/\fka_i \to 0
\]
of $R$-modules, where $\fka_i$ is an ideal of $R$ for all $2\le i \le r$.
\end{condition}

Note that this condition is automatically satisfied if $R$ is a generically Gorenstein ring of $r=2$, $R$ is a $\GGL$ ring of dimension one, or $R$ is a $2$-$\AGL$ ring in the sense of \cite{CGKM} (Remark \ref{a5.1}). Suppose that there exists a surjective ring homomorphism $\varphi:S \to R$ from a Gorenstein local ring $S$ such that the projective dimension of $R$ over $S$ is finite (this condition is satisfied if $R$ is complete). Then, we characterize Condition in terms of the minimal $S$-free resolution of $R$.

\begin{thm}{\rm (Theorem \ref{a5.2})}\label{xxx18}
Let $(S, \fkn)$ be a Gorenstein local ring and $I, \fka_2, \fka_3, \dots, \fka_r$ be ideals of $S$. Suppose that $R = S/I$ is a Cohen-Macaulay ring but not a Gorenstein ring. Assume that the projective dimension of $R$ over $S$ is finite. 
Then,  the following conditions are equivalent:
\begin{enumerate}[{\rm (1)}]
\item There exists an exact sequence 
$$0\to R \to \rm{K}_R \to \bigoplus _{i=2}^r S/\fka_i \to 0 $$
of $S$-modules.
\item There exist a minimal $S$-free resolution $0\to S^{\oplus r}  \xrightarrow{\mathbb{M}} S^{\oplus q} \to \dots \to S \to R \to 0$
 of $R$ and a non-negative integer $m$ such that
\begin{equation*}
{}^t\!{\mathbb{M}}=\left(
\begin{array}{cccc|c}
y_{21}~y_{22} ~\cdots ~y_{2u_2} & y_{31}~y_{32} ~\cdots ~y_{3u_3} & \cdots & y_{r1}~y_{r2} ~\cdots ~y_{ru_r} & z_1~z_2 ~\cdots ~z_m \\ \hline
x_{21}~x_{22} ~\cdots ~x_{2u_2} &  & & & \\
 & x_{31}~x_{32}~ \cdots~ x_{3u_3} & & \mbox{\huge{0}} & \\
 & & \ddots  & & \mbox{\huge{0}} \\
\mbox{\huge{0}} & & & x_{r1}~x_{r2}~ \cdots ~x_{ru_r} & \\
\end{array}
\right) ,
\end{equation*}
where $\mu_S (\fka_i)=u_i$, $\fka_i=(x_{i1}, x_{i2}, \dots, x_{iu_i})$, and $\dim S/\fka_i=\dim R-1$ for all $2\le i \le r$.
\end{enumerate}

Furthermore, if $x_{i1}, x_{i2}, \dots, x_{iu_i}$ is an $S$-regular sequence for all $2\le i \le r$, then we have the equality
\begin{align}\label{tousiki}
I=\sum_{i=2}^r  I_2\left(\begin{matrix}
y_{i1}&y_{i2} & \cdots &y_{iu_i}\\
x_{i1}&x_{i2} & \cdots &x_{iu_i}\\
\end{matrix}
\right)  + (z_1, z_2, \dots, z_m).
\end{align}
\end{thm}

As a consequence of Theorem \ref{xxx18}, we characterize a one-dimensional $\GGL$ rings $R$ in terms of the form of a minimal $S$-free resolution of $R$. This characterization is used in Section \ref{section6}.
In Section \ref{sec7}, we give several results for higher dimensional $\GGL$ rings obtained from the results in dimension one. In particular, we prove Theorem \ref{def1.3.1}(2)$\Rightarrow$(1) in the section. We also revisit the problem of when $\tr_R(\rmK_R)$ is an Ulrich ideal in arbitrary dimension (Theorem \ref{a7.7}).
In Section \ref{section6}, we construct higher-dimensional $\GGL$ rings via certain specific determinantal ideals (Theorem \ref{a6.3}). We apply the result to the Rees algebras of parameter ideals over Gorenstein local rings (Corollary \ref{a6.5}).

Below is our notation used in this paper.

\begin{conv}
In what follows, all rings are commutative Noetherian rings. 
For a ring $R$, $\rmQ(R)$ denotes the total ring of fractions of $R$. 
We say that $I$ is a {\it fractional ideal} if $I$ is a finitely generated $R$-submodule of $\rmQ(R)$ containing a non-zerodivisor of $R$. For a fractional ideals $I$ and $J$, we denote by $I:J$ the fractional ideal $\{x\in \rmQ(R) \mid xJ\subseteq I\}$. It is known that an isomorphism $I:J\cong \Hom_R(J, I)$ is given by the correspondence $x \mapsto \hat{x}$, where $\hat{x}$ denotes the multiplication map of $x\in I:J$ (\cite[Lemma 2.1]{HK}). We say that an ideal $I$ is {\it regular} if $I$ contains a non-zerodivisor of $R$. For ideals $I$ and $J$ of $R$, we set $I:_RJ := (I:J) \cap R.$

We use further notations for local rings. $(R, \fkm)$ denotes a local ring with the unique maximal ideal $\fkm$. Let $M$ be a finitely generated $R$-module. We denote by $\ell_R (M)$ and $\mu_R (M)$ the length of $M$ and the number of minimal generators of $M$, respectively. If $M$ is a Cohen-Macaulay $R$-module of dimension $s$, then $\rmr_R (M)$ denotes the Cohen-Macaulay type $\ell_R(\Ext_R^s(R/\fkm, M))$ of $M$. If $M=R$, then we just write $\rmr(R)$.

Let $\fka$ be an $\fkm$-primary ideal of $R$. For a finitely generated $R$-module $M$ of dimension $s$, there exist integers $\{ \rme_{\fka}^i(M)\}_{0\le i\le s}$ such that 
\[
\ell_R(M/\fka^{n+1} M)=\rme_{\fka}^{0}(M){\cdot}\binom{n+s}{s}-\rme_{\fka}^{1}(M){\cdot}\binom{n+s-1}{s-1}+\dots+(-1)^{s}\rme_{\fka}^{s}(M)
\]
for all $n\gg 0$. $\rme_{\fka}^i(M)$ are called the {\it $i$th Hilbert coefficient} of $M$ with respect to $\fka$. In particular, $\rme (R)=\rme_{\fkm}^0(R)$ is called the {\it multiplicity} of $R$.
We denote by $\rmv(R)$ the {\it embedding dimension} of $R$. An ideal $\fkq$ is called a {\it reduction of $\fka$} if $\fkq\subseteq \fka$ and $\fka^{\ell+1}=\fkq \fka^\ell$ for some $\ell\ge 0$. We call a reduction {\it minimal} if it is generated by a system of parameters of $R$. It is known that $\rme_\fka^0(M)=\rme_\fkq^0(M)$ for a reduction $\fkq$ of $\fka$ (\cite[Lemma 4.6.5]{BH}). 

For a matrix $\mathbb{M}$ whose entries are in $R$ and a positive integer $t$, we denote by $I_t(\mathbb{M})$ the {\it determinantal ideal of $\mathbb{M}$}, that is, the ideal of $R$ generated by the $t \times t$ minors of the matrix $\mathbb{M}$.
\end{conv}


\section{Preliminaries}\label{section1}

\subsection{Ulrich modules with respect to an $\fkm$-primary ideal}
In this subsection, we summarize some necessary preliminaries on the properties of Ulrich $R$-modules. Suppose that $(R, \fkm)$ is a local ring and $M$ is a finitely generated $R$-module. We start to recall the definition of Ulrich modules.

\begin{defn}(\cite[Definition 1.2]{GOTWY})\label{a2.2}
Let $M$ be a finitely generated $R$-module and $\fka$ be an $\fkm$-primary ideal of $R$. We say that $M$ is an {\it Ulrich $R$-module with respect to $\fka$} if $M$ satisfies the following conditions:
\begin{enumerate}[{\rm (1)}]
\item $M$ is a nonzero Cohen-Macaulay $R$-module (i.e., $\depth_R M=\dim_R M$),
\item $\rme_{\fka}^0 (M)=\ell_R (M/\fka M)$, and
\item $M/\fka M $ is an $R/\fka$-free module. 
\end{enumerate}
\end{defn}

\begin{remark}\label{a2.25}
Note that the definition of Ulrich $R$-modules with respect to $\fka$ in the sense of \cite[Definition 1.2]{GOTWY} supposes that the module is maximal Cohen-Macaulay. Here, we only assume that the module is Cohen-Macaulay. This is a necessary generalization to define $\GGL$ rings, see Fact \ref{a3.1}. 
\end{remark}

\begin{remark}\label{a2.3}
We can replace the condition (3) of Definition \ref{a2.2} with the equality 
$$\ell_R (M/\fka M)=\mu_R (M) {\cdot} \ell_R (R/\fka)$$ 
since there is a surjection $(R/\fka)^{\oplus \mu_R (M)} \to M/\fka M$.
\end{remark}

We can also rephrase the condition (2) of Definition \ref{a2.2} as follows:

\begin{lem}\label{a2.4}
Let $\fka$ be an $\fkm$-primary ideal of $R$. Suppose that $R/\fkm$ is infinite and $M$ is a Cohen-Macaulay $R$-module of dimension $s$. Then,  the following conditions are equivalent:
\begin{enumerate}[{\rm (1)}]
\item $\rme_{\fka}^0 (M)=\ell_R (M/\fka M)$ and
\item $\fka M = (f_1, f_2, \dots , f_s)M $ for some elements $f_1, f_2, \dots , f_s \in \fka$.
\end{enumerate}
When this is the case, $f_1, f_2, \dots , f_s$ form a system of parameters of $M$ and part of a minimal system of generators of $\fka$. 
\end{lem}

\begin{proof}
Set $A=R/[(0):_R M]$ and $\overline{\fka}=\fka A$. For $a \in R$, $\ol{a}$ stands for the image of $a$ in $A$.

(1) $\Rightarrow$ (2): Since $A/\fkm A \cong R/\fkm$ is infinite, we can choose a parameter ideal $\fkq=(\overline{f_1}, \overline{f_2}, \dots , \overline{f_s})$  of $A$ as a reduction of $\overline{\fka}$, where  $f_1, f_2, \dots , f_s \in \fka$. Hence,
$$\ell_R (M/\fka M)=\rme_{\fka}^0 (M)=\rme_{\ol{\fka}}^0 (M)=\rme_{\fkq}^0 (M)=\ell_{A} (M/\fkq M)=\ell_R (M/(f_1, f_2, \dots , f_s) M).$$
It follows that $\fka M = (f_1, f_2, \dots , f_s)M $ because $(f_1, f_2, \dots , f_s) \subseteq \fka$.

(2) $\Rightarrow$ (1): Since $\ol{\fka} M=(\overline{f_1}, \overline{f_2}, \dots , \overline{f_s})M$ and $M$ is a faithful $A$-module, $\fkq=(\overline{f_1}, \overline{f_2}, \dots , \overline{f_s})$ is a reduction of $\overline{\fka}$ (\cite[Corollary 1.1.8]{SH}). 
Hence, 
$$\rme_{\fka}^0 (M)=\rme_{\overline{\fka}}^0 (M)=\rme_{\fkq}^0 (M)=\ell_{A} (M/\fkq M)=\ell_R (M/\fka M).$$

When this is the case, $f_1, f_2, \dots , f_s$ form a system of parameters of $M$ because $M/\fka M=M/(f_1, f_2, \dots , f_s) M$ is of finite length and $s=\dim_R M$. Assume that $f_1, f_2, \dots , f_s$ is not part of a minimal system of generators of $\fka$. Then,  we have $(f_1, f_2, \dots , f_s)+\fkm \fka=(f_1, f_2, \dots , f_{s-1})+\fkm \fka$ after renumbering $f_1, f_2, \dots , f_s$. It follows that $(\overline{f_1}, \overline{f_2}, \dots , \overline{f_{s-1}}) +\fkm \overline{\fka}$ is a reduction of $\overline{\fka}$, and thus $(\overline{f_1}, \overline{f_2}, \dots , \overline{f_{s-1}})$ is also a reduction of $\overline{\fka}$. This contradicts $s=\dim A$.
\end{proof}

To establish fundamental properties of Ulrich $R$-modules, we prepare a lemma. For an ideal $I$ of $R$, 
\[
\mbox{\rm gr}_{I}(R):=R[It]/IR[It]\cong \bigoplus_{n\ge0} I^n /I^{n+1}
\]
is called the {\it associated graded ring} of $I$, where $R[t]$ is the polynomial ring over $R$. For a finitely generated $R$-module $X$, 
\[
\mbox{\rm gr}_{I}(X)=\bigoplus_{n\ge0} I^n X/I^{n+1}X
\]
is called the {\it the associated graded $\mbox{\rm gr}_{I}(R)$-module} of $X$ with respect to $I$ (\cite[Introduction]{RV}).

\begin{lemma}\label{a2.1}
Let $\varphi : R \to S$ be a flat local homomorphism of local rings such that $S/\fkm S$ is a Cohen-Macaulay local ring of dimension  $\ell$. Let $g_1, g_2, \dots, g_{\ell} \in S$ be a system of parameters of $S/\fkm S$, and set $J=(g_1, g_2, \dots, g_{\ell})$. Let $X$ be a finitely generated $R$-module. 
Then,  $g_1t, g_2t, \dots, g_{\ell}t$ is a $\mbox{\rm gr}_{IS+J}(S\otimes_R X)$-regular sequence for any ideal $I$ of $R$. 
\end{lemma}

\begin{proof}
First, we show that $g_1, g_2, \dots, g_{\ell}$ is a $\mbox{\rm gr}_{IS}(S\otimes_R X)$-regular sequence.
Indeed, the exact sequence $0 \to S \xrightarrow{g_1} S \to S/g_1S \to 0$ of $S$-modules induces the exact sequence
$$
0 \to S\otimes_R I^n X \xrightarrow{g_1} S\otimes_R I^n X \to (S/g_1S)\otimes_R I^n X \to 0
$$
as $S$-modules for all $n \ge 0$ because $R\to S/g_1S$ is a flat local homomorphism. By letting $L=S\otimes_R X$, the above exact sequence is the same as the exact sequence 
\[
0 \to I^nL \xrightarrow{g_1} I^n L \to I^n (S/g_1S\otimes_R X) \to 0
\]
of $S$-modules for all $n \ge 0$. This induces the graded exact sequence
$$
0 \to \mbox{\rm gr}_{IS}(L) \xrightarrow{g_1} \mbox{\rm gr}_{IS}(L) \to \mbox{\rm gr}_{IS}((S/g_1S)\otimes_R X) \to 0
$$
of graded $\mbox{\rm gr}_{IS}(S)$-modules. Hence, $g_1$ is a $\mbox{\rm gr}_{IS}(L)$-regular element. By induction on $\ell$,  $g_1, g_2, \dots, g_{\ell}$ is a $\mbox{\rm gr}_{IS}(L)$-regular sequence.
Hence, 
\begin{align*}
JL\cap (IS+J)^{k+1} L &= JL\cap [I^{k+1} L+ J(IS+J)^{k} L] = JL \cap I^{k+1} L+ J(IS+J)^{k} L \\
&= J I^{k+1} L+ J(IS+J)^{k} L=J(IS+J)^k L
\end{align*}
for all $k \ge 0$. Therefore, $g_1t, g_2t, \dots, g_{\ell}t$ is a $\mbox{\rm gr}_{IS+J}(L)$-regular sequence by \cite{VV} or \cite[Theorem 1.1]{RV}.
\end{proof}

The following are the fundamental properties of Ulrich $R$-modules with respect to $\fka$. Note that these are a generalization of \cite[Proposition 2.2]{GTT}, but the proof is essentially the same.

\begin{prop}\label{a2.5}
Let $M$ be a nonzero finitely generated $R$-module and set $s=\dim_R M$. Let $\fka$ be an $\fkm$-primary ideal of $R$. Then, the following assertions hold true: 
\begin{enumerate}[{\rm (1)}]
\item Suppose that $s=0$. Then,  $M$ is an Ulrich $R$-module with respect to $\fka$ if and only if $M$ is $R/\fka$-free. 
\item Suppose that $s >0$ and $M$ is a Cohen-Macaulay $R$-module. Assume that $f \in \fka$ is a superficial element for $M$ with respect to $\fka$.  
Then,  $M$ is an Ulrich $R$-module with respect to $\fka$ if and only if $M/fM$ is an Ulrich $R/(f)$-module with respect to $\fka /(f)$. 
\item Let $\varphi : R \to S$ be a flat local homomorphism such that $S/\fkm S$ is a Cohen-Macaulay local ring of dimension  $\ell$. 
Let $g_1, g_2, \dots, g_{\ell} \in S$ be a system of parameters of $S/\fkm S$.
Then,  $M$ is an Ulrich $R$-module with respect to $\fka$ if and only if $S \otimes_R M$ is an Ulrich $S$-module with respect to $\fka S+(g_1, g_2, \dots, g_{\ell})$.  
\item Suppose that $f\in \fkm$ is $M$-regular and $M/fM$ is an Ulrich $R/(f)$-module with respect to $[\fka +(f)]/(f)$. Then,  $M$ is an Ulrich $R$-module with respect to $\fka +(f)$ and $f\not\in \fkm \fka$. 

\end{enumerate}
\end{prop}

\begin{proof}
(1): Since $s=0$, we have $\rme_{\fka}^0 (M)= \ell_R (M)$. Hence, $M$ is an Ulrich $R$-module with respect to $\fka$ if and only if $\fka M=0$ and $M/\fka M$ is $R/\fka$-free.

(2): Set $\overline{M}=M/fM$, $\overline{R}=R/fR$, and $\overline{\fka}=\fka \overline{R}$. Then,  we have the equalities $\rme_{\fka}^0 (M)=\rme_{\overline{\fka}}^0 (\overline{M})$ and $\ell_R (M/\fka M)=\ell_{\overline{R}} (\overline{M}/\overline{\fka} \overline{M})$ since $f\in \fka$ is a superficial element for $M$ with respect to $\fka$. By noting that $f$ is a non-zerodivisor of $M$ since $s>0$ (\cite[Lemma 1.2]{RV}), we have the equivalence.

(3): By Lemma \ref{a2.1} and (2), $S \otimes_R M$ is an Ulrich $S$-module with respect to $\fka S+(g_1, g_2, \dots, g_{\ell})$ if and only if $\ol{S} \otimes_R M$ is an Ulrich $\ol{S}$-module with respect to $\fka \ol{S}$, where $\ol{S}=S/(g_1, g_2, \dots, g_{\ell})$. Hence, by passing to $R\to S/(g_1, g_2, \dots, g_{\ell})$, we may assume that $\ell=0$.
Then, by noting that $\ell_S (S\otimes_R N)=\ell_S (S/\fkm S)\cdot \ell_R (N)$ for all $R$-modules $N$ of finite length and $\mu_S(S\otimes_R L)=\mu_R(L)$ for all finitely generated $R$-modules $L$, we obtain the assertion (see also Remark \ref{a2.3}).

(4): By passing to $R \to R[X]_{\fkm R[X]}$ if necessary, we may assume that $R/\fkm$ is infinite. To prove that $M$ is an Ulrich $R$-module with respect to $\fka + (f)$, we have only to show that $\rme_{\fka + (f)}^0 (M)=\ell_R (M/[\fka + (f)]M)$. By Lemma \ref{a2.4}, we have $[\fka+(f)]{\cdot}M/fM=(f_2, \dots , f_s){\cdot}M/fM$ for some elements $f_2, \dots, f_s \in [\fka+(f)]$; hence, $[\fka+(f)]  M=(f, f_2, \dots , f_s)M$. Therefore, $M$ is an Ulrich $R$-module with respect to $\fka +(f)$ and $f\not\in \fkm \fka$. 
\end{proof}


\subsection{Ulrich ideals and the trace ideal of the canonical module}\label{subsection22}
In this subsection, we summarize the properties of Ulrich ideals. We also note a result related to the trace ideal of the canonical module. 
Suppose that $(R, \fkm)$ is a Cohen-Macaulay local ring of dimension $d\ge0$ possessing the canonical module $\rmK_R$. 

\begin{defn}{\rm (\cite[Definition 2.1]{GOTWY})}\label{Ulrich ideal}
Let $I$ be an $\m$-primary ideal of $R$. Suppose that $I$ has a minimal reduction $Q$ of $I$. We say that $I$ is an {\it Ulrich ideal} of $R$ if the following conditions are satisfied:
\begin{enumerate}[{\rm (1)}]
\item $I\ne Q$, but $I^2=QI$.
\item $I/I^2$ is a free $R/I$-module.
\end{enumerate}
\end{defn}

The following are fundamental properties of Ulrich ideals.

\begin{fact} (\cite{GOTWY, GTT2})\label{a7.0.2}
Let $I$ be an Ulrich ideal of $R$ and set $n = \mu_R (I)$. The following assertions hold true.
\begin{enumerate}[$(1)$]
\item $\rmr(R) = (n-d){\cdot}\rmr(R/I)$.
\item  For $i \in \Bbb Z$,  $\Ext_R^i(R/I, R)\cong
\begin{cases}
(0) & (i < d),\\
(R/I)^{\oplus (n-d)} & (i = d),\\
(R/I)^{\oplus \{(n-d)^2-1\}{\cdot}(n-d)^{i - (d+1)} } & (i > d). 
\end{cases}$
\item Let $\cdots \to F_i \overset{\partial_i}{\to} F_{i-1} \to \cdots \to F_1 \overset{\partial_1}{\to}  F_0 = R \to R/I \to 0$
be a minimal free resolution of $R/I$. Let $I(\partial_i)$ denote the ideal of $R$ generated by the entries of the matrix representing $\partial_i$, and let $\beta_i = \rank_RF_i$ denote the $i$th Betti number of $R/I$. 
\begin{enumerate}[\rm(i)] 
\item  $I(\partial_i) = I$ for $i \ge 1$.
\item  For $i \ge 0$, $\beta_i = 
\begin{cases}
(n-d)^{i-d}{\cdot}(n-d+1)^d & \ \ (i \ge d),\\
\binom{d}{i}+ (n-d){\cdot}\beta_{i-1} & \ \ (1 \le i \le d),\\
1  & \ \  (i=0).
\end{cases}$
\end{enumerate}
\end{enumerate}

\end{fact}

Next, we note a result related to Ulrich ideals and the trace ideal of the canonical module (Proposition \ref{a7.2}). Before stating it, we recall the definition of trace ideals. 

\begin{defn}{\rm (\cite[Corollary 2.2]{Lin})} \label{a7.1}
For an $R$-module $M$, the {\it trace ideal} of $M$ is the ideal
\begin{align*} 
\tr_R(M)=\sum_{f\in \Hom_R(M, R)} \Im f .
\end{align*}
In other words, the trace ideal of $M$ is the image of the evaluation map 
\begin{align}\label{eq291}
t_M:\Hom_R(M, R)\otimes_R M \to R; \quad f\otimes x \mapsto f(x)
\end{align}
for $f\in \Hom_R(M, R)$ and $x\in M$.

An ideal $I$ of $R$ is called a {\it trace ideal} if $I$ is a trace ideal for some $R$-module $M$.
\end{defn}

\begin{fact}{\rm (\cite[Corollary 2.2]{GIK})} \label{trace}
For a fractional ideal $I$ of $R$, $\tr_R(I)=(R: I)I$. 
In addition, if $I$ is a regular ideal of $R$, the following are equivalent:
\begin{enumerate}[{\rm (1)}] 
\item $I$ is a trace ideal, 
\item $I=(R:I)I$, and 
\item $I:I=R:I$.
\end{enumerate} 
\end{fact}

Recall that $R$ is called {\it generically Gorenstein} if $\rmQ(R)$ is Gorenstein. 

\begin{fact} {\rm (\cite[Proposition 3.3.18]{BH})} \label{genegor}
The following are equivalent:
\begin{enumerate}[\rm(1)] 
\item $R$ is generically Gorenstein.
\item $\rmK_R$ can be embedded into $R$.
\item $R$ can be embedded into $\rmK_R$.
\end{enumerate}
\end{fact}

Thus, if $R$ is generically Gorenstein, $\tr_R(\rmK_R)=(R:\omega)\omega$ for each ideal $\omega$ of $R$ such that $\omega\cong \rmK_R$ since $\rmK_R$ is faithful. 
We note that $\tr_R (\rmK_R)$ defines the non-Gorenstein locus of $R$, i.e.,
\[
\{\fkp\in \Spec R \mid \text{$R_\fkp$ is not a Gorenstein ring}\}=\{\fkp\in \Spec R  \mid \tr_R(\rmK_R) \subseteq \fkp\};
\]
for example, see \cite[Lemma 2.1]{HHS} or \cite[p.199]{LW}.

\begin{prop}\label{a7.2}
Suppose that $R$ is a generically Gorenstein ring, but not a Gorenstein ring. If $I$ is an Ulrich ideal such that $\mu_R(I)>d+1$, then $\tr_R(\rmK_R) \subseteq I$.
\end{prop}

\begin{proof}
Since $R$ is generically Gorenstein, there exists an ideal $\omega \subsetneq R$ such that $\omega\cong \rmK_R$.
For any non-zerodivisor $f \in \omega$ of $R$, we obtain an exact sequence
$$0 \to R \to \omega \to \omega/(f) \to 0$$
of $R$-modules. By applying the functor $\Hom_R(R/I, *)$ to the above exact sequence, by noting Fact \ref{a7.0.2}(2), we obtain that
\begin{align}\label{ulricheq}
\Ext_R^i(R/I, \omega/(f)) \cong \Ext_R^{i+1}(R/I, R) \cong (R/I)^{\oplus u_i}
\end{align}
for all $i > d$, where $u_i=\{ (\mu_R(I)-d)^2 -1 \}(\mu_R(I) -d)^{i-d-1}>0$ by the hypothesis. By considering the annihilators of \eqref{ulricheq}, 
$$\sum_{\text{$f\in \omega$ is a non-zerodivisor of $R$}}(f):_R \omega \subseteq I.$$
We show the left hand side of the equation  above is $\tr_R(\rmK_R)$.
Indeed, by Davis's lemma, we can choose non-zerodivisors $f_1, f_2, \dots, f_r\in \omega$ of $R$ such that $\omega=(f_1, f_2, \dots, f_r)$, where $r=\rmr(R)$. 
Therefore, because $(f):_R \omega=f(R: \omega)$, we have
\begin{eqnarray*}
\sum_{\text{$f\in \omega$ is a non-zerodivisor of $R$}}(f):_R \omega
&=& \sum_{\text{$f\in \omega$ R-regular}} f(R: \omega) \\
&=& \sum_{i=1}^{r} f_i (R: \omega) = \omega(R:\omega)=\tr_R(\rmK_R).
\end{eqnarray*}
\end{proof}

\begin{rem}
We cannot remove the condition $\mu_R(I)>d+1$ in Proposition \ref{a7.2}. See Example \ref{2gene}. 
\end{rem}

On the other hand, it is known that if $R$ is $G$-regular in the sense of \cite{Tak} (that is, every totally reflexive module is free), $\mu_R(I)>d+1$ hold for all Ulrich ideals $I$ of $R$ (\cite[Theorem 2.8]{GTT2}). Thus, by combining \cite[Theorem 2.8]{GTT2} and Proposition \ref{a7.2}, we have the following.

\begin{cor}\label{a7.3}
Suppose that $R$ is generically Gorenstein and $G$-regular.
If $I$ is an Ulrich ideal, then $\tr_R(\rmK_R) \subseteq I$.
\end{cor}

Corollary \ref{a7.3} gives rise to the question of when $\tr_R (\rmK_R)$ is an Ulrich ideal. We explore the question in Subsection \ref{subsection4.3}.


\section{$\GGL$ rings}\label{section2}
In this section, we introduce the notion of $\GGL$ rings. Throughout this section, let $(R, \fkm)$ be a Cohen-Macaulay local ring of dimension $d \ge 0$ possessing the canonical module $\rmK_R$. Let $\fka$ be an $\fkm$-primary ideal of $R$. Before introducing $\GGL$ rings, we note several propositions on the exact sequence $0 \to R \to \rmK_R \to C \to 0$ that is used to define $\GGL$ rings.

\begin{fact}\label{a3.1}
Suppose that there exists an exact sequence
$0 \to R \xrightarrow{\varphi} \rmK_R \to C \to 0.$
Then the following assertions hold true:
\begin{enumerate}[{\rm (1)}]
\item {\rm (\cite[Lemma 3.1 (3)]{GTT})} If $d=0$, then $C=(0)$. Hence, $R$ is a Gorenstein ring. 
\item {\rm (\cite[Lemma 3.1 (2)]{GTT})} $C$ is a Cohen-Macaulay $R$-module of dimension $d-1$ if $C \neq (0)$.
\end{enumerate}
\end{fact}

\begin{lem}\label{lem32}
Suppose that there exists an exact sequence
$0 \to R \xrightarrow{\varphi} \rmK_R \to C \to 0$, where $C$ is an Ulrich $R$-module with respect to $\fka$. If $\varphi(1)\in \fka \rmK_R$, then $R$ is a Gorenstein ring and $\fka$ is a parameter ideal of $R$.
\end{lem}

\begin{proof}
From Proposition \ref{a2.5}(3), we may assume that the residue field $R/\fkm$ is infinite. We know that $d>0$ by Fact \ref{a3.1} (note that the zero module is not an Ulrich module). 
Suppose that $d=1$ and choose a canonical ideal $\omega \subsetneq R$. We may assume that the given exact sequence forms
$$0 \to R \xrightarrow{\varphi} \omega \to C \to 0.$$
Set $a=\varphi(1) \in \omega$. Then,  $C \cong \omega/(a)$. Since $C$ is an Ulrich $R$-module with respect to $\fka$, we have $\fka \omega \subseteq (a)$ from Fact \ref{a3.1} and Proposition \ref{a2.5}(1); hence, $\fka \omega=(a)$ by the assumption. It follows that $\fka$ and $\omega$ are cyclic (see, for example, \cite[Lemma 3.9]{Ku}). Therefore, $R$ is a Gorenstein ring and $\fka$ is a parameter ideal of $R$.
Suppose that $d \ge 2$ and choose $f \in \fka$ such that $f$ is $R$-regular and superficial for $C$ with respect to $\fka$. We then obtain
$$0 \to R/(f) \xrightarrow{\ol{\varphi}} \rmK_{R/(f)} \to C/fC \to 0,$$
where $C/fC$ is an Ulrich $R/(f)$-module with respect to $\fka/(f)$ by Proposition \ref{a2.5}(2). Therefore, we have that $R/(f)$ is a Gorenstein ring and $\mu_{R/(f)} (\fka/(f))=d-1$ by the induction hypothesis. It follows that $R$ is a Gorenstein ring and $\fka$ is a parameter ideal of $R$.
\end{proof}

We are now in a position to define $\GGL$ rings.

\begin{defn}\label{a3.2}
Let $\fka$ be an $\fkm$-primary ideal of $R$. We say that $R$ is a {\it generalized Gorenstein local ring with respect to $\fka$} ($\GGL$ ring with respect to $\fka$) if there exists an exact sequence
$$0 \to R \xrightarrow{} \rmK_R \to C \to 0$$
of $R$-modules such that 
\begin{enumerate}
\item[$\mathrm{(i)}$] $C$ is an Ulrich $R$-module with respect to $\fka$ and
\item[$\mathrm{(ii)}$] $\rmK_R/\fka \rmK_R$ is $R/\fka$-free.
\end{enumerate}

We say that $R$ is a {\it $\GGL$ ring} if $R$ is a $\GGL$ ring with respect to some $\fkm$-primary ideal of $R$.
\end{defn}

\begin{rem}\label{rem3.4}
By definition, $R$ is a non-Gorenstein $\AGL$ ring if and only if $R$ is a non-regular $\GGL$ ring with respect to $\fkm$ (see Definition \ref{zzz2.1}). Furthermore, if $R$ is a Gorenstein ring of dimension $>0$, then we regard $R$ as a $\GGL$ ring by  a canonical exact sequence 
$$0 \to R \xrightarrow{a_1} R \to R/(a_1) \to 0,$$
where $(a_1, a_2, \dots, a_d)$ is a parameter ideal of $R$ since $R/(a_1)$ is an Ulrich $R$-module with respect to $(a_1, a_2, \dots, a_d)$. On the other hand, if $R$ is a $\GGL$ ring with respect to some parameter ideal of $R$, then $R$ is Gorenstein (Corollary \ref{xxxc72}).
\end{rem}

\begin{rem}
For a $\GGL$ ring $R$, defining $\fkm$-primary ideals $\fka$ of $R$ are not unique in general. For instance, if $R$ is Gorenstein, Remark \ref{rem3.4} shows that $R$ is a $\GGL$ ring with respect to all parameter ideals of $R$. Examples of non-Gorenstein $\GGL$ rings whose defining $\fkm$-primary ideals of $R$ are not unique are also obtained by using Proposition \ref{a3.6}. On the other hand, we also prove that the defining $\fkm$-primary ideal $\fka$ of $R$ is to be the trace ideal of the canonical module if $\dim R=1$ (hence, the defining $\fkm$-primary ideal $\fka$ of $R$ is unique), see Corollary \ref{a7.1.1}.
\end{rem}

The notion of $\GGL$ rings is rephrased as follows.

\begin{prop}\label{a3.2.5}
Assume that $d>0$. Then, $R$ is a $\GGL$ ring if and only if one of the following conditions holds.
\begin{enumerate}[{\rm (1)}]
\item $R$ is a Gorenstein ring.
\item $R$ is not a Gorenstein ring, but there exists an exact sequence
$$0 \to R \xrightarrow{\varphi} \rmK_R \to C \to 0$$
of $R$-modules and an $\fkm$-primary ideal $\fka$ of $R$ such that 
\begin{enumerate}
\item[$\mathrm{(i)}$] $C$ is an Ulrich $R$-module with respect to $\fka$ and
\item[$\mathrm{(ii)}$] the induced homomorphism $R/\fka \otimes_R \varphi : R/\fka \to \rmK_R/\fka \rmK_R$ is injective.
\end{enumerate}
\end{enumerate}

\end{prop}

\begin{proof}
(if part): If $R$ is a Gorenstein ring, then $R$ is a $\GGL$ ring by Remark \ref{rem3.4}.
Suppose that the case (2) occurs. Then, we have the exact sequence
\[
0 \to R/\fka \xrightarrow{R/\fka \otimes_R \varphi} \rmK_R/\fka \rmK_R \to C/\fka C \to 0
\]
of $R/\fka$-modules and $C/\fka C$ is $R/\fka$-free. It follows that the above exact sequence splits; hence, $\rmK_R/\fka \rmK_R$ is $R/\fka$-free.

(only if part): Let 
$$0 \to R \xrightarrow{\varphi} \rmK_R \to C \to 0$$
be a defining exact sequence of $\GGL$ rings. If $\varphi(1)\not\in \fkm \rmK_R$, then $\mu_R (C)=\mu_R(\rmK_R)-1=\rmr (R)-1$.
Hence, we obtain the exact sequence
\[
R/\fka \xrightarrow{R/\fka \otimes_R \varphi} \rmK_R/\fka \rmK_R \to (R/\fka)^{\oplus (\rmr(R)-1)} \to 0.
\]
Since $\rmK_R/\fka \rmK_R$ is an $R/\fka$-free module of rank $\rmr(R)$, by considering the length of the kernel of $R/\fka \otimes_R \varphi$, $R/\fka \otimes_R \varphi$ is injective.
Suppose that $\varphi(1)\in \fkm \rmK_R$. Then, we have the exact sequence
\[
R/\fka \xrightarrow{R/\fka \otimes_R \varphi} \rmK_R/\fka \rmK_R \to (R/\fka)^{\oplus \rmr(R)} \to 0,
\]
and thus $R/\fka \otimes_R \varphi$ is the zero map. It follows that we have $\varphi(1)\in \fka \rmK_R$, and  $R$ is a Gorenstein ring and $\fka$ is a parameter ideal of $R$ by Lemma \ref{lem32}.
\end{proof}

Let us note some examples of $\GGL$ rings that we will examine later.

\begin{ex}\label{a3.4}
Let $k$ be a field. Let $k[[t]]$, $k[[X_1, X_2, X_3, Y_1, Y_2, Y_3]]$, and $k[[X, Y, Z, W]]$ denote the formal power series ring over $k$. 
\begin{enumerate}
\item Let $R_1=k[[t^{5}, t^{6}, t^{8}]]$. $R_1$ is a $\GGL$ ring of dimension one, but not an $\AGL$ ring.
\item Set $S=k[[X_1, X_2, X_3, Y_1, Y_2, Y_3]]$. $R_2=S/I$ is a $\GGL$ ring of dimension one, where
$$I= (Y_1, Y_2, Y_3)^2 + I_2\left(
\begin{smallmatrix}
X_1^2 & X_2 & X_3 \\
Y_1 & Y_2 & Y_3 
\end{smallmatrix}
\right)+ 
I_2\left(
\begin{smallmatrix}
X_1^2 & X_2 & X_3 \\
X_2 & X_3 & X_1^3
\end{smallmatrix}
\right)+
I_2\left(
\begin{smallmatrix}
X_1^2 & X_2 & X_3 \\
Y_2 & Y_3 & X_1 Y_1
\end{smallmatrix}
\right).$$
\item Let $R_3=k[[X, Y, Z, W]]/I_2\left(
\begin{smallmatrix}
X^2 & Y^2 & Z^2 \\
Y^2 & Z^2 & X^2 
\end{smallmatrix}
\right)$. $R_3$ is a $\GGL$ ring of dimension two, but not an $\AGL$ ring.
\item Let $(R, \fkm)$ be a Cohen-Macaulay local ring of dimension $>0$ possessing the canonical module. If $R/\fkm$ is infinite and $\rme(R) \le 3$, then $R$ is a $\GGL$ ring.
\item If $R$ is a $\GGL$ ring, then $R[X]/(X^n)$ is also a $\GGL$ ring for all $n>1$, where $R[X]$ denotes the polynomial ring over $R$.
\end{enumerate}
\end{ex}

\begin{proof}
(1): See Corollary \ref{a4.34}.

(2): Let $A=k[[t^3, t^7, t^8]]$ and set $I=(t^6, t^7, t^8)$. $A$ is a $\GGL$ ring with respect to $I$ by Corollary \ref{a4.10}. It is straightforward to check that $R_2 \cong A \ltimes I$ as rings, where $A \ltimes I$ denotes the idealization of $I$ in the sense of \cite{AW}. On the other hand, $A \ltimes I$ is a $\GGL$ ring by Corollary \ref{a4.16}. 

(3): By Corollary \ref{a6.4}, $R_3$ is a $\GGL$ ring. If $R_3$ is an $\AGL$ ring, then $k[[X, Y, Z]]/I_2\left(
\begin{smallmatrix}
X^2 & Y^2 & Z^2 \\
Y^2 & Z^2 & X^2 
\end{smallmatrix}
\right)$ also is by Theorem \ref{a3.5}. This contradicts Corollary \ref{a7.1.1} and Proposition \ref{a7.6}.

(4): See Proposition \ref{a4.12}.

(5): $R \to R[X]/(X^n)$ is a flat local ring homomorphism such that the fiber is isomorphic to $(R/\fkm)[X]/(X^n)$. Hence, $R[X]/(X^n)$ is a $\GGL$ ring by Proposition \ref{a3.6}.
\end{proof}

\begin{rem}
We note that in the assertions (4) and (5) of Example \ref{a3.4}, we cannot replace a $\GGL$ ring with an $\AGL$ ring. Indeed, $k[[t^3, t^7, t^8]]$ is a $\GGL$ ring but not an $\AGL$ ring by \cite[Corollary 4.2]{GMP}. If $(R, \fkm)$ is a non-Gorenstein $\AGL$ ring of dimension one such that $R/\fkm$ is infinite, then $R[X]/(X^n)$ where $n>1$ is not an $\AGL$ ring by \cite[Proposition 3.12]{GTT}. 
\end{rem}

The following theorem and proposition assert that the $\GGL$ property is inherited by the canonical change of local rings.

\begin{thm}\label{a3.5}
Let $\fka$ be an $\fkm$-primary ideal of $R$. 
The following assertions hold true.
\begin{enumerate}[{\rm (1)}]
\item Suppose that $R$ is a $\GGL$ ring with respect to $\fka$ and $d\ge 2$. Choose a defining exact sequence
$$
0 \to R \xrightarrow{\varphi} \rmK_R \to C \to 0 
$$
of $R$-modules. If $f\in \fka$ is a superficial element for $C$ with respect to $\fka$ and a non-zerodivisor of $R$, then $R/(f)$ is a $\GGL$ ring with respect to $\fka/(f)$.
\item Let $f\in \fkm$ be a non-zerodivisor of $R$ and suppose that $R/(f)$ is a $\GGL$ ring with respect to $[\fka+(f)]/(f)$. Then,  $R$ is a $\GGL$ ring with respect to $\fka+(f)$ and $f\not\in \fkm \fka$.

\end{enumerate}

\end{thm}

\begin{proof}
(1): Since $f$ is a non-zerodivisor of $R$ and $C$, we have the exact sequence
$$
0 \to \overline{R} \xrightarrow{\ol{\varphi}} \rmK_{\overline{R}} \to C/fC \to 0 
$$
as $\overline{R}$-modules, where $\overline{R}=R/(f)$. Since $C/fC$ is an Ulrich $\overline{R}$-module by Proposition \ref{a2.5}, $\overline{R}$ is a $\GGL$ ring.

(2): Set $\overline{R}=R/(f)$ and $\overline{\fka}=\fka \overline{R}.$ We choose a defining exact sequence $0 \to \overline{R} \xrightarrow{\psi} \rmK_{\overline{R}} \to D \to 0 $ of $\overline{R}$-modules. Choose $x \in \rmK_{R}$ such that $\overline{x}=\psi(1)$, and consider the exact sequence
$$
R \xrightarrow{\varphi} \rmK_R \to C \to 0 
$$
of $R$-modules, where $\varphi(1)=x$. By noting that $C/fC\cong D$ and $\dim_R D=d-2$ by Fact \ref{a3.1}, we obtain that $\dim_R C<d$. Hence, $\varphi_\fkp$ is injective for all $\fkp \in \Ass R$. Since $\ell_{R_\fkp} (R_\fkp)= \ell_{R_\fkp} (\rmK_{R_\fkp})$ by the Matlis dual (see, for example, \cite[Proposition 3.2.12 (b)]{BH}), we obtain that $\varphi_\fkp$ is bijective. This follows that $\varphi$ is injective since $\Ass_R (\Ker \varphi) \subseteq \Ass R$. 
Therefore, $C$ is a Cohen-Macaulay $R$-module of dimension $d-1$ by Fact \ref{a3.1}. It  follows that $C$ is an Ulrich $R$-module with respect to $\fka$ and $f\not\in \fkm \fka$ by Proposition \ref{a2.5}.
\end{proof}

\begin{prop}\label{a3.6}
Let $R$ and $S$ be local rings of dimension $>0$. Let $\psi : R \to S$ be a flat local homomorphism. Suppose that $S/\fkm S$ is a Cohen-Macaulay local ring of dimension  $\ell$. Let $J\subseteq S$ be a parameter ideal in $S/\fkm S$. Consider the following two conditions:
\begin{enumerate}[{\rm (1)}]
\item $R$ is a $\GGL$ ring with respect to $\fka$, and $S/\fkm S$ is a Gorenstein ring.
\item $S$ is a $\GGL$ ring with respect to $\fka S+ J$.
\end{enumerate}
Then,  $\rm(1)$ $\Rightarrow$ $\rm(2)$ holds true. 
\end{prop}

\begin{proof}
Choose a defining exact sequence
$$
0 \to R \xrightarrow{\varphi} \rmK_R \to C \to 0 
$$
of $R$-modules. Then, since $S\otimes_R \rmK_R\cong \rmK_S$ by \cite[Thorem 3.3.14]{BH}, we obtain the exact sequence
$$
0 \to S \xrightarrow{S\otimes_R \varphi} \rmK_S \to S\otimes_R C \to 0 
$$
as $S$-modules. $S\otimes_R C$ is an Ulrich $S$-module with respect to $\fka S+J$ by Proposition \ref{a2.5}. Moreover, we have $\rmK_S/\fka \rmK_S \cong S\otimes_R (\rmK_R/\fka \rmK_R)$ is $S/\fka S$-free. Thus, $\rmK_S/(\fka S+J) \rmK_S$ is $S/(\fka S + J)$-free.
\end{proof}

Furthermore, Proposition \ref{a3.6} $\rm(2)$ $\Rightarrow$ $\rm(1)$ also holds true if $R/\fkm$ is infinite as shown in the introduction (Theorem \ref{def1.3.1}). 
Before proving the implication (2) $\Rightarrow$ (1), we need to further discuss the case of dimension one.


\section{One-dimensional $\GGL$ rings}\label{section3}
In this section, we focus on the one-dimensional case that is the heart of this paper. Throughout this section, let $(R, \fkm)$ be a Cohen-Macaulay local ring of dimension one possessing the canonical module $\rmK_R$. Let $\fka$ be an $\fkm$-primary ideal of $R$. We denote by $\ol{R}$ the {\it integral closure} of $R$. 
We say that a fractional ideal $I$ is a {\it fractional canonical ideal} of $R$ if $I\cong \rmK_R$. If a regular ideal $I$ satisfies $I\cong \rmK_R$, we say that $I$ is a {\it canonical ideal} of $R$.

\subsection{Defining $\fkm$-primary ideals of one-dimensional $\GGL$ rings}
First, we see that the defining $\fkm$-primary ideal $\fka$ of a $\GGL$ ring with respect to $\fka$ is unique in dimension one. 

\begin{proposition}\label{a4.1}
Suppose that $R$ is a non-Gorenstein $\GGL$ ring with respect to $\fka$.
Then, there exists a fractional canonical ideal $K$ satisfying the following conditions:
\begin{enumerate}[{\rm (1)}]
\item $R \subseteq K \subseteq \overline{R}$.
\item $0 \to R \to K \to K/R \to 0$ is a defining exact sequence, that is, $K/\fka K\cong (R/\fka)^{\oplus \rmr(R)}$ and $K/R\cong (R/\fka)^{\oplus(\rmr(R)-1)}$.
\item $\fka=R:K=(R:K)K$.
\end{enumerate}
\end{proposition}

\begin{proof} 
We choose a defining exact sequence 
\begin{align}\label{ttt4.1}
0 \to R \xrightarrow{\varphi} \rmK_R \to C \to 0
\end{align}
of $R$-modules. 
Since there exists a canonical ideal $\omega\subsetneq R$, we can replace $\rmK_R$ with $\omega$ in (\ref{ttt4.1}). Set $a=\varphi(1)$ and $K = \frac{\omega}{a} = \left\{\frac{x}{a}\in \rmQ(R) \mid x \in \omega\right\}$.
Then,  (\ref{ttt4.1}) can be replaced by $0 \to R \xrightarrow{\iota} K \to K/R \to 0$, where $\iota$ is the inclusion. Hence, $K/\fka K$ and $K/R$ are $R/\fka$-free modules. In particular, by considering the annihilator of $K/R$, we obtain that $\fka=R:K$.
On the other hand, since $R$ is not a Gorenstein ring, $R/\fka\otimes_R \iota$ is injective by Proposition \ref{a3.2.5}. That is, $\fka K\cap R=\fka$. Therefore, we obtain that 
\[
\fka = \fka K\cap R = (R:K)K\cap R = (R:K)K = \fka K.
\]
By noting that $K=\frac{\omega}{a}$, $\fka = \fka K$ is equivalent to $a\fka =\omega \fka$.  Hence, $(a)$ is a reduction of $\omega$, that is, $\omega^{n+1}=a\omega^{n}$ for $n\gg0$ (\cite[Corollary 1.1.8]{SH}). 
This is equivalent to saying that $K^{n+1}=K^n$ for $n\gg0$. Therefore, by noting that $R\subseteq K$, $K^n$ is a subring of $\rmQ(R)$ and finitely generated as an $R$-module. It follows that $R\subseteq K \subseteq \ol{R}$. 

It is remained to check the rank of  $K/\fka K$ and $K/R$ as free $R/\fka$-modules. Since $K$ is a fractional canonical ideal, $K/\fka K\cong (R/\fka)^{\oplus \rmr(R)}$. For $K/R$, since $R/\fka\otimes_R \iota$ is injective, we obtain that $K/R\cong (R/\fka)^{\oplus(\rmr(R)-1)}$.
\end{proof}

As a corollary of Proposition \ref{a4.1}(3), by noting Fact \ref{trace}, we obtain the following.

\begin{cor}\label{a7.1.1}
If $R$ is a non-Gorenstein $\GGL$ ring with respect to $\fka$ of dimension one, then $\fka=\tr_R(\rmK_R)$.
\end{cor}

Corollary \ref{a7.1.1} leads a relation among the notions of $\AGL$ rings, $\GGL$ rings, and $\NGL$ rings. Recall that a Cohen-Macaulay local ring $(A, \fkn)$ having $\rmK_A$ (not necessarily of dimension one) is called a {\it nearly Gorenstein local ring} ($\NGL$ ring) if $\tr_A(\rmK_A)\supseteq \fkn$ (\cite[Definition 2.2]{HHS}). 

\begin{cor}
For a one-dimensional Cohen-Macaulay local ring $R$, $R$ is a $\GGL$ ring and a $\NGL$ ring if and only if $R$ is an $\AGL$ ring. 
\end{cor}

\begin{proof}
To prove each implication, we may assume that $R$ is not a Gorenstein ring. 
If $R$ is a $\GGL$ ring with respect to $\fka$ and a $\NGL$ ring, $\fka=\tr_R(\rmK_R)=\fkm$ by Corollary \ref{a7.1.1} and \cite[Lemma 2.1]{HHS}. Hence, $R$ is an $\AGL$ ring (see Remark \ref{rem3.4}). Suppose that $R$ is a non-Gorenstein $\AGL$ ring, equivalently, $R$ is a $\GGL$ ring with respect to $\fkm$. By Corollary \ref{a7.1.1}, we have $\fkm=\tr_R(\rmK_R)$. Hence $R$ is a $\NGL$ ring.
\end{proof}

The fractional canonical ideal $K$ in Proposition \ref{a4.1} is useful to explore the $\GGL$ property. Thus, throughout this section, we maintain the following setting unless otherwise noted. 

\begin{setup}\label{setting1}
Let $(R, \fkm)$ be a one-dimensional Cohen-Macaulay local ring possessing the canonical module $\rmK_R$. Suppose that there exists a fractional canonical ideal $K$ such that $R\subseteq K\subseteq \overline{R}$. 
Set 
\begin{center}
$S:=R[K]$ and $\fkc:=R:S$. 
\end{center}
Note that $S=K^n$ for all $n\gg 0$; hence, we have $\fkc \subseteq R \subseteq K \subseteq S \subseteq \ol{R}$. 
\end{setup}

Let us note some remarks on Setup \ref{setting1}.

\begin{rem}\label{a4.3}
\begin{enumerate}[{\rm (1)}]
\item By Proposition \ref{a4.1}(1), the assumption of the existence of $K$ in Setup \ref{setting1} is necessary to study the $\GGL$ property in dimension one. On the other hand, $R$ satisfies Setup \ref{setting1} if $R$ is a generically Gorenstein ring and $R/\fkm$ is infinite. Indeed, by choosing a canonical ideal $\omega$ and its reduction $(a)\subseteq \omega$, $K:=\frac{\omega}{a}$ satisfies the assertion.
\item Assume Setup \ref{setting1}. Then,  the following assertions hold true:
\begin{enumerate}[{\rm (i)}]
\item (\cite[Theorem 2.5]{CGKM} and \cite[Proposition 2.4(b)]{Ku}) $S$ is independent of the choice of a fractional canonical ideal $K$ such that $R\subseteq K \subseteq \ol{R}$.
\item (\cite[Bemerkung 2.5 c)]{HK} and \cite[Lemma 2.1]{HK2}) Let $I$ and $J$ be fractional ideals. If $J \subseteq I$, then $\ell_R(I/J)=\ell_R((K:J)/(K:I))$.
\end{enumerate}
\end{enumerate}
\end{rem}

We often use the following equations in the discussion.

\begin{lem}\label{useful}
The following hold true.
\begin{enumerate}[\rm(1)] 
\item $\fkc=K:S$.
\item $S=K:\fkc=R:\fkc=\fkc:\fkc$.
\end{enumerate}
\end{lem}

\begin{proof}
(1): This follows from $\fkc=R:S=(K:K):S=K:KS=K:S$. 

(2): The first two equalities follow from  $R:\fkc=(K:K):\fkc=K:\fkc K=K:\fkc=K:(K:S) =S$ by (1) and $\fkc \subseteq \fkc K \subseteq \fkc S =\fkc$. The last equality follows from $\fkc:\fkc=(R:S):(R:S)=R:(R:S)S=R:\fkc=S$.
\end{proof}

Under Setup \ref{setting1}, it is known that the Gorenstein property of $R$ is characterized as follows.

\begin{fact}\label{a4.4}
\begin{enumerate}[\rm(1)] 
\item {\rm (\cite[Proposition 3.6]{GMP})} The following inequalities hold true:
\[
0\le \rmr(R)-1 \le \ell_R(K/R) \le \rme_{\omega}^{1} (R) = \ell_R(R/\fkc) + \ell_R(K/R) = \ell_R(S/R). 
\]

\item {\rm (\cite[Theorem 3.7]{GMP})} The following conditions are equivalent:
\begin{align*}
&\text{{\rm (i)} $R$ is a Gorenstein ring.} & &\text{{\rm (ii)} $R=K$.} & &\text{{\rm (iii)} $R=S$.}\\
&\text{{\rm (iv)} $K=S$.} && \text{{\rm (v)} $K=K^2$.} & &\text{{\rm (vi)} $\fkc=R$.}\\
&\text{{\rm (vii)} $\rme_{\omega}^{1} (R) = \ell_R (R/\fkc)$.} & &\text{{\rm (viii)} $\rme_{\omega}^{1} (R) = 0$.}&&
\end{align*}
\end{enumerate}
\end{fact}

Therefore, if $R$ is not a Gorenstein ring, then  $\fkc \subsetneq R \subsetneq K \subsetneq S$. It seems that the difference between $R$ and $S$ (or other inclusions) describes the distance of the Gorensteinness of $R$. From this perspective, we consider such characterizations of $\GGL$ rings in dimension one (Theorem \ref{a4.7}).

\begin{lemma}\label{a4.5}
Suppose that $R$ is not a Gorenstein ring.
Then, the following hold true: 
\begin{enumerate}[\rm(1)] 
\item $K/\fkc=K/\fkc K$ is independent of the choice of $K$ up to isomorphism.
\item $\ell_R(K/R)$ is independent of the choice of $K$. In particular, the $R/\fkc$-freeness of $K/R$ is independent of the choice of $K$. 
\end{enumerate}
\end{lemma}

\begin{proof}
(1): Note that $\fkc=\fkc K$ since $\fkc=\fkc S$ and $R\subseteq K\subseteq S$.
Let $L$ be another fractional canonical ideal such that $R\subseteq L \subseteq \ol{R}$. 
There is an element $\alpha \in \rmQ(R)$ such that $L=\alpha K$ because $L\cong K$. $\alpha$ is a unit of $S$ since $S=L^n=\alpha^n K^n=\alpha^n S$ for $n\gg 0$ by Remark \ref{a4.3}(2)(i). Therefore, 
$$L/\fkc=\alpha K/\fkc\cong K/\alpha^{-1} \fkc=K/\fkc$$
because $\fkc=\fkc S$. 

(2): By (1), it follows that 
\[
\ell_R(L/R) = \ell_R(L/\fkc) -\ell_R(R/\fkc) = \ell_R(K/\fkc) -\ell_R(R/\fkc) = \ell_R(K/R). 
\]
On the other hand, we have $\mu_R(K/R) =\rmr(R)-1$. Otherwise, we have $R\subseteq \fkm K \cap R \subseteq \fkm S \cap R\subseteq \fkm$, which is a contradiction.
Therefore, since the annihilator of $K/R$ is $R:K$ and $R:K \supseteq \fkc$, we have a surjection $(R/\fkc)^{\oplus (\rmr(R)-1)} \to K/R \to 0$. 
Since $R/\fkc$ is Artinian, $K/R$ is $R/\fkc$-free if and only if $\ell_R(K/R) = \ell_R(R/\fkc)(\rmr(R)-1)$, which is independent of the choice of $K$. 
\end{proof}

The following lemma shows the equivalent conditions of Proposition \ref{a4.1}(3). 

\begin{lemma}\label{a4.6}
The following conditions are equivalent.
\begin{align*}&\text{{\rm (1)} $\tr_R(\rmK_R)=R:K$.}  & &\text{{\rm (2)} $R:K=\fkc$.} & &\text{{\rm (3)} $K^2=K^3$.}
\end{align*}
\end{lemma}

\begin{proof}
Set $\fka=R:K$.

(2) $\Leftrightarrow$ (3): Since $\fka=(K:K):K=K:K^2$ and $\fkc=K:S$ by Lemma \ref{useful}, we have $\ell_R(R/\fka)=\ell_R(K^2/K)$ and $\ell_R(R/\fkc)=\ell_R(S/K)$ by Remark \ref{a4.3}. Hence, $\fka=\fkc$ if and only if $K^2=S$ if and only if $K^2=K^3$.

(2) $\Rightarrow$ (1): Since $\fkc S=\fkc$, we have $\tr_R(\rmK_R)=\fka K=\fkc K=\fkc=\fka$.

(1) $\Rightarrow$ (2): Since $\fka K=\fka$, by induction, we have $\fka K^n=\fka$ for all $n>0$; hence, $\fka S= \fka\subseteq R$. It follows that $\fka\subseteq \fkc$. The reverse inclusion is clear.
\end{proof}

\begin{lemma}\label{tttp4.9} 
$S/K$ is the canonical module of $R/\fkc$.
\end{lemma}

\begin{proof}
By applying the $K$-dual to the exact sequence $0 \to \fkc \to R \to R/\fkc \to 0$, we obtain the exact sequence $0 \to K \to K:\fkc \to \Ext_R^1(R/\fkc, K) \to 0$. We obtain that $\Ext_R^1(R/\fkc, K) \cong (K:\fkc)/K=S/K$ by Lemma \ref{useful}. This shows that $S/K$ is the canonical module of $R/\fkc$ by \cite[Theorem 3.3.7(b)]{BH}.
\end{proof}

\begin{thm}\label{a4.7}
Suppose that $R$ is not a  Gorenstein ring.
Then,  the following conditions are equivalent:
\begin{align*}
&\text{{\rm (1)} $R$ is a $\GGL$ ring.}  & &\text{{\rm (2)} $K/R$ is an $R/\fkc$-free module.} \\
&\text{{\rm (3)} $K/R$ is an $R/\tr_R(\rmK_R)$-free module.} & &\text{{\rm (4)} $S/R$ is an $R/\fkc$-free module.} \\
&\text{{\rm (5)} $S/R$ is an $R/\tr_R(\rmK_R)$-free module.} & &\text{{\rm (6)} $\rme_{\omega}^{1} (R)=\ell_R(R/\fkc){\cdot}\rmr(R)$.} 
\end{align*}

When this is the case, we have the following: 
\begin{enumerate}[{\rm (i)}] 
\item $\tr_R(\rmK_R)=R:K=\fkc$.
\item $R/\fkc=R/\tr_R(\rmK_R)$ is a Gorenstein ring.
\end{enumerate}
\end{thm}

\begin{proof}
(1) $\Rightarrow$ (2) follows from Proposition \ref{a4.1} and Lemma \ref{a4.6}. 

(2) $\Rightarrow$ (3): Since the annihilator of $K/R$ is $R:K$, the condition (2) implies that $R:K=\fkc$. Hence, $\fkc=\tr_R(\rmK_R)$ by Lemma \ref{a4.6}. 

The proofs of (3) $\Rightarrow$ (2) and (5) $\Rightarrow$ (4) are similar to the proof of (2) $\Rightarrow$ (3). 

(2) $\Rightarrow$ (1): Note that $K/\fkc=K/\fkc K$ is also an $R/\fkc$-free module since the exact sequence $0 \to R/\fkc \to K/\fkc \to K/R \to 0$ splits. Hence, the assertion follows by observing that $0 \to R \to K \to K/R \to 0$ is a defining exact sequence for $R$ to be a $\GGL$ ring with respect to $\fkc$. 

Hence, (1), (2), and (3) are equivalent. When this is the case, the assertion (i) holds by Proposition \ref{a4.1} and Lemma \ref{a4.6}. 

(1) $\Rightarrow$ (ii): Let $r=\rmr(R)$ denote the Cohen-Macaulay type of $R$. Then, we have the exact sequence $0 \to R \to K \to (R/\fkc)^{\oplus (r-1)} \to 0$ by Proposition \ref{a4.1}. By applying the $K$-dual to the exact sequence, we have 
\[
0 \to R \to K \to \Ext_R^1(R/\fkc, K)^{\oplus (r-1)} \to 0.
\]
Thus, we have $(R/\fkc)^{\oplus (r-1)} \cong K/R\cong \Ext_R^1(R/\fkc, K)^{\oplus (r-1)} \cong \rmK_{R/\fkc}^{\oplus (r-1)}$; hence, $R/\fkc$ is a Gorenstein ring.

(2) $\Rightarrow$ (4): Consider the exact sequence $0 \to K/R \to S/R \to S/K \to 0.$
Note that $S/K\cong R/\fkc$ by (ii) and Lemma \ref{tttp4.9}. Since $K/R$ and $S/K$ are $R/\fkc$-free, $S/R$ also is.

(4) $\Rightarrow$ (2): Note that $\fkc S_M$ is a canonical ideal of $S_M$ for all $M\in \Max S$ since $\fkc=K:S\cong \Hom_R (S, \rmK_R)$ (see \cite[Theorem 3.3.7(b)]{BH}). Hence, $S/\fkc$ is a Gorenstein ring by \cite[Proposition 3.3.18 (b)]{BH}. Since $R/\fkc \to S/\fkc$ is flat, $R/\fkc$ is also a Gorenstein ring. Thus, we have $S/K\cong R/\fkc$ by Lemma \ref{tttp4.9}. Hence, using the exact sequence $0 \to K/R \to S/R \to S/K \to 0$, $K/R$ is an $R/\fkc$-free module.

Hence, (1)-(4) are equivalent. It follows that (4) $\Rightarrow$ (5) by (i).

(4) $\Rightarrow$ (6): Note that $K/R$ is also an $R/\fkc$-free module. Hence, $S/R$ is an $R/\fkc$-free module of rank $\rmr(R)$ by considering the exact sequence $0 \to K/R \to S/R \to S/K\cong R/\fkc \to 0$. 
It follows that $\rme_{\omega}^{1} (R)=\ell_R(S/R)=\ell_R(R/\fkc){\cdot}\rmr(R)$ by \cite[Lemma 2.1]{GMP}.

(6) $\Rightarrow$ (2): By \cite[Lemma 2.1]{GMP}, $\rme_{\omega}^{1} (R)=\ell_R(S/R)$. Hence, we have $\ell_R(S/R)=\ell_R(R/\fkc){\cdot}\rmr(R)$. It follows that 
\[
\ell_R(K/R)=\ell_R(S/R)-\ell_R(S/K)=(\rmr(R)-1){\cdot}\ell_R(R/\fkc)
\] 
since $S/K$ is the canonical $R/\fkc$-module by Lemma \ref{tttp4.9}. Therefore, $K/R$ is an $R/\fkc$-free module because there is a surjection $(R/\fkc)^{\oplus (\rmr(R)-1)} \to K/R$.
\end{proof} 

\begin{rem}
The conditions $\rm(i)$ and $\rm(ii)$ of Theorem \ref{a4.7} do not imply that the ring $R$ is a $\GGL$ ring. For example, set $R=k[[t^4, t^{7}, t^{9}, t^{10}]]\subseteq k[[t]]$, where $k[[t]]$ denotes a formal power series ring over a field $k$. Then,  $R$ satisfies the conditions $\rm(i)$ and $\rm(ii)$ but $R$ is not a $\GGL$ ring (\cite[Example 3.5]{CGKM}). 
This example also shows that $2$-$\AGL$ rings need not be $\GGL$ rings. Here, we say that $R$ is a $2$-$\AGL$ ring if $K^2=K^3$ and $\ell_R (K^2/K)=2$ (see \cite[Theorem 1.4]{CGKM}). 
$2$-$\AGL$ rings are another generalization of $\AGL$ rings observed by the Hilbert function of canonical ideals, and $2$-$\AGL$ rings have the similar properties to one-dimensional $\GGL$ rings (\cite[Proposition 3.3 (4)]{CGKM}).
On the othe hand, $2$-$\AGL$ rings are defined only in dimension one because canonical ideals are $\fkm$-primary only in dimension one.
\end{rem}

\begin{cor}\label{a4.8}
Let $(R_1, \m_1)$ be a Cohen-Macaulay local ring of dimension one, and let $\varphi : R \to R_1$ be a flat local homomorphism such that $R_1/\m R_1$ is a Gorenstein ring. 
Then,  the following conditions are equivalent:
\begin{enumerate}[{\rm (1)}]
\item $R$ is a $\GGL$ ring and
\item $R_1$ is a $\GGL$ ring.
\end{enumerate}
\end{cor}

\begin{proof}
By {\cite[Proposition 3.8]{CGKM}}, for each $n \ge 0$, the following assertions hold true:
\begin{enumerate}[{\rm (i)}]
\item $K_1^n = K_1^{n+1}$ if and only if $K^n = K^{n+1}$ and 
\item $\ell_{R_1}(K_1^{n+1}/K_1^n) = \ell_{R_1}(R_1/\m R_1){\cdot}\ell_{R}(K^{n+1}/K^n)$,
\end{enumerate}
where  $K_1 = R_1{\cdot} K \cong \rmK_{R_1}$.
Hence, $K^2 = K^{3}$ if and only if $K_1^2 = K_1^{3}$ and $K/R$ is $R/(R:K)$-free if and only if $K_1/R_1$ is $R_1/(R_1:K_1)$-free. Therefore, $R$ is a $\GGL$ ring if and only if $R_1$ is a $\GGL$ ring by Lemma \ref{a4.6} and Theorem \ref{a4.7}.
\end{proof}

\begin{cor}\label{a4.9}
Suppose that $\rmr(R)=2$. Then,  $R$ is a $\GGL$ ring if and only if $K^2=K^3$.
\end{cor}

\begin{proof}
Since $\rmr(R)=2$, we have $K/R\cong R/(R:K)$.
Hence, $R$ is a $\GGL$ ring if and only if $K^2=K^3$ by Lemma \ref{a4.6}.
\end{proof}

\begin{cor}\label{a4.10}
If $e(R)\le3$, then $R$ is a $\GGL$ ring.
\end{cor}

\begin{proof}
We may assume that $R$ is not a Gorenstein ring and $R/\fkm$ is infinite.
Then, because $\rmv(R)=\rme(R)=3$, we have $K^2=K^3$ (see \cite[Proposition 3.7]{Ku} or \cite[Corollary 3 (p.451)]{ES}) and $\rmr(R)=2$. 
\end{proof}

\subsection{Endomorphism algebra of the maximal ideal}
We maintain Setup \ref{setting1}.
Let $B=\fkm:\fkm\cong \Hom_R (\fkm, \fkm)$ be the endomorphism algebra of the maximal ideal $\fkm$. Let $\rmJ(B)$ denote the Jacobson radical of $B$. Our next aim is to explore the $\GGL$ property of $B$ in connection with that of $R$. Let us begin with the following.

\begin{prop}\label{a4.22}
Suppose that $R$ is a $\GGL$ ring but not an $\AGL$ ring. 
Then,  we have the following assertions:
\begin{enumerate}[{\rm (1)}]
\item $B$ is a local ring with the maximal ideal $\rmJ(B)=\fkm S\cap B$.
\item $R/\fkm \cong B/\rmJ(B)$.
\end{enumerate}

\end{prop}

\begin{proof}
Let $r=\rmr(R)$ be the Cohen-Macaulay type of $R$. By Proposition \ref{a4.1} and Theorem \ref{a4.7}, we can choose elements $f_1, f_2, \dots, f_{r-1}\in K$, $g\in S$, and $v\in \fkc:_R \fkm$ such that 
\begin{center}
$K=R+\left<f_1, f_2, \dots, f_{r-1}\right>$, $S=K+\left<g\right>$, and $\fkc:_R \fkm=\fkc+(v)$.
\end{center}
We then show that $B=R+\left<vf_1, vf_2, \dots, vf_{r-1}\right>+\left<vg\right>$. 
Indeed, we have $\fkm vf_i\in \fkc K=\fkc\subseteq R$ for all $1 \le i \le r-1$; hence, $vf_i\in R:\fkm=\fkm:\fkm=B$  for all $1 \le i \le r-1$. Since $K/R\cong (R/\fkc)^{\oplus (r-1)}$ and $R/\fkc$ is Gorenstein, we have $\ell_R([(R:~\fkm)\cap~K]/R)=\rmr_R(K/R)=r-1$. Hence, 
\[
B\cap K=(R:\fkm)\cap K=R+\left<vf_1, vf_2, \dots, vf_{r-1}\right>.
\]

On the other hand, since $S/K \cong R/\fkc$ and $\overline{g}$ is a free basis of $S/K$, we obtain that $vg\not\in K$ and $vg\in (\fkc:_R\fkm)S\subseteq \fkc:\fkm \subseteq B$. Hence, we have the equality $B=R+\left<vf_1, vf_2, \dots, vf_{r-1}\right>+\left<vg\right>$ because $\ell_R((R:\fkm)/R)=r$.

Therefore, since $\left<vf_1, vf_2, \dots, vf_{r-1}\right>+\left<vg\right>\subseteq \fkm S\cap B$, we obtain that $B=R+\fkm S\cap B$. 
It follows that we have a canonical surjection $R \to B/(\fkm S\cap B)$ and the kernel contains $\fkm$. This leads that $R/\fkm \cong B/(\fkm S\cap B)$. Hence,  $\fkm S\cap B\in \Max B$. 
On the other hand, for all $\fkM\in \Max B$, there exists $\fkN\in \Max S$ such that $\fkM=\fkN\cap B$. Since $\fkN\cap R=\fkm$, $\fkm S \subseteq \fkN$. Thus, $\fkm S\cap B\subseteq \fkN\cap B=\fkM$. Hence, $\fkm S\cap B\subseteq \rmJ(B)$.
It follows that $B$ is a local ring and the unique maximal ideal is $\rmJ(B)=\fkm S\cap B$.
\end{proof}

In the proof of Proposition \ref{a4.22}, we use the assumption that $R$ is not an $\AGL$ ring to choose $v\in \fkm$. Indeed, if $R$ is an $\AGL$ ring, then $\fkc=\fkm$. It follows that $v\in \fkc:_R \fkm$ must be a unit of $R$. We note that Proposition \ref{a4.22} does not hold true without the assumption that $R$ is not an $\AGL$ ring. 

\begin{ex}\label{a4.22.5}
Let $k[[t]]$ be the formal power series ring over an infinite field $k$, and set $A=k[[t^3, t^4, t^5]]$. We consider the fiber product 
\[
R=A \times_k  A := k{\cdot}(1,1) + \rmJ(A\times A)\subseteq A\times A
\]
with respect to the natural surjection $A \to k$, where $\rmJ(A\times A)$ denotes the Jacobson radical of the direct product $A\times A$. Since $A$ is an $\AGL$ ring, $R$ is also an $\AGL$ ring with $\rmv(R) = \rme(R)=6$ by \cite[Theorem 1.1 and Proposition 2.2]{EGI}. We note that $\overline{R}=k[[t]]\times k[[t]]$ (\cite[Setting 4.1]{EGI}). Therefore, since $R\subseteq \fkm:\fkm\subseteq \overline{R}$ and 
\begin{align*}
\ell_R(\overline{R}/R) = \ell_R(\overline{R}/\fkm)-1 
=5=\rme(R)-1=\rmr(R)=\ell_R((\fkm:\fkm)/R),
\end{align*}
we have $\fkm : \fkm = \overline{R} = k[[t]]\times k[[t]]$, which is not a local ring.

\end{ex}


\begin{fact}{\rm (cf. \cite[Fact 2.6]{HKS2})}\label{a4.23}
Let $(A, \fkn)$ be a Cohen-Macaulay local ring of dimension $d=\dim A\ge0$. Assume that $\rme(A)>1$.
Then,  $\rmr(A)\le \rme(A)-1$ holds. The equality holds true if and only if $A$ has minimal multiplicity, that is, $\rmv(A)=\rme(A)+d-1$.
\end{fact}

\begin{fact} {\rm (\cite[Proposition 5.1]{CGKM})} \label{f420}
Suppose that there exists an element $\alpha\in \fkm$ such that $\fkm^2=\alpha \fkm$ and that $R$ is not a Gorenstein ring. Set $L=BK$. Then the following assertions hold true. 
\begin{enumerate}[$\rm(1)$]
\item $B = R: \m = \frac{\m}{\alpha}$.
\item $B \subseteq L \subseteq \overline{B}$ and $L=K:\fkm \cong \rmK_B$ as a $B$-module.
\item $S = B[L]$.
\end{enumerate}
\end{fact}

We can now present a theorem illustrating the relationship between the $\GGL$ properties of $R$ and $B$.

\begin{thm}\label{a4.24}
Suppose that there is an element $\alpha \in \fkm$ such that $\fkm^2 = \alpha \fkm$. Set $\fkn=\rmJ(B)$ and $L=BK$.
Then,  the following conditions are equivalent:
\begin{enumerate}[{\rm (1)}]
\item $R$ is a $\GGL$ ring but not an $\AGL$ ring.
\item $B$ is a $\GGL$ ring with $\rmv(B)=\rme(B)=\rme(R)$ but not a Gorenstein ring.
\end{enumerate}
When this is the case, we have the following:
\begin{enumerate}[{\rm (i)}]
\item $R/\fkm \cong B/\fkn$,
\item $\ell_B(B/(B:B[L]))=\ell_R(R/\fkc)-1$, and
\item $\fkn^2 = \alpha \fkn$.
\end{enumerate}

\end{thm}

\begin{proof}
Note that $\ell_R(L/K)=1$ because $(K:\fkm)/K \cong R/\fkm$ by Fact \ref{f420} and Remark \ref{a4.3}.
Set $\fka=B:B[L]=B:S$.

(1) $\Rightarrow$ (2): We have the following commutative diagram
\[
\xymatrix{
0 \ar[r] & B/R\ar[r] \ar[d]^\cong&S/R \ar[r] \ar[d]^\cong & S/B\ar[r] \ar@{=}[d] & 0\\
0 \ar[r] & (R/\fkm)^{\oplus r}\ar[r]^\varphi  &  (R/\fkc)^{\oplus r} \ar[r] & S/B \ar[r] & 0 \\
}
\]
as $R$-modules.
Since $R/\fkc$ is an Artinian Gorenstein ring, $\Im \varphi=((\fkc:_R \fkm)/\fkc)^{\oplus r}$. Hence, we have $S/B\cong (R/(\fkc:_R \fkm))^{\oplus r}$ as $R$-modules.
For an element $a\in R$, we have
$$a\in\fka \Leftrightarrow aS\subseteq B \Leftrightarrow a\fkm S \subseteq R \Leftrightarrow a\in \fkc:_R \fkm,$$
thus $\fka \cap R=\fkc:_R \fkm$. Thus, we have a canonical injection $R/(\fkc:_R \fkm) \to B/\fka$. 
On the other hand, we have $$\ell_R(B/\fka)=\ell_R(S/L)=\ell_R(S/K)-1=\ell_R(R/\fkc)-1=\ell_R(R/(\fkc:_R \fkm)),$$
where the first and third equalities follow by the facts that $\rmK_{B/\fka} \cong S/L$ and $\rmK_{R/\fkc} \cong S/K$ from Lemma \ref{tttp4.9}. It follows that the injection $R/(\fkc:_R \fkm) \to B/\fka$ is bijective; hence, $S/B\cong (R/(\fkc:_R \fkm))^{\oplus r}\cong (B/\fka)^{\oplus r}$ as $R$-modules. It follows that $S/B \cong (B/\fka)^{\oplus r}$ as $B$-modules because there is a surjection $(B/\fka)^{\oplus r} \to S/B$ as $B$-modules. Therefore, $B$ is a $\GGL$ ring, $\ell_B(B/\fka)=\ell_R(R/\fkc)-1$, and $\rmr(R)=\rmr(B)$ by Theorem \ref{a4.7}.
We need to show that $\alpha \fkn=\fkn^2$. Indeed, because $\alpha S$ is a reduction of $\fkm S$, we have
$$B\subseteq \frac{\fkm S\cap B}{\alpha}\subseteq \frac{\fkm S}{\alpha}\subseteq \overline{S}=\overline{B}$$
(\cite[Corollary 1.2.5]{SH}).
Hence, $\alpha B \subseteq \fkm S\cap B = \fkn$ is  a reduction by Proposition \ref{a4.22}. In addition $\rmr (B)=\rmr (R)= \rme(R)-1 = \rme(B)-1$, where the last equality follows from 
$$\rme (R)= \rme _\fkm ^0 (R)=\rme _\fkm ^0 (B) =\ell_R(B/\alpha B) =\ell_B(B/\alpha B) = \rme (B)$$ 
by the multiplicative formula (e.g., see \cite[Theorem 14.6]{Mat}) and Proposition \ref{a4.22}(2). Therefore, by Fact \ref{a4.23}, we have $\rme (B)=\rmv (B)=\rme (R)$. Thus, $\alpha \fkn=\fkn^2$.

(2) $\Rightarrow$ (1): Consider the following exact sequences:
\begin{align*}
&0 \to L/K \to S/K \to S/L \to 0,\\
&0 \to B/R \to S/R \to S/B \to 0, \text{ and}\\
&0 \to K/R \to S/R \to S/K \to 0. 
\end{align*}
Note that $\rmr(R)=\rme(R)-1=\rme(B)-1=\rmr(B)$ by hypothesis and Fact \ref{a4.23}, and $S/K\cong \rmK_{R/\fkc}$ and $S/L\cong \rmK_{B/\fka}$ by Lemma \ref{tttp4.9}. Hence, we obtain
\begin{align*}
\ell _R (K/R) + \ell_R (R/\fkc) =& \ell _R (K/R) + \ell_R (S/K) =\ell_R (S/R)\\
=&  \ell_R(B/R) + \ell_R(S/B) =\rmr(R){\cdot}(\ell_R (B/\fka)+1)\\
=& \rmr(R){\cdot}\ell_R (R/\fkc),
\end{align*}
where the fourth equality follows from the hypothesis that $B$ is a $\GGL$ ring with respet to $\fka$ and the last equality follows from $\ell_R(B/\fka)=\ell_R(S/L)=\ell_R(S/K)-1=\ell_R(R/\fkc)-1$. 
Therefore, since $\ell _R (K/R) = (\rmr(R)-1){\cdot}\ell_R (R/\fkc)$ and there is a surjection $(R/\fkc)^{\oplus (\rmr(R)-1)} \to K/R$, $K/R$ is an $R/\fkc$-free module. 
\end{proof}

Theorem \ref{a4.24} assumes that $R$ is not an $\AGL$ ring, whereas it is known that $R$ is an $\AGL$ ring possessing minimal multiplicity if and only if $B\cong \Hom_R (\fkm, \fkm)$ is a (not necessarily local) Gorenstein ring (\cite[Theorem 5.1]{GMP}). In this case, we obtain that $B=S$. Therefore, by combining Theorem \ref{a4.24} and \cite[Theorem 5.1]{GMP}, a $\GGL$ ring $R$ possessing minimal multiplicity reaches a Gorenstein ring $S$ by taking the endomorphism algebra of the maximal ideal in $\ell_R (R/\fkc)$ steps.

\begin{cor}\label{a4.25}
Suppose that $R$ is a $\GGL$ ring but not a Gorenstein ring. If $\rmv(R)=\rme(R)$, then $S$ is a Gorenstein ring.
\end{cor}

\begin{proof}
After enlarging the residue class field of $R$, we may assume that $R/\fkm $ is infinite. This proof uses the induction on $n:=\ell_R(R/\fkc)$. If $n=1$, then $R$ is an $\AGL$ ring. Therefore, $S = B$ is a Gorenstein ring  by \cite[Theorem 5.1]{GMP}. Suppose that $n>1$ and our assertion holds true for $n-1$. Then,  $B$ is a $\GGL$ ring such that $\ell_B(B/(B:B[BK]))=n-1$ by Theorem \ref{a4.24}. Hence, $S=B[BK]$ is a Gorenstein ring.
\end{proof}

As an application of Theorem \ref{a4.24}, we determine $\tr_R(\rmK_R)$ for non-Gorenstein $\GGL$ rings having minimal multiplicity.

\begin{thm}\label{a7.12}
Suppose that $R$ is a non-Gorenstein $\GGL$ ring. 
Assume that $\fkm^2=\alpha \fkm$ for some element $\alpha \in \fkm$.
Set $v:=\rmv(R)$ and $n:=\ell_R(R/\tr_R(\rmK_R))$.
Then,  there exist elements $x_2, x_3, \dots, x_v \in \fkm$ satisfying the following conditions:
\begin{enumerate}[{\rm (1)}]
\item $\fkm=(\alpha, x_2, x_3, \dots, x_v)$,
\item $\tr_R(\rmK_R)=(\alpha^n, x_2, x_3, \dots, x_v)$, and $(\alpha^n)$ is a minimal reduction of $\tr_R(\rmK_R)$.
\end{enumerate}
\end{thm}

\begin{proof}
We prove this by induction on $n>0$. The case of $n=1$ is trivial because $\tr_R(\rmK_R)=\fkm$. Let $n>1$ and assume that our assertion holds true for $n-1$. $R$ is not an $\AGL$ ring since $\tr_R(\rmK_R) \ne \fkm$. Thus, by Theorem \ref{a4.24}, $B=\fkm:\fkm$ is also a $\GGL$ ring but not a Gorenstein ring. Let $\fkn$ be the unique maximal ideal of $B$. By Theorem \ref{a4.24} and the induction hypothesis, there exist elements $y_2, y_3, \dots, y_v \in \fkn$ that satisfy the following conditions:
\begin{enumerate}[{\rm (a)}]
\item $\fkn=B \alpha + \sum_{i=2}^{v}B y_i$,
\item $\tr_B(\rmK_B)=B\alpha^{n-1} + \sum_{i=2}^{v}B y_i$, and $B\alpha^{n-1}$ is a minimal reduction of $\tr_B(\rmK_B)$.
\end{enumerate}
We need to prove the following.

\begin{claim*}
The following assertions hold true:
\begin{align*}
&\text{{\rm (i)}  $\tr_B(\rmK_B)=\frac{1}{\alpha} \tr_R(\rmK_R)$.}  & &\text{{\rm (ii)} $\fkm=R\alpha + \sum_{i=2}^{v}R \alpha y_i$.} 
\end{align*}
\end{claim*}

\begin{proof}[Proof of Claim]
(i): Since $B$ is a $\GGL$ ring, by noting Fact \ref{f420}, we have $\tr_B(\rmK_B)=B:S=\frac{\fkm}{\alpha}:S=\frac{1}{\alpha} (\fkm:S)$. Note that $\fkc=\fkc:S$ because $\fkc=\fkc S$. Since $\fkc\subseteq \fkm \subseteq R$, it follows that $\fkc=\fkm:S$. Hence, $\tr_B(\rmK_B)=\frac{1}{\alpha} \fkc=\frac{1}{\alpha} \tr_R(\rmK_R)$.

(ii): Since $\fkn^2=\alpha \fkn \subseteq \alpha B=\fkm$, $\fkn/\fkm$ is a $B/\fkn$-vector space. Hence, $\fkn/\fkm=\sum_{i=2}^{v}B/\fkn {\cdot} y_i$ by (a).
By Theorem \ref{a4.24}, we have a natural isomorphism $R/\fkm \cong B/\fkn$, hence $\fkn/\fkm=\sum_{i=2}^{v}R/\fkm {\cdot} y_i$. It follows that $\fkn=\sum_{i=2}^{v}R y_i+\fkm$. Thus, $\alpha \fkn=\sum_{i=2}^{v}R \alpha y_i+\alpha \fkm$. On the other hand, because $\fkm/\alpha \fkn \cong B/\fkn \cong R/\fkm$ and $\alpha \in \fkm \backslash \alpha \fkn$, we have $\fkm=\alpha \fkn + \alpha R$. Therefore, 
\[
\fkm=\alpha \fkn + R \alpha =\sum_{i=2}^{v}R \alpha y_i+\alpha \fkm +  R \alpha.
\] 
This concludes (ii) by Nakayama's lemma.
\end{proof}

Set $J=R\alpha^{n} + \sum_{i=2}^{v}R \alpha y_i$. By (b) and Claim (i), we obtain that $\tr_R(\rmK_R)=B\alpha^{n} + \sum_{i=2}^{v}B \alpha y_i \supseteq J$ and $\ell_R(R/J) \le n$. Hence, we have $\tr_R(\rmK_R)=J$ by the definition of $n$.
It remains to show that $\alpha^n R$ is a reduction of $\tr_R(\rmK_R)$.
Indeed, we have $\alpha^{n-1} R \subseteq \alpha^{n-1} B \subseteq \tr_B(\rmK_B)=\frac{1}{\alpha} \tr_R(\rmK_R) \subseteq \alpha^{n-1} \overline{B}=\alpha^{n-1} \overline{R}$. Hence, $\alpha^{n} R \subseteq \tr_R(\rmK_R) \subseteq \alpha^{n} \overline{R} \cap R$. This implies that $\alpha^n R$ is a reduction of $\tr_R(\rmK_R)$.
\end{proof}

\begin{cor}\label{a7.12.5}
Suppose that $R$ has minimal multiplicity. If $R$ is a $\GGL$ ring, then $\rmv(R/\tr_R(\rmK_R))\le 1$. That is, $R/\tr_R(\rmK_R)$ is a hypersurface.
\end{cor}

Note that Corollary \ref{a7.12.5} does not necessarily hold true without the assumption that $R$ has minimal multiplicity. For instance, see Proposition \ref{a7.6}.


\subsection{The Ulrich property for $\fkc$ and $\tr_R(\rmK_R)$}\label{subsection4.3}

We maintain Setup \ref{setting1}. In this subsection, we study the Ulrich property for $\fkc$ and $\tr_R(\rmK_R)$ with the $\GGL$ property of $R$ (see Definition \ref{Ulrich ideal} for the definition of Ulrich ideals).
 The goal is to prove the following and Theorem \ref{a7.14}.

\begin{thm}\label{a7.5}
Suppose that $R$ is not a Gorenstein ring. Then,  the following conditions are equivalent:
\begin{enumerate}[{\rm (1)}]
\item $\tr_R(\rmK_R)$ is an Ulrich ideal of $R$.
\item $\fkc$ is an Ulrich ideal of $R$.
\item $R$ is a $\GGL$ ring and $S$ is a Gorenstein ring.
\item $R$ is a $\GGL$ ring and $\tr_R(\rmK_R):\tr_R(\rmK_R)$ is a Gorenstein ring.
\end{enumerate}
When this is the case, $\mu_R(\tr_R(\rmK_R))=\rmr(R)+1$.
\end{thm}

We begin with the following.

\begin{lemma}\label{a7.4}
Suppose that $\tr_R(\rmK_R)^2=\alpha {\cdot}\tr_R(\rmK_R)$ for some non-zerodivisor $\alpha \in \tr_R(\rmK_R)$ of $R$ (that is, $\tr_R(\rmK_R)$ is stable in the sense of \cite{L}).
Then,  $\tr_R(\rmK_R)=\fkc$.
\end{lemma}

\begin{proof}
By the hypothesis, we have $\tr_R(\rmK_R):\tr_R(\rmK_R)=(R:K) K:(R:K) K=((R:K) K)^n:((R:K) K)^n$ for all $n>0$ (\cite[Lemma 1.11]{L}). Since $((R:K) K)^n=(R:K)^n S$ for $n \gg 0$, we have $S \subseteq \tr_R(\rmK_R): \tr_R(\rmK_R)$. This shows that $\tr_R(\rmK_R)$ is an ideal of $S$; hence, $\tr_R(\rmK_R) \subseteq \fkc$. The reverse inclusion follows from $\fkc\subseteq R:K\subseteq (R:K) K=\tr_R(\rmK_R)$.
\end{proof}

Lemma \ref{a7.4} proves that Theorem \ref{a7.5} (1) $\Rightarrow$ (2) holds. Thus, in what follows, we explore the Ulrich property for $\fkc$. We proceeds the arguments in more generality that $\fka\subseteq \fkc$.

\begin{lemma}\label{a4.17.5}
Suppose that  $\fka$ is an Ulrich ideal such that $\fka \subseteq \fkc$. Let $(\alpha)\subseteq \fka$ be a minimal reduction of $\fka$, and set $T=\frac{\fka}{\alpha}$. Then,  we have the following:
\begin{enumerate}[{\rm (1)}] 
\item $\fka=R:T=K:T$.
\item $T=R:\fka=\fka:\fka$ and $S\subseteq T$; in particular, $T$ is a subring of $\rmQ(R)$.
\end{enumerate}
\end{lemma}

\begin{proof}
(1): Since $\fka/(\alpha)\cong T/R$ is $R/\fka$-free by \cite[Lemma 2.3(2)]{GOTWY}, considering the annihilator, we obtain $\fka=R:T$. Hence, we have 
\[
T=\frac{\fka}{\alpha}=\frac{R:T}{\alpha}=R:\fka\supseteq R:\fkc= S
\]
by Lemma \ref{useful}.
It follows that $\fka=(K:K):T=K:T$.

(2): $T=R:\fka$ and $S\subseteq T$ were proved in the proof of (1). 
Since $\fka^2=\alpha \fka$, we obtain $T^2=T$. Hence, $T$ is a subring of $\rmQ(R)$ and $T=T:T=\fka:\fka$.
\end{proof}

By Lemma \ref{a4.17.5}, to study when an $\fkm$-primary ideal $\fka (\subseteq \fkc)$ of $R$ is an Ulrich ideal, we may assume that $\fka$ has the form 
\begin{align}\label{ttttt4.22}
\fka=R:T, 
\end{align}
where $T$ is a subring of $\rmQ(R)$ containing $K$ and finitely generated as an $R$-module. In this subsection, suppose that $\fka$ has the form of (\ref{ttttt4.22}). Note that $\fka=K:T$ since $\fka=(K:K):T=K:KT$ and $K\subseteq T$. We then obtain the following.

\begin{lemma}\label{a4.18}
\begin{enumerate}[{\rm (1)}]
\item $\fka T_M$ is isomorphic to the canonical module of $T_M$ for all $M\in \Max T$.
\item The following conditions are equivalent:
\begin{enumerate}[{\rm (i)}]
\item $T$ is a Gorenstein ring,
\item $\fka=\alpha T$ for some $\alpha \in \fka$, and
\item $\fka^2=\alpha\fka$ for some $\alpha \in \fka$.
\end{enumerate}
\end{enumerate}

\end{lemma}

\begin{proof}
(1): This follows from $\fka=K:T\cong \Hom_R (T, \rmK_R)$ and \cite[Theorem 3.3.7(b)]{BH}.

(2) (i) $\Rightarrow$ (ii): If $T$ is a Gorenstein ring, then $\fka T_M \cong T_M$ for all $M\in \Max T$ by (1). Hence, because $T$ is a semi-local ring, $\fka \cong T$ and $\fka=\alpha T$ for some $\alpha \in \fka$.

(ii) $\Rightarrow$ (iii): This is trivial.

(iii) $\Rightarrow$ (i): Assume that $\fka^2=\alpha\fka$ for some $\alpha \in \fka$ and set $L=\frac{\fka}{\alpha}$. Then,  $T\subseteq L\subseteq \overline{T}$ because $L^2=L$. Hence, $T_M \subseteq L_M\subseteq \overline{T_M}$ and $L_M\cong \rmK_{T_M}$ for all $M\in \Max T$. Therefore, because $L_M^2=L_M$, $T_M$ is a Gorenstein ring by Fact \ref{a4.4}.
\end{proof}

\begin{prop}\label{a4.19}
Suppose that $R$ is not a Gorenstein ring.
Then,  the following conditions are equivalent:
\begin{enumerate}[{\rm (1)}]
\item $\fka$ is an Ulrich ideal of $R$.
\item $T$ is a Gorenstein ring and $T/R$ is $R/\fka$-free.
\end{enumerate}
When this is the case, $R/\fka$ is a Gorenstein ring and $\mu_R(\fka)=\mu_R(T)$.
\end{prop}

\begin{proof}
(1) $\Leftrightarrow$ (2): By Lemma \ref{a4.18}, we may assume that $\fka=\alpha T$ for some $\alpha \in \fka$. We then show that $T/R$ is $R/\fka$-free if and only if $\fka/\fka^2$ is $R/\fka$-free. Indeed, the equivalence follows from the  exact sequence  $0 \to R/\fka \to T/\fka \to T/R \to 0$ and the isomorphism $T/\fka \cong \fka/\fka^2$.

When this is the case, $T/\fka=T/\alpha T$ is a Gorenstein ring because $\alpha$ is a non-zerodivisor of $T$ contained in the Jacobson radical of $T$.
Since $T/\fka$ is $R/\fka$-free, we have $R/\fka \to T/\fka$ is a flat homomorphism. Hence, we obtain that $R/\fka$ is a Gorenstein ring. The equality $\mu_R(\fka)=\mu_R(T)$ follows from $\fka\cong T$.
\end{proof}

\begin{proof}[Proof of Theorem \ref{a7.5}]
(1) $\Rightarrow$ (2): This follows from Lemma \ref{a7.4}.

(2) $\Leftrightarrow$ (3): This follows from Theorem \ref{a4.7} and Proposition \ref{a4.19} by applying $T=S$ and $\fka=\fkc$. 

(2) $\Rightarrow$ (1): By the above equivalence,  $R$ is a $\GGL$ ring. Hence, $\fkc=\tr_R(\rmK_R)$ by Theorem \ref{a4.7}(i).

(3) $\Leftrightarrow$ (4): This follows from the equality $\tr_R(\rmK_R):\tr_R(\rmK_R)=\fkc:\fkc=S$ by Lemma \ref{useful} and Theorem \ref{a4.7}.

When this is the case, $\mu_R(\tr_R(\rmK_R))=\rmr(R)+1$ by Fact \ref{a7.0.2}(1).
\end{proof}

\begin{cor}\label{a4.26}
Suppose that $\rmv(R)=\rme(R)\ge 3$. 
Then,  $R$ is a $\GGL$ ring if and only if $\fkc$ is an Ulrich ideal, if and only if $\tr_R(\rmK_R)$ is an Ulrich ideal.
\end{cor}

\begin{proof}
Note that $R$ is not Gorenstein by the hypothesis. Thus, by combining Corollary \ref{a4.25} and Theorem \ref{a7.5}, we have the assertion.
\end{proof}


We next aim to describe the set of all Ulrich ideals of $\GGL$ rings having minimal multiplicity. Let $\calX_R$ denote the set of all Ulrich ideals of $R$. We need a lemma.

\begin{lemma}\label{a7.13}
Let $A$ be a ring.
Suppose that $Q_0$, $Q$, and $J$ are ideals of $A$. Set 
$$
I_0=Q_0+J \ \ \text{and}\ \  I=Q+J.
$$
Assume that $Q_0 \subseteq Q$. Then, $I_0^{m+1}=Q_0 I_0^{m}$ implies that $I^{m+1}=Q I^{m}.$
\end{lemma}

\begin{proof}
Since $I_0^{m+1}=Q_0 {I_0}^m+J^{m+1}=Q_0 {I_0}^m$, we have $J^{m+1} \subseteq Q_0 {I_0}^m \subseteq Q {I}^m$. Therefore, $I^{m+1}=Q {I}^m+J^{m+1}=Q {I}^m$.
\end{proof}

\begin{thm}\label{a7.14}
Suppose that $R$ is not a Gorenstein ring. Set $v=\rmv(R)$ and $n=\ell_R(R/\tr_R(\rmK_R))>0$. Then,  the following conditions are equivalent:
\begin{enumerate}[{\rm (1)}]
\item $R$ is a $\GGL$ ring having minimal multiplicity.
\item $\tr_R(\rmK_R)$ and $\fkm$ are Ulrich ideals.
\item $R$ is $G$-regular in the sense of \cite{Tak}, and a length of a maximal chain of Ulrich ideals is $n-1$.
\item There exist elements $\alpha, x_2, x_3, \dots, x_v \in \fkm$ satisfying the following conditions:
\begin{enumerate}[{\rm (i)}]
\item $\fkm=(\alpha, x_2, x_3, \dots, x_v)$ and
\item $\calX_R=\left\{ (\alpha^i, x_2, x_3, \dots, x_v) \mid 1\le i \le n \right\}$.
\end{enumerate}

\end{enumerate}

\end{thm}

\begin{proof}
(4) $\Rightarrow$ (3): Set $I_i=(\alpha^i, x_2, x_3, \dots, x_v)$ and $R_i=R/I_i$ for $1\le i \le n$.
We need to show that $I_{i+1} \subsetneq I_i$ for all $1 \le i \le n-1$.
Assume that $I_i = I_{i+1}$ for some $1 \le i \le n-1$. Then, $\alpha^i=c_1 \alpha^{i+1} + \sum_{j=2}^v c_j x_j$ for some $c_1, c_2, \dots, c_v \in R$. 
This shows that $\alpha^i\in (x_2, x_3, \dots, x_v)$ because $1-c_1 \alpha$ is a unit element in $R$. Thus, $I_i=(x_2, \dots, x_v)$.
On the other hand, $R_i$ is Gorenstein because $\mu_{R_i}(\fkm/I_i)\le 1$. 
Therefore, we have
$$
\mu_R (I_i)=1+\rmr(R)=\rme(R)=v
$$
by Fact \ref{a7.0.2}(1) and Fact \ref{a4.23}. This contradicts the equation $I_i=(x_2, \dots, x_v)$. Therefore,
$$I_n \subsetneq I_{n-1} \subsetneq \cdots \subsetneq I_1=\fkm$$
is a chain of Ulrich ideals. $R$ is $G$-regular because $\fkm\in \calX_R$ by \cite[Corollary 2.5]{Y}.

(3) $\Rightarrow$ (2): Choose $J_0, J_1, \dots, J_{n-1} \in \calX_R$ such that $R \supsetneq J_0 \supsetneq J_1 \supsetneq \cdots \supsetneq J_{n-1}$. Thus, $\ell_R(R/J_{n-1}) \ge n$. note that $J_{n-1} \supseteq \tr_R(\rmK_R)$ by Corollary \ref{a7.3}; hence, we have $J_{n-1}=\tr_R(\rmK_R)$ and $J_0=\fkm$.

(1) $\Leftrightarrow$ (2): This follows from Corollary \ref{a4.26}.

(1) $\Rightarrow$ (4): By Theorem \ref{a7.12}, there exist elements $x_2, x_3, \dots, x_v \in \fkm$ that satisfy the following conditions:
\begin{enumerate}[{\rm (a)}]
\item $\fkm=(\alpha, x_2, x_3, \dots, x_v)$,
\item $\tr_R(\rmK_R)=(\alpha^n, x_2, x_3, \dots, x_v)$, and $(\alpha^n)$ is a minimal reduction of $\tr_R(\rmK_R)$.
\end{enumerate}
On the other hand, since $R$ has minimal multiplicity, $R$ is $G$-regular (\cite[Corollary 2.5]{Y}). Hence, by  Corollary \ref{a7.3}, we have 
$$\calX_R \subseteq \left\{ (\alpha^i, x_2, x_3, \dots, x_v) \mid 1\le i \le n \right\}$$
Set $I_i=(\alpha^i, x_2, x_3, \dots, x_v)$ for $1\le i \le n$.
Note that $I_{i+1} \subsetneq I_i$ and $\mu_R (I_i)=v$ for all $1 \le i \le n-1$ because $I_{n}=\tr_R(\rmK_R)$ and $\mu_R (\tr_R(\rmK_R))=1+\rmr (R)=v$ by Fact \ref{a7.0.2}(1), Fact \ref{a4.23}, and Theorem \ref{a4.7}(ii).
By Lemma \ref{a7.13}, $I_i^2=\alpha^{i} I_i$ holds. Therefore, to show that $I_i$ is an Ulrich ideal for all $1\le i \le n$, we have only to show that $I_i/I_i^2$ is an $R/I_i$-free module for $1\le i \le n$. This is equivalent to showing that $I_i/(\alpha^i)$ is an $R/I_i$-free module for $1\le i \le n$ by the exact sequence
$$
0\to (\alpha^i)/I_i^2 \to I_i/I_i^2 \to I_i/(\alpha^i) \to 0
$$
and that $(\alpha^i)/I_i^2=(\alpha^i)/\alpha^i I_i \cong R/I_i$. 
We prove the $R/I_i$-freeness of $I_i/(\alpha^i)$ by descending induction on $1 \le i \le n$.
The case where $i=n$ follows from Corollary \ref{a4.26}. Let $1 \le i<n$ and assume that our assertion holds true for $i+1$. 
Then, 
\begin{align*}
\ell_R(I_i/(\alpha^i))=&\ell_R(R/(\alpha^{i+1}))-\ell_R(R/I_i)-\ell_R((\alpha^i)/(\alpha^{i+1}))\\
=&[\ell_R(R/I_{i+1}) + \ell_R(I_{i+1}/(\alpha^{i+1}))]-\ell_R(R/I_i)-\ell_R(R/(\alpha))\\
=&[(i+1) + (i+1)(v-1)]-i-\rme(R)\\
=&i(v-1).
\end{align*}
Thus, $I_i/(\alpha^i)$ is an $R/I_i$-free module because there is a surjection $(R/I_i)^{\oplus (v-1)} \to I_i/(\alpha^i)$.
\end{proof}

\begin{cor}\label{a7.15}
Suppose that $\rme(R)=\rmv(R)=3$. Then,  there exist elements $\alpha, x_2, x_3 \in \fkm$ that satisfy the following conditions:
\begin{enumerate}[{\rm (i)}]
\item $\fkm=(\alpha, x_2, x_3)$ and
\item $\calX_R=\left\{ (\alpha^i, x_2, x_3) \mid 1\le i \le \ell_R (R/\tr_R(\rmK_R)) \right\}$.
\end{enumerate}

\end{cor}

\begin{proof}
By Corollary \ref{a4.10}, $R$ is a $\GGL$ ring if $\rme(R)=\rmv(R)=3$. Thus, the assertion follows from Theorem \ref{a7.14}.
\end{proof}

\begin{Example}\label{a7.16}
Let $k$ be a field and $k[[t]]$ be the formal power series ring over $k$. Then,  the following assertions hold true:
\begin{enumerate}[{\rm (1)}]
\item Let $R_1 =k[[t^{5}, t^{6}, t^{8}]]$. Then,  $\tr_{R_1}(\rmK_{R_1})=(t^{10}, t^6, t^8)$ is an Ulrich ideal of $R_1$. On the other hand, $R_1$ does not have minimal multiplicity.
\item Let $R_2 =k[[t^{5}, t^{18}, t^{26}, t^{34}, t^{42}]]$. Then,  $\tr_{R_2}(\rmK_{R_2})=(t^{10}, t^{18}, t^{26}, t^{34}, t^{42})$ is an Ulrich ideal of $R_2$. Moreover, $R_2$ has minimal multiplicity. Hence, $\calX_{R_2} = \left\{ \tr_{R_2}(\rmK_{R_2}), \fkm \right\}$.
\item Let $n>0$ and $R_3 =k[[t^{3}, t^{3n+1}, t^{3n+2}]]$. Then, 
\[
\calX_{R_3} = \left\{ (t^{3i}, t^{3n+1}, t^{3n+2}) \mid 1 \le i \le \ell_{R_3}(R_3/\tr_{R_3} (\rmK_{R_3})) = n \right\}.
\]
\end{enumerate}

\end{Example}

\subsection{$\GGL$ rings arising from idealizations}\label{subsection4.4}
Next, we explore $\GGL$ rings arising from idealizations. To recall the definition of idealization and the fundamental properties of idealizations, for a little while, let $R$ be an arbitrary ring and $M$ an $R$-module. Then the direct sum $R \oplus M$ with coordinate-wise addition and multiplication 
\[
(a,x)(b,y) = (ab, bx + ay),
\]
where $a,b \in R$ and $x,y \in M$, is a ring. The ring is called the {\it idealization of $M$ over $R$}, and we denote by $R \ltimes M$ the idealization (\cite{AW}). Set $A=R \ltimes M$.

Let $K$ be an $R$-module and set $L = \Hom_R(M, K) \oplus K$. We consider $L$ to be an $A$-module under the following action of $A$:
$$(a,x)\circ (f,y) = (af, f(x) + ay),$$
where $(a,x) \in A$ and $(f,y) \in L$. Then,  it is standard to check that the map
\begin{align}\label{idealizationeq}
\Hom_R(A,K) \to L, ~\alpha \mapsto (\alpha \circ j, \alpha (1))
\end{align}
is an isomorphism of $A$-modules, where $j : M \to A;~ x \mapsto (0,x)$, and $1 = (1,0)$ denotes the identity of the ring $A$ (\cite[Section 4]{CGKM}).

We are now back to Setup \ref{setting1}. 
The problem here is as follows: for a nonzero $R$-module $M$, when does the idealization $R\ltimes M$ become a $\GGL$ ring? It is known that $R\ltimes M$ is Gorenstein if and only if $M\cong \rmK_R$ (\cite{R}). Therefore, since $R$ is generically Gorenstein, either $M_{\fkp}\cong \rmK_{R_\fkp} \cong R_\fkp$ or $M_\fkp\cong 0$ holds for each $\fkp \in \Ass R$ (recall that $\Ass A=\{\fkp \oplus M \mid \fkp\in \Ass R\}$, see \cite[Theorem 3.2]{AW}). From this observation, we concentrate on the case where $M$ is an $\fkm$-primary ideal $\fka$ of $R$. 

\begin{setup}
Suppose that Setup \ref{setting1}. 
Let $\fka$ be an $\fkm$-primary ideal of $R$. Set $A=R\ltimes \fka$ and $L = (K:\fka) \oplus K$.
\end{setup}

\begin{rem}
 $A$ is a one-dimensional Cohen-Macaulay local ring and 
$\rmK_A \cong \Hom_R(A,K) \cong L$
as $A$-modules by \eqref{idealizationeq}.
Moreover, 
$$  A \subseteq L  \subseteq \overline{R} \ltimes \rmQ(R)=\overline{A}$$
(\cite[Theorem 4.1]{AW}). Thus, $A$ also satisfies the assumption of Setup \ref{setting1}. 
\end{rem}

\begin{lem}\label{a4.14}
The following assertions hold true:
\begin{enumerate}[{\rm (1)}]
\item $L^n=(K:\fka)^n\oplus (K:\fka)^{n-1}K$ for all $n>0$.
\item $A:L=(\fka : K) \oplus [\fka:(K:\fka)]$.
\end{enumerate}
\end{lem}

\begin{proof}
(1): We prove by induction on $n>0$. Assume that $n>1$ and the assertion holds for $n-1$. Then 
\begin{align*}
L^n=&L^{n-1}{\cdot}L=[(K:\fka)^{n-1}\oplus (K:\fka)^{n-2}K]{\cdot}[(K:\fka) \oplus K]\\
=&(K:\fka)^n\oplus (K:\fka)^{n-1}K.
\end{align*}

(2): For an element $\alpha\in \rmQ(A)=\rmQ(R)\oplus \rmQ(R)$, we have the following. Write $\alpha=(a,x)$ where $a, x\in \rmQ(R)$.
\begin{align*}
\alpha\in A:L&\Leftrightarrow (a,x)L\subseteq A \Leftrightarrow a(K:\fka) \oplus [aK+x(K:\fka)]\subseteq R\ltimes \fka\\
&\Leftrightarrow a\in [R:(K:\fka)]  \cap (\fka:K) \quad \text{and} \quad x\in \fka:(K:\fka).
\end{align*}
By noting that $R:(K:\fka)=(K:K):(K:\fka)=[K:(K:\fka)]:K=\fka:K$, we have the assertion.
\end{proof}

By Lemma \ref{a4.14}, if we further assume that $\fka$ has the form of (\ref{ttttt4.22}), we have the following.

\begin{prop}\label{a4.15}
Suppose that $R$ is not a Gorenstein ring. If $\fka$ has the form of (\ref{ttttt4.22}), then the following conditions are equivalent:
\begin{enumerate}[{\rm (1)}]
\item $A$ is a $\GGL$ ring.
\item $R$ is a $\GGL$ ring and $S=T$.
\item $T/R$ is $R/\fka$-free.
\end{enumerate}
When this is the case, we have $\fka=\fkc$.
\end{prop}

\begin{proof}
Note that $\fka=K:T$ (see after (\ref{ttttt4.22})). Hence, we have 
\begin{align*}
K:\fka&=K:(K:T)=T,\\
\fka:K&=(K:T):K = K:KT = K:T=\fka,\\
\fka:(K:\fka)&=\fka:T=(K:T):T=K:TT=K:T=\fka.
\end{align*}
It follows that 
\begin{center}
$L^2=L^3$, \ $L/A\cong T/R\oplus K/\fka$, \  $A[L]/A\cong T/R\oplus T/\fka$, \ and \ $A/(A:L)\cong R/\fka$
\end{center}
 by Lemma \ref{a4.14}. Note that $A:L=A:A[L]$ by Lemma \ref{a4.6}.

(1) $\Leftrightarrow$ (3): By Theorem \ref{a4.7}, we have 
\begin{align*}
\text{$A$ is a $\GGL$ ring} &\Leftrightarrow \text{$A[L]/A$ is $A/(A : L)$-free} \\
 &\Leftrightarrow \text{$T/R \oplus T/\fka$ is $R/\fka$-free}\\
 &\Leftrightarrow \text{$T/R$ is $R/\fka$-free.}
\end{align*}

(1) $\Leftrightarrow$ (2): By Theorem \ref{a4.7}, we have 
\begin{align*}
\text{$A$ is a $\GGL$ ring} &\Leftrightarrow \text{$L/A$ is $A/(A : L)$-free}\\
&\Leftrightarrow  \text{$T/R$ and $K/\fka$ are $R/\fka$-free}\\
 &\Leftrightarrow \text{$T/R$ and $K/R$ are $R/\fka$-free}.
\end{align*}
Here, by considering the annihilators of $K/R$ and $R/\fka$, we obtain that $\fka=R:K$. It follows that $\fka=R:K\supseteq R:S\supseteq R:T=\fka$.
Thus, $K:S=R:S=R:T=K:T$ by Lemma \ref{useful}. By applying the $K$-dual, $T=S$. Therefore, 
\begin{align*}
\text{$T/R$ and $K/R$ are $R/\fka$-free} &\Leftrightarrow  \text{$T=S$, $T/R$ is $R/\fka$-free, and $R$ is a $\GGL$ ring}\\
&\Leftrightarrow \text{$T=S$ and $R$ is a $\GGL$ ring}
\end{align*}
by Theorem \ref{a4.7}.
\end{proof}

\begin{cor}\label{a4.16}
Suppose that $R$ is not a Gorenstein ring.
Then,  the following conditions are equivalent:
\begin{enumerate}[{\rm (1)}]
\item $R$ is a $\GGL$ ring and
\item $R\ltimes \fkc$ is a $\GGL$ ring.
\end{enumerate}
\end{cor}


Combining Propositions \ref{a4.19} and \ref{a4.15}, we have the following.

\begin{cor}\label{a4.20}
Suppose that $R$ is not a Gorenstein ring. If there exists an Ulrich ideal $\fka$ of $R$ contained in $\fkc$, then $\fka=\fkc$. Hence, $R$ is a $\GGL$ ring and $S$ is Gorenstein.
\end{cor}

\begin{proof}
Since $\fka$ is an Ulrich ideal of $R$ contained in $\fkc$, $\fka$ have the form of (\ref{ttttt4.22}) by Lemma \ref{a4.17.5}. By Propositions \ref{a4.19}, we have $T$ is a Gorenstein ring and $T/R$ is $R/\fka$-free. On the other hand, by Proposition \ref{a4.15}, it follows that $\fka=\fkc$. 
\end{proof}

We cannot replace $\fkc$ with $\tr_R(\rmK_R)$ in the assertion of Corollary \ref{a4.20}. 

\begin{ex}\label{2gene}
Let $k$ be a field and $n\ge 3$. Let $k[[t]]$ be the formal power series ring over $k$. Set $R=k[[t^{2n}, t^{2n+1}, \dots, t^{3n}]]$ and $I=(t^{2n}, t^{3n})$.
Then, the following hold true.
\begin{enumerate}[\rm(1)] 
\item $\rmK_R\cong (t^{2n}, t^{2n+1}, \dots, t^{3n-2})$. Hence, a fractional canonical ideal $K$ is $\langle 1, t, \dots, t^{n-2}\rangle$ and $S=R[K]=k[[t]]$.
\item $\tr_R(\rmK_R)=(t^{2n}, t^{2n+1}, \dots, t^{3n})$. That is, $R$ is a $\NGL$ ring. (See after Corollary \ref{a7.1.1} to recall the definition of $\NGL$ rings.)
\item $\fkc=t^{4n}k[[t]]=(t^{4n}, t^{4n+1}, \dots, t^{5n-1})$.
\item $I$ is an Ulrich ideal of $R$ contained in $\tr_R(\rmK_R)$, but not contained in $\fkc$.
\end{enumerate}
\end{ex}


\section{Numerical semigroup rings}\label{sec5}
Numerical semigroup rings are one of the main classes of rings satisfying Setup \ref{setting1}. Thus, in this section, we examine the $\GGL$ property in numerical semigroup rings. We note that if $(R, \fkm)$ is a numerical semigroup ring, then the algebra $B = \m : \m$ is also a numerical semigroup ring. This follows that Theorem \ref{a4.24} provides the structure theorem for $\GGL$ numerical semigroup rings with minimal multiplicity.
We also have numerous examples of $\GGL$ rings via numerical semigroup rings.
First, we recall the notations for numerical semigroup rings. One can consult \cite{RG} for more details about numerical semigroup rings.

\begin{setting}\label{setting of numerical semigroup}
Let $0 < a_1, a_2, \ldots, a_\ell \in \Bbb Z~(\ell >0)$ be positive integers such that $\mathrm{GCD}~(a_1, a_2, \ldots, a_\ell)=1$. We call 
\[
H = \left<a_1, a_2, \ldots, a_\ell\right>=\left\{\sum_{i=1}^\ell c_ia_i \mid 0 \le c_i \in \Bbb Z~\text{for~all}~1 \le i \le \ell \right\}
\]
the {\it numerical semigroup generated by the numbers $\{a_i\}_{1 \le i \le \ell}$}. Let $V = k[[t]]$ be the formal power series ring over a field $k$. We set
$$R = k[[H]] = k[[t^{a_1}, t^{a_2}, \ldots, t^{a_\ell}]]$$
in $V$, which we call the {\it semigroup ring of $H$ over $k$}. The ring $R$ is a one-dimensional Cohen-Macaulay local domain with $\overline{R} = V$ and $\m = (t^{a_1},t^{a_2}, \ldots, t^{a_\ell} )$. 
Let 
$$
\rmc(H) = \min \{n \in \Bbb Z \mid m \in H~\text{for~all}~m \in \Bbb Z~\text{such~that~}m \ge n\}
$$
be the {\it conductor of $H$} and set $\rmf(H) = \rmc(H) -1$. Hence, $\rmf(H) = \max ~(\Bbb Z \setminus H)$, and we call $\rmf(H)$ the {\it Frobenius number of $H$}. Let 
\[
\mathrm{PF}(H) = \{n \in \Bbb Z \setminus H \mid n + a_i \in H~\text{for~all}~1 \le  i \le \ell\}
\]
denote the set of pseudo-Frobenius numbers of $H$. Therefore, $\rmf(H)$ equals the $\rma$-invariant of the graded $k$-algebra $k[t^{a_1}, t^{a_2}, \ldots, t^{a_\ell}]$, and $\sharp \mathrm{PF}(H) = \rmr(R)$  (\cite[Example (2.1.9) and Definition (3.1.4)]{GW}).  We set  $f = \rmf(H)$ and write $\mathrm{PF}(H)=\{c_1 < c_2 < \dots < c_r=f\}$, where $r=\rmr(R)$. Set 
\begin{align}\label{formK}
K = \sum_{c \in \mathrm{PF}(H)}Rt^{f-c}
\end{align}
in $V$. Then,  $K$ is a fractional ideal of $R$ such that $R \subseteq K \subseteq \overline{R}$ and 
$$
K \cong \rmK_R = \sum_{c \in \mathrm{PF}(H)}Rt^{-c}
$$
as an $R$-module (\cite[Example (2.1.9)]{GW}). 
\end{setting}

Hence, the numerical semigroup ring $R$ satisfies Setup \ref{setting1}. Set $S=R[K]$ and $\fkc=R:~S$.
Note that 
\[
\fkm:\fkm=R:\fkm=R+\sum_{c\in \mathrm{PF}(H)} R{\cdot}t^c
\] 
if $R\subsetneq V$.

\begin{prop}\label{a4.30}
Suppose that $R$ is not a Gorenstein ring. Set $r=\rmr(R)$.
The following are equivalent:
\begin{enumerate}[{\rm (1)}]
\item $R$ is a $\GGL$ ring.
\item $R/\fkc$ is a Gorenstein ring and let $t^b$ be the element in $[\fkc :_R\fkm] \setminus \fkc$. Then,  $f+b=c_i+c_{r-i}$ for all $1\le i \le r-1$.
\end{enumerate}
\end{prop}

\begin{proof}
(1) $\Rightarrow$ (2): By Theorem \ref{a4.7}, $R/\fkc$ is a Gorenstein ring and the $R$-linear map $(R/\fkc)^{\oplus (r-1)}\to K/R$, where $\mathbf{e}_i \mapsto \overline{t^{f-c_i}}$, is an isomorphism. Here, $\mathbf{e}_1, \mathbf{e}_2, \ldots, \mathbf{e}_{r-1}$ denotes a free basis of $(R/\fkc)^{\oplus (r-1)}$. Let $t^b$ be the element in $[\fkc :_R\fkm] \setminus \fkc$. Then,  $t^{f-c_i+b} \not\in R$ and $\fkm \cdot t^{f-c_i+b} \subseteq R$. Hence, $f-c_i+b \in \mathrm{PF}(H)$ for all $1\le i \le r-1$. Thus, we obtain $\mathrm{PF}(H)= \{ f-c_{r-1}+b, f-c_{r-2}+b, \dots, f-c_1+b, f\}$, and $f-c_{r-i}+b=c_i$ for all $1\le i \le r-1$.

(2) $\Rightarrow$ (1): Consider the surjection $\varphi:Y=(R/\fkc)^{\oplus (r-1)}\to K/R$, where $\mathbf{e}_i \mapsto \overline{t^{f-c_i}}$ for all $1\le i \le r-1$. Then,  
\begin{center}
$R$ is a $\GGL$ ring $\Leftrightarrow$ $\varphi$ is injective $\Leftrightarrow$ $(0):_{\Ker {\varphi}} \fkm=(0)$ $\Leftrightarrow$ $\varphi \mid_{\Soc Y}$ is injective,
\end{center}
where $\Soc Y$ denotes the socle $(0):_Y\fkm$ of $Y$ and  $\varphi \mid_{\Soc Y}$ is the restriction $\Soc Y \to Y \xrightarrow{\varphi} K/R$.
Note that $\Soc Y= ([\fkc:_R\fkm]/\fkc)^{\oplus (r-1)}=\langle t^b \mathbf{e}_1,t^b \mathbf{e}_2, \dots, t^b \mathbf{e}_r \rangle$ by the condition of (2). Hence, if $\varphi \mid_{\Soc Y}$ is not injective, then we have $t^b \mathbf{e}_i \mapsto 0$ for some $1\le i \le r-1$.
This follows that $c_{r-i}=f-c_i+b \in H$, which is a contradiction.
\end{proof}

\begin{cor}\label{a4.31}
Suppose that $R$ is not a Gorenstein ring. Let $n_1, n_2, \dots , n_{\ell}$ be integers and set $J=(t^{(n_1+1)a_1}, t^{(n_2+1)a_2}, \dots , t^{(n_{\ell}+1)a_{\ell}})$. Suppose that 
\begin{enumerate}[{\rm (1)}]
\item $J\subseteq \fkc$ and
\item $f+b=c_i+c_{r-i}$ for all $1\le i \le r-1$, where $b=\sum_{j=1}^{\ell} n_j a_j$.
\end{enumerate}
Then,  $R$ is a $\GGL$ ring, $J=\fkc$, and $\fkc:_R\fkm=\fkc+(t^b)$.
\end{cor}

\begin{proof}
Since $b+f-c_i=c_{r-i}$ for all $1\le i \le r-1$, we have $t^b\not\in R:t^{f-c_i}$. By noting that $t^{f-c_i}\in K\subseteq S$ (see \eqref{formK}), $\fkc\subseteq R:t^{f-c_i}$. In particular, $t^b\not\in J$ by the hypothesis (1). Meanwhile, $t^b\in J:_R\fkm$ by the forms of $b$ and $J$. We show that $R/J$ is a Gorenstein ring. Indeed, let $t^h\in [J:_R\fkm] \setminus J$ and write $h=m_1 a_1+m_2 a_2+\dots+m_{\ell} a_{\ell}$ for some non-negative integers $m_1, m_2, \dots, m_{\ell}$. 
Since $t^h\not\in J$,  $m_j\le n_j$ for all $1\le j\le\ell$. Hence, $t^b=t^h\cdot t^{b-h}$ and $b-h\in H$. 
It follows that $h=b$ because both $t^b$ and $t^h$ are in $[J:_R \fkm]\setminus J$. Hence, $R/J$ is a Gorenstein ring. By Proposition \ref{a4.30}, we have only to show that $J=\fkc$. Indeed, if $J\subsetneq \fkc$, then $t^b\in J:_R\fkm \subseteq \fkc$. This is a contradiction because $t^b\not\in R:t^{f-c_i}\supseteq  \fkc$. Thus, $J=\fkc$ and $R$ is a $\GGL$ ring. 
\end{proof}

\begin{cor}\label{a4.32}
Assume $e\le h_1\le h_2\le \cdots\le h_{e-1}$ and let $H=\left<e, h_1, h_2, \ldots, h_{e-1}\right>$. Set $R=k[[H]]$.
Suppose that $e=\rme(R)=\rmv(R)\ge3$.
Then,  the following conditions are equivalent:
\begin{enumerate}[{\rm (1)}]
\item $R$ is a $\GGL$ ring.
\item There exists an integer $n_0\ge 0$ such that $\fkc=(t^{(n_0+1)e}, t^{h_1}, t^{h_2}, \dots , t^{h_{e-1}})$ and \\
$h_{e-1}+(n_0+1)e=h_i+h_{e-1-i}$ for all $1\le i\le e-2$.
\end{enumerate}
When this is the case, $n_0=\ell_R(R/\fkc)-1$.
\end{cor}

\begin{proof}
Note that $\mathrm{PF}(H)=\{h_1-e, \dots, h_{e-1}-e\}$, that is, $c_i=h_i-e$ for all $1\le i\le e-1$ because $R$ has minimal multiplicity. 

(2) $\Rightarrow$ (1):  The assertion follows from Corollary \ref{a4.31} and the fact that $c_i=h_i-e$ for all $1\le i\le e-1$.

(1) $\Rightarrow$ (2): Since $K/R= \langle \ol{t^{f-c_1}}, \ol{t^{f-c_2}}, \dots, \ol{t^{f-c_{e-2}}} \rangle$ is $R/\fkc$-free, by considering the annihilators of the free basis $\ol{t^{f-c_i}}$ of $K/R$, we have $\fkc=R:t^{f-c_i}$ for all $1\le i \le e-2$. It follows that $t^{h_i}\in \fkc$ for all $1\le i \le e-2$ since $t^{h_i}=t^{c_i+e}\in R:t^{f-c_i}$.
Furthermore, because $f-c_1+h_{e-1}=f-(h_1-e)+h_{e-1}=f+e+(h_{e-1}-h_1)>f$, we have $t^{h_{e-1}}\in R:t^{f-c_1}=\fkc$. 
Hence, $(t^{h_1}, t^{h_2}, \dots , t^{h_{e-1}})\subseteq \fkc$. It follows that $\fkc=(t^{(n_0+1)e}, t^{h_1}, t^{h_2}, \dots , t^{h_{e-1}})$, where $n_0=\ell_R (R/\fkc)-1$.
Therefore, by applying Proposition \ref{a4.30} as $b=n_0 e$, we have $f+n_0 e=c_i+c_{e-1-i}$ for all $1 \le i \le e-2$. Since $c_i=h_i-e$ for all $1\le i\le e-1$, we have the assertion.
\end{proof}

Let $0 < e \in H$ and set $\alpha_i = \min \{h \in H \mid h \equiv i ~\mod ~e\}$ for each $0 \le i \le e-1$. Then,  the set 
\[
Ap_e(H) = \{\alpha_i \mid 0 \le i \le e-1\} = \{h \in H \mid h - e \not\in H\}
\]  
is called the {\it Ap\'ery set} of $H$ mod $e$ (\cite[before Lemma 2.4]{RG}). With the notation, the following result is known.

\begin{fact}{\rm (\cite[Theorem 6.9]{CGKM})}
Let $H$ be a numerical semigroup and assume that $H$ is symmetric, that is, $R=k[[H]]$ is a Gorenstein ring. Choose an element $0<e\in H$ and consider 
$Ap_e(H)=\{0< h_1< h_2< \cdots < h_{e-1} \}.$
Set 
\begin{center}
$H_n=\left<e, h_1+ne, h_2+ne, \ldots, h_{e-1}+ne\right>$ and $R_n=k[[H_n]]$
\end{center}
for all $n>0$. Let $K_n$ denote a fractional canonical ideal of $R_n$ such that $R_n \subseteq K_n \subseteq \ol{R_n}$.
Then,  we have the following:
\begin{enumerate}[{\rm (1)}]
\item $R_n$ is a $\GGL$ ring,
\item $\rmv(R_n)=\rme(R_n)$, and
\item $\ell_{R_n} (R_n/\fkc_n)=n,$
\end{enumerate}
where $\fkc_n=R_n:R_n[K_n]$.
\end{fact}

From Theorem \ref{a4.24}, we obtain the reverse of \cite[Theorem 6.9]{CGKM} as follows.

\begin{thm}\label{a4.29}
Let $e\le h_1\le h_2\le \cdots\le h_{e-1}$ be positive integers such that $\mathrm{GCD}~(e, h_1, h_2, \ldots, h_{e-1})=1$. Set $H=\left<e, h_1, h_2, \ldots, h_{e-1}\right>$ and $R=k[[H]]$.
Assume that $R$ is a $\GGL$ ring with $e=\rme(R)=\rmv(R)\ge3$. Set $n=\ell_R(R/\fkc)>0$ and $H'=\left<e, h_1-ne, h_2-ne, \ldots, h_{e-1}-ne\right>$. Then,  the following assertions hold true:
\begin{enumerate}[{\rm (1)}]
\item $H'$ is symmetric, that is, $R'=k[[H']]$ is a Gorenstein ring and 
\item $Ap_e(H') =\{0, h_1-ne, h_2-ne, \ldots, h_{e-1}-ne \}$.
\end{enumerate}

Hence, $R$ is reconstructed by $Ap_e(H')$ and $n>0$.
\end{thm}

\begin{proof}
Suppose that $n>1$, that is, $R$ is not an $\AGL$ ring. Set $R_0=R$ and $\fkm_0=(t^{e}, t^{h_1}, t^{h_2}, \ldots, t^{h_{e-1}})$. Set 
\begin{center}
$R_j=\fkm_{j-1}:\fkm_{j-1}$ and $\fkm_{j}=\fkm_{j-1} S \cap R_{j}$
\end{center}
recursively for $0<j\le n$. Note that $R_j=k[[H_j]]$ and the maximal ideal of $R_j$ is $\fkm_j$ for all $0<j\le n$, where $H_j=\left<e, h_1-je, h_2-je, \ldots, h_{e-1}-je\right>$ by Proposition \ref{a4.22} and Theorem \ref{a4.24}. Furthermore, we have the following:
\begin{enumerate}[{\rm (i)}]
\item $R_j$ is a $\GGL$ ring but not a Gorenstein ring,
\item $\ell_{R_j}(R_j/\fkc_j)=n-j$,
\item $\rmv(R_j)=\rme(R_j)=e$, and
\item $Ap_e(H_j) =\{0, h_1-je, h_2-je, \ldots, h_{e-1}-je \}$
\end{enumerate}
for all $0<j<n$, where $\fkc_j=R_j:S$.
Hence, we may assume that $n=1$. This is the case where $R$ is an $\AGL$ ring but not a Gorenstein ring. Hence, the assertion follows from \cite[Theorem 5.1]{GMP}.
\end{proof}

\begin{ex}\label{ex57}
Let $n\ge 1$ be an integer and set $R_n=k[[t^4, t^{4n+1}, t^{4n+2}, t^{4n+3}, t^{4n+7}]]$. Then $R_1=k[[t^4, t^5, t^6]]$ is a Gorenstein ring not of minimal multiplicity. On the other hand, $R_n$ is a $\GGL$ ring having minimal multiplicity and $\ell_{R_n}(R_n/\tr_{R_n}(\rmK_{R_n}))=n-1$ for all $n\ge 2$. In particular, $R_2$ is a non-Gorenstein $\AGL$ ring and $R_n$ is not an $\AGL$ ring for $n\ge 3$.
\end{ex}


We continue to explore $\GGL$ numerical semigroup rings. With the notation of Setup \ref{setting of numerical semigroup}, suppose that $\ell = 3$ and set $T = k[t^{a_1}, t^{a_2}, t^{a_3}]$ in the polynomial ring $k[t]$. Let $P = k[X_1, X_2, X_3]$ be the polynomial ring over $k$. We consider  $P$ to be a $\Bbb Z$-graded ring such that $P_0 = k$ and $\deg X_i = a_i$ for  $i=1, 2, 3$. Let 
$$\varphi : P =k[X_1, X_2, X_3] \to T= k[t^{a_1}, t^{a_2}, t^{a_3}]$$
denote the homomorphism of graded $k$-algebras defined by $\varphi (X_i) = t^{a_i}$ for each $i=1, 2, 3$. For conciseness, we write  $X = X_1$, $Y=X_2$, and $Z = X_3$.
If $T$ is not a Gorenstein ring, then, it is known that 
\[
\Ker \varphi = I_2\left(\begin{smallmatrix}
X^\alpha&Y^\beta&Z^\gamma\\
Y^{\beta'}&Z^{\gamma'}&X^{\alpha'}\\
\end{smallmatrix}
\right)
\] 
for some integers $\alpha, \beta, \gamma, \alpha', \beta', \gamma' > 0$ (\cite{H}). 

Here, we use a result from \cite[Section 4]{GMP}.
Let $\Delta_1 = Z^{\gamma+\gamma'}-X^{\alpha'}Y^{\beta}$, $\Delta_2=X^{\alpha+\alpha'}-Y^{\beta'}Z^{\gamma}$, and $\Delta_3=Y^{\beta+\beta'}-X^{\alpha}Z^{\gamma'}$. Then,  $\Ker \varphi  = (\Delta_1,\Delta_2, \Delta_3)$, and from the theorem of Hilbert--Burch (\cite[Theorem 20.15]{E}), the graded ring $T$ possesses a graded minimal $P$-free resolution of the  form
$$
0\longrightarrow \begin{matrix} P(-m )\\ \oplus\\ P(-n)\end{matrix} \overset{\left( \begin{smallmatrix}
X^{\alpha} & Y^{\beta'}\\
Y^{\beta} & Z^{\gamma'}\\
Z^{\gamma} & X^{\alpha'}
\end{smallmatrix} \right)}{\longrightarrow  } \begin{matrix} P(-d_1)\\ \oplus\\ P(-d_2)\\ \oplus\\ P(-d_3)\end{matrix} \overset{\left(\Delta_1~-\Delta_2~\Delta_3\right)}{\longrightarrow } P\overset{\varphi }{\longrightarrow } T \longrightarrow 0,
$$
where $d_1 = \deg\Delta_1 = a_3(\gamma + \gamma')$, $d_2 = \deg\Delta_2 = a_1(\alpha + \alpha')$, $d_3 = \deg\Delta_3 = a_2(\beta + \beta')$, $m = a_1\alpha + d_1 = a_2\beta + d_2 = a_3\gamma + d_3$, and $n = a_1\alpha' + d_3 = a_2\beta' + d_1 = a_3\gamma' + d_2$. Therefore,
\begin{align}\label{threegene}
n - m   = a_2\beta'-a_1\alpha  = a_3\gamma' -a_2\beta=a_1\alpha'-a_3\gamma.
\end{align}
Let $\rmK_P = P(-d)$ denote the graded canonical module of $P$, where $d=a_1 + a_2 + a_3$. Then, by applying the $\rmK_P$-dual to the above resolution, we obtain the  minimal presentation
\begin{equation}\label{herzogmatrix}
\begin{matrix} P(d_1-d)\\ \oplus\\ P(d_2-d)\\ \oplus\\ P(d_3-d)\end{matrix}\overset{\left( \begin{smallmatrix} X^{\alpha} & Y^{\beta} & Z^{\gamma} \\
Y^{\beta'} & Z^{\gamma'} & X^{\alpha'}
\end{smallmatrix} \right)}{\longrightarrow }\begin{matrix} P(m  -d)\\ \oplus\\ P(n-d)\end{matrix}\overset{\varepsilon }{\longrightarrow }\rmK_{T}\longrightarrow 0 
\end{equation}
of the graded canonical module $\rmK_{T} = \Ext_P^2(T,\rmK_P)$ of $T$. Therefore, because $\rmK_T = \sum_{c \in \mathrm{PF}(H)}Tt^{-c}$ (\cite[Example (2.1.9)]{GW}), we have $\ell_k([\rmK_T]_i) \le 1$ for all $i \in \Bbb Z$; hence, $m \ne n$. After the permutation of $a_2$ and $a_3$ if necessary, we may assume without loss of generality that $n > m$. Then,  the presentation (\ref{herzogmatrix}) shows that $\mathrm{PF}(H) = \{m-d, n-d\}$ and $f = n-d$.

We set $a = n-m$. Note that $R\cong \widehat{T_M}$, where $M= (t^{a_i}\mid i = 1,2,3)$ denotes the graded maximal ideal of $T$ and $\widehat{T_M}$ denotes the $MT_M$-adic completion of $T_M$.
Hence,  $a > 0$, $f = a + (m-d)$, and $K = R + Rt^a$ in the sense of Setup \ref{setting1}. With this notation, we have the following.

\begin{thm}\label{a4.33}
Suppose that $H$ is $3$-generated. Assume that $R=k[[H]]$ is not a Gorenstein ring and $a>0$. Then,  the following conditions are equivalent:
\begin{enumerate}[{\rm (1)}]
\item $R$ is a $\GGL$ ring.
\item $3a\in H$.
\item $\alpha \le \alpha'$, $\beta \le \beta'$, and $\gamma \le \gamma'$.
\end{enumerate}
When this is the case, $\fkc=(t^{\alpha a_1}, t^{\beta a_2}, t^{\gamma a_3})$ and $\ell_R (R/\fkc) = \alpha \beta \gamma$. 
\end{thm}

\begin{proof}

(1) $\Rightarrow$ (3): By \eqref{herzogmatrix}, we obtain $\rmK_{T}/P\xi = (Tt^{d-m} + Tt^{d-n})/Tt^{d-n}\cong P/(X^{\alpha},  Y^{\beta},  Z^{\gamma})$, where $\xi=\varepsilon \left(\begin{smallmatrix}0 \\1 \end{smallmatrix} \right)$. Note that we have the isomorphism  $\widehat{(\rmK_{T}/P\xi)_M} \cong K/R\cong R/\fkc$ because $n>m$ and $R$ is a $\GGL$ ring of $\rmr(R)=2$. Hence, we have $\fkc=(X^{\alpha},  Y^{\beta},  Z^{\gamma})R=(t^{\alpha a_1}, t^{\beta a_2}, t^{\gamma a_3})$. Meanwhile, by applying the functor $*\otimes_P P/(X^{\alpha},  Y^{\beta},  Z^{\gamma})$ to \eqref{herzogmatrix}, we obtain the exact sequence
{\small
\begin{equation*}
\begin{matrix} P/(X^{\alpha},  Y^{\beta},  Z^{\gamma})(d_1-d)\\ \oplus\\ P/(X^{\alpha},  Y^{\beta},  Z^{\gamma})(d_2-d)\\ \oplus\\ P/(X^{\alpha},  Y^{\beta},  Z^{\gamma})(d_3-d)\end{matrix}\overset{\left( \begin{smallmatrix} \overline{X^{\alpha}} & \overline{Y^{\beta}} & \overline{Z^{\gamma}} \\
\overline{Y^{\beta'}} & \overline{Z^{\gamma'}} & \overline{X^{\alpha'}}
\end{smallmatrix} \right)}{\longrightarrow }\begin{matrix} P/(X^{\alpha},  Y^{\beta},  Z^{\gamma})(m  -d)\\ \oplus\\ P/(X^{\alpha},  Y^{\beta},  Z^{\gamma})(n-d)\end{matrix}\overset{\varepsilon }{\longrightarrow }\rmK_{T}/(X^{\alpha},  Y^{\beta},  Z^{\gamma})\rmK_{T} \longrightarrow 0 
\end{equation*}
}
as $P/(X^{\alpha},  Y^{\beta},  Z^{\gamma})$-modules, where $\overline{*}$ denotes the image of an element $*\in P$ in $P/(X^{\alpha},  Y^{\beta},  Z^{\gamma})$. Hence, $\left( \begin{smallmatrix} \overline{X^{\alpha}} & \overline{Y^{\beta}} & \overline{Z^{\gamma}} \\
\overline{Y^{\beta'}} & \overline{Z^{\gamma'}} & \overline{X^{\alpha'}}
\end{smallmatrix} \right)$ is a zero matrix because $K/\fkc K\cong (R/\fkc)^{\oplus 2}$; hence, $(Y^{\beta'},   Z^{\gamma'}, X^{\alpha'})\subseteq (X^{\alpha},  Y^{\beta},  Z^{\gamma})$, that is, we obtain $\alpha \le \alpha'$, $\beta \le \beta'$, and $\gamma \le \gamma'$.

(3) $\Rightarrow$ (2): This follows from the equality $3a=(a_2\beta'-a_1\alpha) + (a_3\gamma' -a_2\beta)+(a_1\alpha'-a_3\gamma)=a_1(\alpha'-\alpha)+a_2(\beta'-\beta)+a_3(\gamma'-\gamma) \in H$ by (\ref{threegene}).

(2) $\Rightarrow$ (1): Since $K=R+Rt^a$, we have $K^2=K^3$. Hence, $R$ is a $\GGL$ ring by Corollary \ref{a4.9}.
\end{proof}

When $H$ is $3$-generated and the multiplicity $\rme(R)=\min\{a_1, a_2, a_3\}$ of $R$ is at most $5$, we have the following structure theorem of non-$\AGL$ $\GGL$ rings. 

\begin{corollary}\label{a4.34}
Let $\ell = 3$. Then,  the following assertions are true.
\begin{enumerate}[{\rm (1)}]
\item If $\min\{a_1, a_2, a_3\} = 3$, then $R$ is  a $\GGL$ ring.

\item Suppose that $\min\{a_1, a_2, a_3\} = 4$. Then,  the following conditions are equivalent:
   \begin{enumerate}[{\rm (a)}]
   \item $R$ is a $\GGL$ ring but not an $\AGL$ ring and
   \item $H= \left<4, 3\alpha+2\alpha', \alpha+2\alpha' \right>$ for some $\alpha' \ge \alpha \ge3$ such that $\alpha \not\equiv 0~\mathrm{mod}~2$.
   \end{enumerate}

\item Suppose that $\min\{a_1, a_2, a_3\} = 5$. Then,  the following conditions are equivalent:
   \begin{enumerate}[{\rm (a)}]
   \item $R$ is a $\GGL$ ring but not an $\AGL$ ring and
   \item $\mathrm{(i)}$ $H= \left<5, 2\alpha+\alpha', \alpha+3\alpha' \right>$ for some $\alpha' \ge \alpha \ge2$ such that $2\alpha+\alpha' \not\equiv 0~\mathrm{mod}~5$ or  $\mathrm{(ii)}$  $H= \left<5, 4\alpha+3\alpha', \alpha+2\alpha' \right>$ for some $\alpha' \ge \alpha \ge2$ such that $\alpha+2\alpha' \not\equiv 0~\mathrm{mod}~5$.
   \end{enumerate}
\end{enumerate}
\end{corollary}

\begin{proof}
(1): This follows from Corollary \ref{a4.10}.

Suppose that $R$ is a $\GGL$ ring but not an $\AGL$ ring. Then, by \cite[Corollary 4.2]{GMP}, after a suitable permutation of $a_1, a_2, a_3$, we may assume that $\alpha' \ge \alpha \ge 2$. 
Note that 
\begin{align}\label{4.37.1}
\begin{split}
a_1&= \beta\gamma+\beta' \gamma' + \beta'\gamma\\
a_2 &=\alpha\gamma+\alpha \gamma'+\alpha' \gamma' \\
a_3 &= \alpha' \beta' +\alpha' \beta+\alpha \beta
\end{split}
\end{align}
because $a_1 =\ell_R(R/t^{a_1}R) =\ell_P (P/[(X)+\Ker\varphi])= \ell_k(k[Y,Z]/(Y^{\beta' + 1}, Y^{\beta'}Z, Z^{\gamma'+1})=\beta\gamma+\beta' \gamma' + \beta'\gamma$. The other two equalities follow in a similar way. By noting that $\alpha' \ge \alpha \ge 2$, we have $a_2\ge 6$ and $a_3\ge 6$.
Therefore, if $\rme(R) \le 5$, then we obtain that 
\[
\rme(R) = a_1=  \beta\gamma+\beta' \gamma' + \beta'\gamma= (\beta + \beta')\gamma +\beta' \gamma' \le 5.
\] 
Note that by Theorem \ref{a4.33}, $\beta \le \beta'$ and $\gamma \le \gamma'$ since $R$ is a $\GGL$ ring and $\alpha' \ge \alpha \ge 2$. It follows that $\beta=\gamma=1$.

(2) (a) $\Rightarrow$ (b): Since $a_1= \beta' \gamma' + \beta' +1 = 4$, we have $\beta'= 1$ and $\gamma'=2$. Hence, $a_2=3\alpha+2\alpha'$ and $a_3=\alpha+2\alpha'$. Note  that $\alpha \not\equiv 0~\mathrm{mod}~2$ because $\mathrm{GCD}~(a_1, a_2, a_3)=1$. 

(3) (a) $\Rightarrow$ (b): Since $a_1= \beta' \gamma' + \beta' +1 = 5$, there are $2$ cases, say, ``$\beta'=2$ and $\gamma'=1$'' or ``$\beta'=1$ and $\gamma'=3$''.
For the former, we obtain  $H= \left<5, 2\alpha+\alpha', \alpha+3\alpha' \right>$ for some $\alpha' \ge \alpha \ge2$ such that $2\alpha+\alpha' \not\equiv 0~\mathrm{mod}~5$ because $\mathrm{GCD}~(5, 2\alpha+\alpha', \alpha+3\alpha')=1$. 
For the latter, we obtain $H= \left<5, 4\alpha+3\alpha', \alpha+2\alpha' \right>$ for some $\alpha' \ge \alpha \ge2$ such that $\alpha+2\alpha' \not\equiv 0~\mathrm{mod}~5$ because $\mathrm{GCD}~(5, 4\alpha+3\alpha', \alpha+2\alpha')=1$.  

(2) (b) $\Rightarrow$ (a): Since we have $\Ker \varphi = I_2\left(\begin{smallmatrix}
X^{\alpha}&Y&Z\\
Y&Z^2&X^{\alpha'}\\
\end{smallmatrix}
\right)$, $R$ is a $\GGL$ ring by Theorem \ref{a4.33}. However, $R$ is not an $\AGL$ ring by \cite[Corollary 4.2]{GMP}. 

(3) (b) $\Rightarrow$ (a): For the case (i) we have $\Ker \varphi = I_2\left(\begin{smallmatrix}
X^{\alpha}&Y&Z\\
Y^2&Z&X^{\alpha'}\\
\end{smallmatrix}
\right)$, and for the case (ii) we have $\Ker \varphi = I_2\left(\begin{smallmatrix}
X^{\alpha}&Y&Z\\
Y&Z^3&X^{\alpha'}\\
\end{smallmatrix}
\right)$. Therefore, $R$ is a $\GGL$ ring but not an $\AGL$ ring in both cases.
\end{proof}

We can also characterize $3$-generated numerical semigroup rings $R$ such that $\tr_R(\rmK_R)$ is an Ulrich ideal.

\begin{thm}\label{a7.6.5}
Suppose that $H$ is $3$-generated. Assume that $R=k[[H]]$ is not a Gorenstein ring and $a>0$. 
Then, the following conditions are equivalent:
\begin{enumerate}[{\rm (1)}]
\item $\tr_R (\rmK_R)$ is an Ulrich ideal.
\item Two of the three pairs $(\alpha, \alpha')$, $(\beta, \beta')$, and $(\gamma, \gamma')$ are equal.
\end{enumerate}
When this is the case, after renumbering, we have the equalities:
\begin{center}
$a_1=3\beta \gamma$, $a_2=\gamma (2\alpha + \alpha')$, and $a_3=\beta (2\alpha' + \alpha)$.
\end{center}
\end{thm}

\begin{proof}
(2) $\Rightarrow$ (1): After a suitable permutation of $a_1$, $a_2$, and $a_3$ if necessary, we may assume that $\alpha < \alpha', \beta = \beta', \gamma = \gamma'$. Note that we have 
\begin{center}
$a_1=3\beta \gamma$, $a_2=\gamma (2\alpha + \alpha')$, and $a_3=\beta (2\alpha' + \alpha)$
\end{center}
by \eqref{4.37.1}. $R$ is a $\GGL$ ring and $\fkc=(t^{\alpha a_1}, t^{\beta a_2}, t^{\gamma a_3})=(t^{3 \alpha \beta \gamma}, t^{(2\alpha + \alpha')\beta \gamma}, t^{(2\alpha' + \alpha)\beta\gamma})$ by Theorem \ref{a4.33} and the above equalities. It is straightforward to verify that $\fkc^2=t^{a_1{\cdot}\alpha}\fkc$.  Hence, $S$ is a Gorenstein ring by Lemma \ref{a4.18}. Therefore, $\tr_R (\rmK_R)$ is an Ulrich ideal by Theorem \ref{a7.5}.

(1) $\Rightarrow$ (2): We have
\begin{center}
$\alpha \le \alpha'$, $\beta \le \beta'$, $\gamma \le \gamma'$, and $\fkc=(t^{\alpha a_1}, t^{\beta a_2}, t^{\gamma a_3})$
\end{center}
by Theorems \ref{a7.5} and \ref{a4.33}. 
We may assume that $(t^{a_1{\cdot}\alpha})$ is a reduction of $\fkc$ after renumbering. Then, 
$$
\left<1, t^a, t^{2a}\right>=K^2=S=\fkc:\fkc=\frac{\fkc}{t^{a_1\alpha}}=\left<1, t^{a_2\beta-a_1\alpha}, t^{a_3\gamma-a_1\alpha} \right>
$$
by Lemma \ref{a4.6} and the fact that $\fkc$ is stable. Therefore, we obtain that 

$\begin{cases}
a_2\beta-a_1\alpha=2a \\
a_3\gamma-a_1\alpha=a
\end{cases}$
or 
$\begin{cases}
a_2\beta-a_1\alpha=a \\
a_3\gamma-a_1\alpha=2a
\end{cases}$.\\
Assume the former case. Then, because $a=n-m=a_1 \alpha'-a_3 \gamma$ by (\ref{threegene}), we have 
$$
2a_3 \gamma=(a+a_1\alpha)+(a_1\alpha'-a)=a_1(\alpha+\alpha')=a_2 \beta' + a_3 \gamma,
$$
where the last equality follows from $X^{\alpha+\alpha'}-Y^{\beta'}Z^{\gamma} \in \Ker \varphi$. Therefore, $a_2 \beta'=a_3 \gamma$; hence, $Y^{\beta'}-Z^{\gamma}\in \Ker \varphi$. This contradicts the construction of $\beta$ and $\gamma$ because $R$ is not a Gorenstein ring (\cite{H}).
Hence, the latter case holds. Then, since $a_2\beta-a_1\alpha=a=a_2\beta'-a_1\alpha$ by (\ref{threegene}), we obtain that $\beta=\beta'$. Similarly, we have $\gamma=\gamma'$ since $a_3\gamma-a_1\alpha=2a=(a_2\beta'-a_1\alpha) +(a_3\gamma' -a_2\beta)$.
\end{proof}


\section{Minimal free resolutions of $\GGL$ rings}\label{section5}
Throughout this section, let $(R, \fkm)$ be a Cohen-Macaulay local ring possessing the canonical module $\rmK_R$. Set $d=\dim R>0$ and $r=\rmr(R)$.
In this section, we consider the following condition.

\begin{condition}
There exists an exact sequence 
\begin{align}\label{ttt501}
0 \to R \to \rmK_R \to \bigoplus_{i=2}^r R/\fka_i \to 0
\end{align}
of $R$-modules, where $\fka_i$ is an ideal of $R$ for all $2\le i \le r$.
\end{condition}

\begin{rem}\label{a5.1}
\begin{enumerate}[{\rm (1)}]
\item Suppose that $R$ has an exact sequence (\ref{ttt501}). Then,  $R/\fka_i$ is a Gorenstein local ring of dimension $d-1$ for all $2\le i \le r$.
\item If $R$ satisfies one of the following assumptions, then  Condition holds true:
\begin{enumerate}[{\rm (i)}] 
\item $R$ is a generically Gorenstein ring with $r=2$,
\item $R$ is a $\GGL$ ring with $d=1$, or
\item $R$ is a $2$-$\AGL$ ring in the sense of \cite[Theorem 1.4]{CGKM}. 
\end{enumerate}
\end{enumerate}
\end{rem}

\begin{proof}
(1): By Fact \ref{a3.1}, $R/\fka_i$ is a Cohen-Macaulay local ring of dimension $d-1$. By applying the functor $\Hom_R (*, \rmK_R)$ to (\ref{ttt501}), we obtain the isomorphism $\bigoplus_{i=2}^r R/\fka_i \cong \bigoplus_{i=2}^r \Ext_R^1 (R/\fka_i, \rmK_R)$. Thus, $\Ext_R^1 (R/\fka_i, \rmK_R)$ are cyclic for all $2 \le i \le r$. Since  $\Ext_R^1 (R/\fka_i, \rmK_R) \cong \rmK_{R/\fka_i}$ (\cite[Theorem 3.3.7(b)]{BH}), it follows that $R/\fka_i$ is a Gorenstein ring.

(2) (i): By assumption, there exists a canonical ideal generated by two non-zerodivisors.

(ii): This follows from Proposition \ref{a4.1}. 

(iii): This follows from \cite[Proposition 3.3 (4)]{CGKM}.
\end{proof}

Recall that $R$ is a {\it semi-Gorenstein local ring} if there exists an exact sequence 
\[
0 \to R \to \rmK_R \to \bigoplus_{i=2}^r R/\fka_i \to 0
\]
of $R$-modules, where $R/\fka_i$ are either $(0)$ or Ulrich $R$-modules with respect to $\fkm$ for all $2\le i \le r$ (\cite[Definition 7.1]{GTT}). Note that semi-Gorenstein rings are $\AGL$ rings by definition. In \cite[Theorem 7.8]{GTT}, it is given that for a Cohen-Macaulay local ring $R$, a characterization of a semi-Gorenstein property for $R$ in terms of the minimal $S$-free resolution of $R$, where $S$ is a regular local ring with a surjective ring homomorphism $S\to R$. 
The following is a generalization of \cite[Theorem 7.8]{GTT}. Note that the outline of the proof is similar to that of \cite[Theorem 7.8]{GTT}. However, we include a proof since this is a wide generalization as seen in Remark \ref{a5.1}, and the details of the proof are different.

\begin{thm}\label{a5.2}
Let $(S, \fkn)$ be a Gorenstein local ring and $I, \fka_2, \fka_3, \dots, \fka_r$ be ideals of $S$. Suppose that $R = S/I$ is a Cohen-Macaulay ring but not a Gorenstein ring. Assume that the projective dimension of $R$ over $S$ is finite. 
Then,  the following conditions are equivalent:
\begin{enumerate}[{\rm (1)}]
\item There exists an exact sequence 
$$0\to R \to \rm{K}_R \to \bigoplus _{i=2}^r S/\fka_i \to 0 $$
of $S$-modules.
\item There exist a minimal $S$-free resolution 
\begin{align}\label{star}
0\to S^{\oplus r}  \xrightarrow{\mathbb{M}} S^{\oplus q} \to \dots \to S \to R \to 0
\end{align}
 of $R$ and a non-negative integer $m$ such that
\begin{equation*}
{}^t\!{\mathbb{M}}=\left(
\begin{array}{cccc|c}
y_{21}~y_{22} ~\cdots ~y_{2u_2} & y_{31}~y_{32} ~\cdots ~y_{3u_3} & \cdots & y_{r1}~y_{r2} ~\cdots ~y_{ru_r} & z_1~z_2 ~\cdots ~z_m \\ \hline
x_{21}~x_{22} ~\cdots ~x_{2u_2} &  & & & \\
 & x_{31}~x_{32}~ \cdots~ x_{3u_3} & & \mbox{\huge{0}} & \\
 & & \ddots  & & \mbox{\huge{0}} \\
\mbox{\huge{0}} & & & x_{r1}~x_{r2}~ \cdots ~x_{ru_r} & \\
\end{array}
\right) ,
\end{equation*}
where $\mu_S (\fka_i)=u_i$, $\fka_i=(x_{i1}, x_{i2}, \dots, x_{iu_i})$, and $\dim S/\fka_i=\dim R-1$ for all $2\le i \le r$.
\end{enumerate}

Furthermore, if $x_{i1}, x_{i2}, \dots, x_{iu_i}$ is an $S$-regular sequence for all $2\le i \le r$, then we have the equality
\begin{align}\label{tousiki}
I=\sum_{i=2}^r  I_2\left(\begin{matrix}
y_{i1}&y_{i2} & \cdots &y_{iu_i}\\
x_{i1}&x_{i2} & \cdots &x_{iu_i}\\
\end{matrix}
\right)  + (z_1, z_2, \dots, z_m).
\end{align}
\end{thm}

\begin{proof}
(1) $\Rightarrow$ (2): Choose the exact sequence
$$0\to R \xrightarrow{\varphi} \rm{K}_R \to \bigoplus _{i=2}^r S/\fka_i \to 0 $$
of $S$-modules and set $f_1=\varphi(1)$. $\rm{K}_R/f_1 S \cong \bigoplus _{i=2}^r S/\fka_i$. Choose elements $f_2, \ldots , f_r \in \rm{K}_R$ such that $\overline{f_i}$ corresponds to $(0, \ldots , 0, 1_{S/\fka_i}, 0, \ldots , 0)$, where $\overline{f_i}$ denotes the image of $f_i$ in $\rm{K}_R/f_1 S$.
Then,  we have a surjective homomorphism
$$\psi:S^{\oplus r} \to \rm{K}_R,~\mathbf{e}_i \mapsto f_i,$$
where $\{ \mathbf{e}_i \}_{1\le i \le r}$ denotes the standard basis of $S^{\oplus r}$.
Set $L=\Ker {\psi}$ and $u_i=\mu_S(\fka_i)$. Choose $x_{i j} \in S$ such that $\fka_i=(x_{i1}, x_{i2}, \ldots, x_{i u_{i}})$ for all $2\le i \le r$ and $1\le j \le u_i$. 
We explore a minimal basis of $L$. Since $\rm{K}_R/f_1 S \cong \bigoplus _{i=2}^r S/\fka_i$, we have $x_{i j} f_{i} \in f_1 S$ for all $2\le i \le r$. That is,  $x_{i j} f_i + y_{i j} f_1=0$ for some $y_{i j} \in S$. Therefore, we obtain that $x_{i j} \mathbf{e}_i + y_{i j} \mathbf{e}_1 \in L$. We set $\mathbf{a}_{i j}=x_{i j} \mathbf{e}_i + y_{i j} \mathbf{e}_1$ for all $2\le i \le r$ and $1\le j \le u_i$.
Let $\mathbf{a}\in L$ and write $\mathbf{a}=\sum_{i=1}^r b_i \mathbf{e}_i$ with $b_i \in S$. Then,  $b_i \in \fka_i$ for all $2\le i \le r$ because $\sum_{i=2}^r b_i \cdot (0, \ldots , 0, 1_{S/\fka_i}, 0, \ldots , 0)=0$ in $\bigoplus _{i=2}^r S/\fka_i (\cong \rm{K}_R/f_1 S)$. Write $b_i=\sum_{j=1}^{u_i} c_{i j} x_{i j}$ with $c_{i j} \in S$. Then,  we have
\begin{align*}
\mathbf{a} &=b_1 \mathbf{e}_1 + \sum_{i=2}^r b_i \mathbf{e}_i =b_1 \mathbf{e}_1 + \sum_{i=2}^r \sum_{j=1}^{u_i} c_{i j} x_{i j} \mathbf{e}_i\\
 &=b_1 \mathbf{e}_1 + \sum_{i=2}^r \sum_{j=1}^{u_i} c_{i j} (\mathbf{a}_{i j}-y_{i j}\mathbf{e}_1).\\
\end{align*}
Hence, we have $\mathbf{a}-\sum_{i=2}^r \sum_{j=1}^{u_i} c_{i j}  \mathbf{a}_{i j} \in L\cap S\mathbf{e}_1$, and  $L$ is generated by $\{ \mathbf{a}_{i j} \}_{2\le i \le r,~1\le j \le u_i} \cup \{ z_k \mathbf{e}_1\}_{1\le k \le m}$ for some integer $m\ge 0$ and $z_k \in S$ for $1\le k \le m$. Thus, by setting $q=\sum_{i=2}^r u_i + m$, we have an $S$-free resolution
\begin{align}\label{presentation}
S^{\oplus q} \xrightarrow{{}^t\!{\mathbb{M}}} S^{\oplus r} \xrightarrow{\psi} \rm{K}_R \to 0
\end{align}
of $\rm{K_R}$, where the matrix ${}^t\!{\mathbb{M}}$ has the required form. We need to show that we may assume that (\ref{presentation}) is minimal, that is, $\{ \mathbf{a}_{i j} \}_{2\le i \le r,~1\le j \le u_i}$ is part of a minimal system of generators of $L$.
Since $\{ \mathbf{a}_{i j} \}_{2\le i \le r,~1\le j \le u_i} \cup \{ z_k \mathbf{e}_1\}_{1\le k \le m}$ generates $L$, we can choose a minimal system of generators of $L$ from them. Assume that $ \mathbf{a}_{i j} $ is not part of a minimal system of generators of $L$ for some $2\le i \le r$ and $1\le j \le u_i$. Then, we can construct an $S$-free resolution
$$
S^{\oplus q} \xrightarrow{{}^t\!{\mathbb{M}}'} S^{\oplus r} \xrightarrow{\psi} \rm{K}_R \to 0
$$
of $\rm{K_R}$, where ${}^t\!{\mathbb{M}}'$ is a matrix such that the column corresponding to $\mathbf{a}_{i j}$ is omitted from ${}^t\!{\mathbb{M}}$.
Hence, we have
\begin{eqnarray*}
\bigoplus _{i=2}^r S/\fka_i &\cong& \rm{K}_R/f_1 S \cong S^{\oplus r}/(\Im {}^t\!{\mathbb{M}}' + S\mathbf{e}_1) \\
&\cong& S/\fka_2 \oplus \dots \oplus S/\fka_{i-1} \oplus S/\fka' \oplus S/\fka_{i+1} \oplus \dots \oplus S/\fka_{r},
\end{eqnarray*}
where $\fka'=(x_{i1}, x_{i2}, \ldots, x_{i\,j-1}, x_{i\,j+1}, \ldots, x_{i u_{i}})$.
This contradicts $u_i=\mu_S(\fka_i)$.
Hence, we may assume that  (\ref{presentation}) is minimal. 
Then,  the $S$-module $\rmK_R$ possesses a minimal free resolution
$$
0 \to S \to \cdots \to S^{\oplus q} \xrightarrow{{}^t\!{\mathbb{M}}} S^{\oplus r} \to \rmK_R \to 0
$$
with $q=\sum_{i=2}^{r}u_i + m$. Therefore, by applying the $S$-dual, the assertion (2) holds.

(2) $\Rightarrow$ (1): By applying the $S$-dual to (\ref{star}), we have the exact sequence (\ref{presentation}). Set $f_i=\psi(\mathbf{e}_i)$ for all $1 \le i \le r$, where $\{\mathbf{e}_i\}_{1\le i\le r}$ denotes the standard basis of $S^{\oplus r}$. 
We then have
$$
\rm{K}_R/f_1 S \cong S^{\oplus r}/(\Im {}^t\!{\mathbb{M}} + S\mathbf{e}_1) \cong \bigoplus _{i=2}^r S/\fka_i.
$$
Hence, we have an exact sequence
$$
R \xrightarrow{\varphi} \rm{K}_R \to \bigoplus _{i=2}^r S/\fka_i \to 0
$$
of $R$-modules, where $\varphi (1)=f_1$.
Since $\dim S/\fka_i=\dim R-1$ for all $2\le i \le r$, $\varphi$ is injective by \cite[Lemma 3.1(1)]{GTT}.

Now, we prove the equality \eqref{tousiki}. Suppose that  $x_{i1}, x_{i2}, \dots, x_{iu_i}$ is an $S$-regular sequence for all $2\le i \le r$. Note that for $a\in S$, we obtain the equivalences
\begin{align}\label{final}
\begin{split}
a\in I \quad \Leftrightarrow & \quad af_1 =0 \quad \Leftrightarrow  \quad a\mathbf{e}_1\in L\\
  \Leftrightarrow & \quad a\mathbf{e}_1=\sum_{2\le i \le r,~1\le j \le u_i} c_{ij} \mathbf{a}_{ij} + \sum_{k=1}^m d_k z_k \mathbf{e}_1 \quad \text{ for some $c_{ij}, d_k\in S$}\\
  \Leftrightarrow & \quad \text{there exist $c_{ij}$ and $d_k\in S$ such that}\\
& \text{$a=\sum_{2\le i \le r,~1\le j \le u_i} c_{ij} y_{ij} + \sum_{k=1}^m d_k z_k \quad \text{ and } \quad 0=\sum_{1\le j \le u_i} c_{ij} x_{ij}$ for $2\le i \le r$},
\end{split}
\end{align}
where the first equivalence follows from the fact that $\varphi: R \to \rmK_R$ is injective and the  second equivalence follows from  $L=\Ker \psi$. The third equivalence follows from 
\[
L=\sum_{2\le i \le r,~1\le j \le u_i} S \mathbf{a}_{ij} + \sum_{k=1}^m S z_k \mathbf{e}_1. 
\]
The fourth equivalence follows from $\mathbf{a}_{i j}=x_{i j} \mathbf{e}_i + y_{i j} \mathbf{e}_1$.

$(\supseteq)$: By \eqref{final}, we obtain that $I\supseteq (z_1, z_2, \dots, z_m).$ For all $2\le i\le r$ and $1\le \alpha<\beta \le u_{i}$, we obtain that 
\[
(x_{i \beta}y_{i \alpha} - x_{i \alpha} y_{i \beta})\mathbf{e}_1 =x_{i\beta}(\mathbf{a}_{i\alpha}-x_{i\alpha} \mathbf{e}_i) - x_{i\alpha} (\mathbf{a}_{i\beta}-x_{i\beta}\mathbf{e}_i)=x_{i\beta}\mathbf{a}_{i\alpha} - x_{i\alpha} \mathbf{a}_{i\beta}\in L;
\]
hence, $x_{i \beta}y_{i \alpha} - x_{i \alpha} y_{i \beta}\in I$ by \eqref{final}.

$(\subseteq)$: Let $a\in I$, and choose $c_{ij}$, $d_{k}\in S$ such that $a=\sum_{2\le i \le r,~1\le j \le u_i} c_{ij} y_{ij} + \sum_{k=1}^m d_k z_k$ and $0=\sum_{1\le j \le u_i} c_{ij} x_{ij}$ for all $2\le i \le r$. It is sufficient to show that for any $2\le i \le r$, we have 
\[
\sum_{1\le j \le u_i} c_{ij} y_{ij} \in I_2\left(\begin{matrix}
y_{i1}&y_{i2} & \cdots &y_{iu_i}\\
x_{i1}&x_{i2} & \cdots &x_{iu_i}
\end{matrix}
\right). 
\]
Let 
\begin{align*}
K_2(\mathbf{x}) \xrightarrow{\partial_2} K_1(\mathbf{x}) \xrightarrow{\partial_1} K_0(\mathbf{x}) \quad \text{and} \quad
K_2(\mathbf{y}) \xrightarrow{\partial_2'} K_1(\mathbf{y}) \xrightarrow{\partial_1'} K_0(\mathbf{y})
\end{align*}
be the parts of Koszul complexes of the sequences $\mathbf{x}=x_{i1},x_{i2}, \dots, x_{iu_i}$ and $\mathbf{y}=y_{i1},y_{i2}, \dots, y_{iu_i}$, respectively. Let $T_1, \dots,T_{u_i}$ be the standard basis of $K_1(\mathbf{x})$ and $K_1(\mathbf{y})$. 
Then, because $\partial_1(\sum_{1\le j \le u_i} c_{ij} T_{j})=\sum_{1\le j \le u_i} c_{ij} x_{ij}=0$, we obtain $\sum_{1\le j \le u_i} c_{ij} y_{ij}\in \partial_1'(\Ker \partial_1)$. Meanwhile, because the sequence $\mathbf{x}=x_{i1},x_{i2}, \dots, x_{iu_i}$ is an $S$-regular sequence, we have $\Im \partial_2=\Ker \partial_1$. 
Hence,  $\sum_{1\le j \le u_i} c_{ij} y_{ij}\in \Im (\partial_1'\circ \partial_2)$. It follows that $\sum_{1\le j \le u_i} c_{ij} y_{ij}\in I_2\left(\begin{matrix}
y_{i1}&y_{i2} & \cdots &y_{iu_i}\\
x_{i1}&x_{i2} & \cdots &x_{iu_i}\\
\end{matrix}
\right)$  
because $\Im (\partial_1'\circ \partial_2)$ is generated by $\partial_1'\circ \partial_2(T_\alpha T_\beta)=x_{i \beta}y_{i \alpha} - x_{i \alpha} y_{i \beta}$ for all $1\le \alpha < \beta \le u_i$.
\end{proof}

\begin{cor}\label{a5.3}
Let $(S, \fkn)$ be a regular local ring and $I, \fka_2, \fka_3, \dots, \fka_r$ be ideals of $S$. Suppose that $R \cong S/I$ and $R$ is a Cohen-Macaulay ring but not a Gorenstein ring. 
Assume that there exists an exact sequence 
$$0\to R \to \rm{K}_R \to \bigoplus _{i=2}^r S/\fka_i \to 0 $$
of $R$-modules. 
If $S/\fka_i$ is a complete intersection for $2\le i\le r$, then 
$$
I=\sum_{i=2}^r  I_2\left(\begin{matrix}
y_{i1}&y_{i2} & \cdots &y_{iu_i}\\
x_{i1}&x_{i2} & \cdots &x_{iu_i}\\
\end{matrix}
\right)  + (z_1, z_2, \dots, z_m)
$$
for some $y_{i1}, y_{i2}, \dots, y_{iu_i} \in S$ and $z_1, z_2, \dots, z_m \in S$,
where $\fka_i=(x_{i1}, x_{i2}, \dots, x_{iu_i})$ and $\mu_S (\fka_i)=u_i$. 

\end{cor}

\begin{proof}
This follows from the fact that $\fka_i$ is generated by an $S$-regular sequence (see \cite[Theorem 2.3.3]{BH}). 
\end{proof}

\begin{cor}\label{a5.4}
With the notation of {\rm Theorem $\ref{a5.2}$}, suppose that {\rm Condition (1)} holds true. Set $n=\dim S-\dim R$. We then have the following:
\begin{enumerate}[{\rm (1)}]
\item If $n=2$, then $r=2$, $q=3$, and $m=0$.
\item Suppose that $S$ is a regular local ring, $I\subseteq \fkn^2$, and $R$ has minimal multiplicity. Then,  $r=n$, $q=n^2-1$, and $m=0$.
\end{enumerate}

\end{cor}

\begin{proof}
Note that $q=\sum_{i=2}^r u_i + m \ge (r-1)(n+1) + m$ because $n+1=\height_S \fka_i \le u_i$.

(1): Since the minimal S-free resolution of $R$ has the form $0\to S^{\oplus r}  \to S^{\oplus (r+1)} \to S \to R \to 0$, we have $q=r+1\ge (r-1)(2+1) + m$. Since $R$ is not a Gorenstein ring, we obtain $r=2$, $q=3$, and $m=0$.

(2): Set $e=\rm{e}(R)$. By \cite[Theorem 1.(iii)]{S2}, we have $n=e-1$, $r=e-1$, and $q=(e-2)e$. Therefore, we obtain $q=(e-2)e \ge (e-2)e + m$; hence, $m=0$.
\end{proof}

By applying Theorem \ref{a5.2} to $\GGL$ rings of dimension one, we have a characterization of one-dimensional $\GGL$ rings in terms of the minimal free resolution.

\begin{cor}\label{a5.5}
Let $(S, \fkn)$ be a Gorenstein local ring. Let $(R, \fkm)$ be a one-dimensional Cohen-Macaulay local ring, but not a Gorenstein ring. Let $\varphi:S \to R$ be a surjective ring homomorphism and suppose that the projective dimension of $R$ over $S$ is finite. Let $\fka$ be an ideal of $S$ such that $\fka \supseteq \Ker \varphi$ and set $n=\mu_S(\fka)$ and $\fka=(x_1, x_2, \dots, x_n)$.
Then,  the following conditions are equivalent:
\begin{enumerate}[{\rm (1)}]
\item $R$ is a $\GGL$ ring with respect to $\fka R$.
\item There exists a minimal $S$-free resolution 
$$0\to S^{\oplus r} \xrightarrow{\mathbb{M}} S^{\oplus q} \to \dots \to S \to R \to 0$$ of $R$ such that
\begin{equation*}
{}^t\!{\mathbb{M}}=\left(
\begin{array}{cccc|c}
&**&&&* \\ \hline
x_{1}~x_{2} ~\cdots ~x_{n} &  & & & \\
 & x_{1}~x_{2}~ \cdots~ x_{n} & & \mbox{\huge{0}} & \\
 & & \ddots  & & \mbox{\huge{0}} \\
\mbox{\huge{0}} & & & x_{1}~x_{2}~ \cdots ~x_{n} & \\
\end{array}
\right) ,
\end{equation*}
where all entries of $**$ and $*$ are in $\fka$.
\end{enumerate}
\end{cor}

\begin{proof}
(1) $\Rightarrow$ (2): By Proposition \ref{a4.1}, there exists an exact sequence 
$$0\to R \to \rm{K}_R \to (S/\fka)^{\oplus (r-1)} \to 0.$$
Hence, we can apply Theorem \ref{a5.2} and obtain the minimal $S$-free resolution of $R$ as stated in the assertion (2). We have only to show that all entries of $**$ and $*$ are in $\fka$. Indeed, by applying the $S$-dual to the minimal $S$-free resolution of $R$, we have $S^{\oplus q} \xrightarrow{{}^t\!{\mathbb{M}}} S^{\oplus r} \to \rm{K}_R \to 0$. By applying the functor $S/\fka \otimes_S -$ to the exact sequence, we obtain the following exact sequence:
$$(S/\fka)^{\oplus q} \xrightarrow{\overline{{}^t\!{\mathbb{M}}}} (S/\fka)^{\oplus r} \to \rm{K}_R/\fka \rm{K}_R \to 0.$$
Therefore, $\overline{{}^t\!{\mathbb{M}}}$ is forced to be a zero matrix because $(S/\fka)^{\oplus r} \cong (R/\fka R)^{\oplus r} \cong \rm{K}_R/\fka \rm{K}_R$.

(2) $\Rightarrow$ (1): By Theorem \ref{a5.2}, there exists an exact sequence 
$$0\to R \xrightarrow{\varphi} \rm{K}_R \to (S/\fka)^{\oplus (r-1)} \to 0.$$
Similarly to the above proof, we obtain the exact sequence
$$(S/\fka)^{\oplus q} \xrightarrow{\overline{{}^t\!{\mathbb{M}}}} (S/\fka)^{\oplus r} \to \rm{K}_R/\fka \rm{K}_R \to 0.$$
Since $\overline{{}^t\!{\mathbb{M}}}$ is a zero matrix, we have $(S/\fka)^{\oplus r} \cong \rm{K}_R/\fka \rm{K}_R$. Hence, $R$ is a $\GGL$ ring with respect to $\fka R$.
\end{proof}

\section{$\GGL$ rings of higher dimension}\label{sec7}

In this section we explore $\GGL$ rings of higher dimension. First, we prove the reverse implication of Proposition \ref{a3.6} (Theorem \ref{aaaa72}). 
Throughout this section, let $(R, \fkm)$ be a Cohen-Macaulay local ring possessing the canonical module $\rmK_R$. Let $\fka$ be an $\fkm$-primary ideal of $R$. Set $d=\dim R > 0$.

\begin{lem}\label{a4.11}
If $R$ is a $\GGL$ ring with respect to $\fka$, then $R/\fka$ is a Gorenstein ring.
\end{lem}

\begin{proof}
By Proposition \ref{a3.6}, we may assume that $R/\fkm$ is infinite. 
The case where $d=1$ is proved in Theorem \ref{a4.7}.
Let $d>1$ and assume that our assertion holds true for $d-1$. 
Then,  we can choose $f \in \fka$ such that $R/(f)$ is a $\GGL$ ring with respect to $\fka/(f)$ by Theorem \ref{a3.5}.
Hence, $R/\fka$ is a Gorenstein ring by the induction hypothesis.
\end{proof}

\begin{cor}\label{xxxc72}
$R$ is a $\GGL$ ring with respect to some (any) parameter ideal of $R$ if and only if $R$ is Gorenstein.
\end{cor}

\begin{proof}
If $R$ is Gorenstein of positive dimension, then $R$ is a $\GGL$ ring with respect to any parameter ideal of $R$ by Remark \ref{rem3.4}. If $R$ is a $\GGL$ ring with respect to $\fka$ and $\fka$ is a parameter ideal of $R$, then $R/\fka$ is Gorenstein by Lemma \ref{a4.11}. It follows that $R$ is also Gorenstein.
\end{proof}

\begin{thm}\label{aaaa72}
Let $R$ and $S$ be local rings of dimension $>0$. Let $\varphi : R \to S$ be a flat local homomorphism. Suppose that $S/\fkm S$ is a Cohen-Macaulay local ring of dimension  $\ell$. Let $J\subseteq S$ be a parameter ideal in $S/\fkm S$. Consider the following two conditions:
\begin{enumerate}[{\rm (1)}]
\item $R$ is a $\GGL$ ring with respect to $\fka$, and $S/\fkm S$ is a Gorenstein ring.
\item $S$ is a $\GGL$ ring with respect to $\fka S+ J$.
\end{enumerate}
Then,  $\rm(1)$ $\Rightarrow$ $\rm(2)$ holds true.  $\rm(2)$ $\Rightarrow$ $\rm(1)$ also holds true if $R/\fkm$ is infinite.
\end{thm}

\begin{proof}
$\rm(1)$ $\Rightarrow$ $\rm(2)$: This is proved in Proposition \ref{a3.6}.

 $\rm(2)$ $\Rightarrow$ $\rm(1)$: If $S$ is Gorenstein, then $R$ and $S/\fkm S$ are also Gorenstein by \cite[Proposition 1.2.16(b)]{BH}. Thus, we may assume that $S$ is not a Gorenstein ring. Set $\fkb=\fka S + J$.
$S/\fkb$ is a Gorenstein ring by Lemma \ref{a4.11}. Since $R\to S \to S/J$ is a flat local homomorphism, so is $R/\fka \to S/(\fka S+J)$. Hence, $R/\fka$ and $S/\fkm S$ are Gorenstein rings (cf. \cite[Proposition 1.2.16]{BH}).

We prove that $R$ is a $\GGL$ ring with respect to $\fka$. Assume that $\ell > 0$. Since $d+\ell =\dim S$ and $d>0$, $\dim S >1$. 
Choose a defining exact sequence 
\[
0 \to S \xrightarrow{\psi} \rmK_S \to D \to 0 
\]
of the $\GGL$ property of $S$. Set $A=S/((0):_S D)$. Then, we have $\dim A=\dim S-1=d+\ell -1$ by Fact \ref{a3.1}(2). Note that $\dim S/\fka S=\dim R/\fka +\ell=\ell$.
Hence, we can choose $g_1, g_2, \dots, g_\ell, h_1, h_2, \dots, h_{d-1} \in \fkb$ that satisfy the following conditions:
\begin{enumerate}[{\rm (i)}]
\item $g_1, g_2, \dots, g_\ell$ is a system of parameter of $S/\fka S$.
\item $(g_1, g_2, \dots, g_\ell, h_1, h_2, \dots, h_{d-1})A$ is a minimal reduction of $\fkb A$.
\end{enumerate}
Set $\fkq=(g_1, g_2, \dots, g_\ell, h_1, h_2, \dots, h_{d-1})$. Then 
\[
\ell_S (D/\fkq D)=\ell_A (D/\fkq D)=\rme_{\fkq A}^0 (D)=\rme_{\fkb A}^0 (D)=\rme_{\fkb}^0 (D)=\ell_S (D/\fkb D),
\]
where the last equality follows from the assumption that $D$ is an Ulrich $S$-module with respect to $\fkb$. It follows that $\fkq D=\fkb D$ and thus, $\fkq^n D=\fkb^n D$ for all $n>0$. In particular, $\fkq D \cap \fkb^{n+1} D = \fkq \fkb^{n} D$ for all $n>0$. Therefore, $g_1t, g_2t, \dots, g_\ell t, h_1t, h_2t, \dots, h_{d-1}t$ is a $\mbox{\rm gr}_{\fkb}(D)$-regular sequence by \cite{VV} or \cite[Theorem 1.1]{RV}.
It follows that $g_1, g_2, \dots, g_\ell, h_1, h_2, \dots, h_{d-1}$ is a superficial sequence for $D$ with respect to $\fkb$. On the other hand, by the condition (ii),  $g_1, g_2, \dots, g_\ell$ is a system of parameters of $S/\fkm S$. Hence, by passing to $R \to S \to S/(g_1, g_2, \dots, g_\ell)$, without loss of generality, we may assume that $\ell=0$ by Theorem \ref{a3.5}.

If $d>1$, then we can choose an element $f\in \fka$ such that $f$ is $R$-regular and $\varphi(f)$ is superficial for $D$ with respect to $\fka S$. Hence, by passing to $R/fR \to S/fS$, we may assume that $d=1$. The case where $\dim R=\dim S=1$ has already been proved in Corollary \ref{a4.8}. Therefore, we have proved the assertion.
\end{proof}

\begin{prop}\label{a4.12}
Suppose that $R/\fkm$ is infinite. If $\rme(R)\le3$, then $R$ is a $\GGL$ ring.
\end{prop}

\begin{proof}
We may assume that $R$ is not a Gorenstein ring.
The case where $d=1$ is proved in Corollary \ref{a4.10}.
Let $d>1$ and assume that our assertion holds true for $d-1$. 
Choose $f \in \fkm$ such that $f$ is superficial for $R$ with respect to $\fkm$.
Then, because $\rme(R/(f))=\rme(R) \le 3$, $R/(f)$ is a $\GGL$ ring by the induction hypothesis. Since $f$ is a non-zerodivisor of $R$, $R$ is also a $\GGL$ ring by Theorem \ref{a3.5}.
\end{proof}

\begin{lem}\label{above}
Suppose that $R$ is a non-Gorenstein generically Gorenstein ring. Let $R^{\oplus q} \xrightarrow{\psi} R^{\oplus r} \xrightarrow{\varepsilon} \rmK_R \to 0$ be a part of a minimal $R$-free resolution of $\rmK_R$. Then $\tr_R(\rmK_R)\subseteq I_1(\psi)$, where $I_1(\psi)$ denotes the ideal of $R$ generated by the entries of the matrix representing $\psi$.
\end{lem}

\begin{proof}
Since $R$ is generically Gorenstein, there exists a canonical ideal $\omega$ of $R$. Choose a free basis $\mathbf{e}_1, \mathbf{e}_2, \dots, \mathbf{e}_r$ of $R^{\oplus r}$ and non-zerodivisors $f_1, f_2, \dots, f_r$ such that $\omega=(f_1, f_2, \dots, f_r)$. We may assume that $\varepsilon(\mathbf{e}_i)=f_i$ for all $1\le i\le r$.  
Then we obtain that for all $x\in R:\omega$ and for all $1\le i, j\le r$, $(xf_j)\mathbf{e}_i - (xf_i)\mathbf{e}_j\in \Ker \varepsilon=\Im \psi$. It follows that $xf_i\in I_1(\psi)$. Thus, $\tr_R(\rmK_R)=(R:\omega)\omega\subseteq I_1(\psi)$.
\end{proof}

\begin{prop}
Suppose that $R$ is a non-Gorenstein $\GGL$ ring with respect to $\fka$. Then $\tr_R(\rmK_R)\subseteq \fka$.
\end{prop}

\begin{proof}
Let $R^{\oplus q} \xrightarrow{\psi} R^{\oplus r} \to \rmK_R \to 0$ be a part of a minimal $R$-free resolution of $\rmK_R$. 
Since $\rmK_R/\fka \rmK_R$ is $R/\fka$-free by the definition of $\GGL$ rings, by applying the functor $R/\fka\otimes_R *$ to the exact sequence, we obtain that $I_1(\psi)\subseteq \fka$. It follows that  $\tr_R(\rmK_R)\subseteq \fka$ by Lemma \ref{above}.
\end{proof}

\begin{cor}\label{c77}
If $R$ is a $\GGL$ ring and a $\NGL$ ring, then $R$ is an $\AGL$ ring.
\end{cor}

The reverse does not hold in general. Indeed, $\AGL$ rings of dimension $\ge 2$ are not necessarily to be Gorenstein on the punctured spectrum. On the other hand, in the statement of Corollary \ref{c77}, we can replace the $\GGL$ property with the condition that ring has minimal multiplicity (Corollary \ref{a7.9}).

\begin{thm}\label{a7.7}
Suppose that the residue field $R/\fkm$ is infinite. If $\tr_R(\rmK_R)$ is an Ulrich ideal of $R$ with $\mu_R(\tr_R(\rmK_R))>d+1$, then $R$ is a $\GGL$ ring with respect to $\tr_R(\rmK_R)$. 
In particular, $R/\tr_R(\rmK_R)$ is a Gorenstein ring and $\mu_R(\tr_R(\rmK_R))=d+\rmr(R)$.
\end{thm}

\begin{proof}
We prove this by induction on $d$. Set $J=\tr_R(\rmK_R)$. The case where $d=1$ is proved in Theorems \ref{a7.5} and \ref{a4.7}.
Let $d>1$ and assume that our assertion holds true for $d-1$. Since $R/\fkm$ is infinite, we can choose a parameter ideal $Q=(f=f_1, f_2, \ldots, f_d)$ as a minimal reduction of $J$. 
Set $\ol{*}=R/(f)\otimes_R *$. Then,  the evaluation map $t_{\rmK_R}:\Hom_R (\rmK_R, R)\otimes_R \rmK_R \to R$ (recall \eqref{eq291}) induces 
$$\ol{t_{\rmK_R}}:\ol{\Hom_R (\rmK_R, R)}\otimes_{\ol{R}} \rmK_{\ol{R}} \to \ol{R}.$$
Hence, $J\ol{R}=\Im \ol{t_{\rmK_R}} \subseteq \sum_{f\in \Hom_{\ol{R}} (\rmK_{\ol{R}}, \ol{R})} \Im f =\tr_{\ol{R}} (\rmK_{\ol{R}})$. Meanwhile, $J \ol{R}$ is an Ulrich ideal of $\ol{R}$ by \cite[Lemma 3.3]{GOTWY} and $\mu_{\ol{R}}(J \ol{R})>(d-1)-1$.
By Proposition \ref{a7.2}, $J \ol{R} = \tr_{\ol{R}}(\rmK_{\ol{R}})$. Therefore, $\ol{R}$ is a $\GGL$ ring with respect to $\tr_{\ol{R}}(\rmK_{\ol{R}})$. It follows that $R$ is also a $\GGL$ ring with respect to $\tr_R(\rmK_R)$ by Theorem \ref{a3.5}.
In particular, $R/\tr_R(\rmK_R)$ is a Gorenstein ring by Lemma \ref{a4.11}, and $\mu_R(\tr_R(\rmK_R))=d+\rmr(R)$ by Fact \ref{a7.0.2}(1).
\end{proof}

\begin{cor}\label{a7.9}
Suppose that $R$ is a $\NGL$ ring having minimal multiplicity.
Then,  $R$ is an $\AGL$ ring if the residue field $R/\fkm$ is infinite.

\end{cor}

\begin{proof}
We may assume that $R$ is not a Gorenstein ring. Then, $\tr_R (\rmK_R)=\fkm$ is an Ulrich ideal, and $\rmv(R)=\rme(R)=d+\rmr(R)>d+1$ by Fact \ref{a4.23}. Therefore, $R$ is a $\GGL$ ring with respect to $\fkm$, that is, $R$ is an $\AGL$ ring.
\end{proof}

In the next section we construct examples of rings satisfying the condition of Theorem \ref{a7.7}, see Example \ref{a7.10}.


\section{Determinantal $\GGL$ rings of certain specific matrices}\label{section6}
In this section, we describe a construction of $\GGL$ rings arising from certain determinantal ideals.
Throughout this section, let $(S, \fkn)$ be a local ring of dimension $d>0$. For an ideal $I$ of $S$ and a finitely generated $S$-module $M$, $\grade (I, M)$ denotes the {\it grade of $M$ in $I$} in the sense of \cite[Definition 1.2.6]{BH}, that is, the length of the maximal $M$-regular sequence in $I$.

\begin{fact}{\rm (\cite[Proposition 1.2.9]{BH})}\label{a6.1}
Let $I$ be an ideal of $S$, and let $x\in \fkn$ be a non-zerodivisor of $S$. Then, 
$\grade (I, S)-1 \le \grade ([I+(x)]/(x), S/(x)) \le \grade (I, S).$
\end{fact}

\begin{lemma}\label{a6.2}
For an integer $n>0$, the following assertions hold true:
\begin{enumerate}[{\rm (1)}]
\item Let $0<\alpha_1, \alpha_2, \dots , \alpha_n, \beta_1, \beta_2, \dots , \beta_n$ be positive integers and $x_1, x_2, \dots, x_n\in \fkn$ be an $S$-regular sequence. Set $$I=I_2\left(\begin{matrix}
x_{1}^{\alpha_1}&x_{2}^{\alpha_2} & \cdots &x_{n-1}^{\alpha_{n-1}}&x_{n}^{\alpha_n}\\
x_{2}^{\beta_2}&x_{3}^{\beta_3} & \cdots &x_{n}^{\beta_n}&x_1^{\beta_1}\\
\end{matrix}
\right) .$$
Then,  $\grade (I, S)=n-1$.
\item Let $x_1, x_2, \dots, x_n, y_1, y_2, \dots, y_n\in \fkn$ be an $S$-regular sequence. Set $$J=I_2\left(\begin{matrix}
x_{1}&x_{2} & \cdots &x_{n}\\
y_{1}&y_{2} & \cdots &y_{n}\\
\end{matrix}
\right) .$$
Then,  $\grade (J, S)=n-1$.
\end{enumerate}

\end{lemma}

\begin{proof}
Note that the heights of $I$ and $J$ are at most $n-1$ (\cite[(2.1) Theorem]{BV2}).  Thus, we have only to show that the reverse inequality holds.

(1): Note that $I+(x_1)=I' +(x_1)$, where $I'=I_2\left(\begin{matrix}
0&x_{2}^{\alpha_2} & \cdots &x_{n-1}^{\alpha_{n-1}}&x_{n}^{\alpha_n}\\
x_{2}^{\beta_2}&x_{3}^{\beta_3} & \cdots &x_{n}^{\beta_n}&0\\
\end{matrix}
\right)$. 
We show that $x_2, x_3, \ldots, x_n \in \sqrt{I'}$ by induction on $2 \le i \le n$. The cases where $i=2$ and $i=n$ are clear. For the case of $2 < i < n$, suppose the assertion holds true for $i-1$. Then, because $\mathrm{det}\left(\begin{matrix}
x_{i-1}^{\alpha_{i-1}}&x_{i}^{\alpha_i}\\
x_{i}^{\beta_i}&x_{i+1}^{\beta_{i+1}}\\
\end{matrix}\right) \in I'$, we have $x_i^{\alpha_i + \beta_i}\in \sqrt{I'}$.
Hence, $n-1 \le \grade ([\sqrt{I'}+(x_1)]/(x_1), S/(x_1)) = \grade ([I'+(x_1)]/(x_1), S/(x_1)) \le \grade (I, S)$ by Fact \ref{a6.1} and \cite[Proposition 1.2.10(b)]{BH}.

(2): Set $y_0=y_n$ and $Q=(x_i-y_{i-1} \mid 1 \le i \le n) + (x_1)$.
Then,  we have 
\[
[J+Q]/Q=\left[I_2\left(\begin{matrix}
0&x_{2} & \cdots &x_{n-1}&x_{n}\\
x_{2}&x_{3} & \cdots &x_{n}&0\\
\end{matrix}
\right)+Q\right]/Q=[(x_2, x_3, \ldots, x_n)^2 + Q]/Q.
\]
Hence, $n-1 \le \grade ([J+Q]/Q, S/Q) \le \grade (J, S)$.
\end{proof}

\begin{thm}\label{a6.3}
Let $S$ be a Gorenstein local ring, and let $n$ be a positive integer with $3\le n\le d$.
Assume that $x_1, x_2, \dots, x_d$ is a system of parameters of $S$.
Set $Q=(x_1, x_2, \dots, x_n)$ and take elements $y_1, y_2, \dots, y_n \in Q$.
Set $$I=I_2\left(\begin{matrix}
x_{1}&x_{2} & \cdots &x_{n}\\
y_{1}&y_{2} & \cdots &y_{n}\\
\end{matrix}
\right) .$$
If $\grade (I, S)=n-1$, then $R=S/I$ is a $\GGL$ ring with respect to $(x_1, x_2, \dots, x_d)R$.
\end{thm}

\begin{proof}
We may reduce the assertion to the case where $n=d$. Indeed, assume that $n<d$. 
Then,  $S/(x_{n+1}, \dots, x_d)$ and $R/(x_{n+1}, \dots, x_d)R$ satisfy the same conditions of $S$ and $R$. 
Thus, we may assume that $n=d$ by Theorem \ref{a3.5}.
By the hypothesis, the Eagon-Northcott complex (\cite{EN}) gives a minimal $S$-free resolution of $R$. Recall that the Eagon-Northcott complex has the form
$$0\to S^{\oplus r} \xrightarrow{\mathbb{M}} S^{\oplus n(n-2)} \to \dots \to S \to R \to 0$$ of $R$ such that
\begin{equation*}
{}^t\!{\mathbb{M}}=\left(
\begin{array}{ccccc}
Y&&&& \\ 
X&Y&&\mbox{\huge{0}}&\\
&X& \ddots&  & \\
& & \ddots  & Y&\\
&\mbox{\huge{0}}&&X& Y\\
& & &  & X\\
\end{array}
\right) ,
\end{equation*}
where $X=(x_1~-x_2~x_3~\cdots~(-1)^{n-1}x_n)$ and $Y=(y_1~-y_2~y_3~\cdots~(-1)^{n-1}y_n)$ are submatrices of ${}^t\!{\mathbb{M}}$.
Since $y_{i} \in Q$ for all $1 \le i \le n$, by applying the elementary column operation, we have
\begin{equation*}
{}^t\!{\mathbb{M}} \sim \left(
\begin{array}{ccccc}
Y&&*&& \\ \hline
X&&&&\\
&X& & & \\
& & \ddots  & &\mbox{\huge{0}}\\
&&&X& \\
\mbox{\huge{0}}& & &  & X\\
\end{array}
\right)=:\mathbb{N},
\end{equation*}
where all entries in $*$ are elements of $Q$.
After replacing the basis of $S^{\oplus n(n-2)}$, we may assume that ${}^t\!{\mathbb{M}}=\mathbb{N}$.
Therefore, $R$ is a $\GGL$ ring with respect to $QR$ by Theorem \ref{a5.5}.
\end{proof}

\begin{cor}\label{a6.4}
Let $S$ be a Gorenstein local ring, and let $n$ be a positive integer with $3\le n\le \dim S=d$.
Then,  the following assertions hold true.
\begin{enumerate}[{\rm (1)}]
\item Let $\alpha_1, \alpha_2, \dots , \alpha_n, \beta_1, \beta_2, \dots , \beta_n$ be positive integers, and \\
let $x_1, x_2, \dots, x_n, z_1, z_2, \dots, z_{d-n}\in \fkn$ be a system of parameters of $S$. Set 
$$I=I_2\left(\begin{matrix}
x_{1}^{\alpha_1}&x_{2}^{\alpha_2} & \cdots &x_{n-1}^{\alpha_{n-1}}&x_{n}^{\alpha_n}\\
x_{2}^{\beta_2}&x_{3}^{\beta_3} & \cdots &x_{n}^{\beta_n}&x_1^{\beta_1}\\
\end{matrix}
\right) .$$
Then, $R_1=S/I$ is a $\GGL$ ring with respect to $(x_1^{\alpha_1}, x_2^{\alpha_2}, \dots, x_n^{\alpha_n}, z_1, z_2, \dots, z_{d-n})R_1$ if and only if $\alpha_i \le \beta_i$ for all $1\le i \le n$.
\item  Let $x_1, x_2, \dots, x_n, y_1, y_2, \dots, y_n, z_1, z_2, \dots, z_{d-2n}\in \fkn$ be a system of parameters of $S$. Set $$J=I_2\left(\begin{matrix}
x_{1}&x_{2} & \cdots &x_{n}\\
y_{1}&y_{2} & \cdots &y_{n}\\
\end{matrix}
\right).$$
Then, $R_2=S/J$ is a $\GGL$ ring with respect to 
\[
(x_1, x_2, \dots, x_n, y_1, y_2, \dots, y_n, z_1, z_2, \dots, z_{d-2n})R_2.
\]
\end{enumerate}

\end{cor}

\begin{proof}
(1): From Theorem \ref{a6.3} and Lemma \ref{a6.2}, we have only to prove the ``only if" part. 
By Lemma \ref{a6.2}, the Eagon-Northcott complex induces the exact sequence
$S^{\oplus n(n-2)} \xrightarrow{{}^t\!{\mathbb{M}}} S^{\oplus r} \to \rm{K}_R \to 0$,
where
\begin{equation*}
{}^t\!{\mathbb{M}}=\left(
\begin{array}{ccccc}
Y&&&& \\ 
X&Y&&\mbox{\huge{0}}&\\
&X& \ddots&  & \\
& & \ddots  &Y &\\
&\mbox{\huge{0}}&&X& Y\\
& & &  & X\\
\end{array}
\right).
\end{equation*}
Here, $X$ and $Y$ denote the submatrices of ${}^t\!{\mathbb{M}}$ such that
\begin{center}
$X=(x_1^{\alpha_1}~-x_2^{\alpha_2}~x_3^{\alpha_3}~\cdots~(-1)^{n-1}x_n^{\alpha_n})$ and $Y=(x_2^{\beta_2}~-x_3^{\beta_3}~x_4^{\beta_4}~\cdots~(-1)^{n-2}x_n^{\beta_n}~(-1)^{n-1}x_1^{\beta_1}).$
\end{center}
Set $\fka=(x_1^{\alpha_1}, x_2^{\alpha_2}, \dots, x_n^{\alpha_n}, z_1, z_2, \dots, z_{d-n})$. Apply the functor $S/\fka \otimes_S *$ to the representation of the canonical module $\rmK_R$. Then, because $\rm{K}_R/\fka \rm{K}_R$ is $S/\fka$-free, we have that all entries in $Y$ are in $\fka$. Thus, we obtain the assertion.

(2): Set $y_0=y_n$ and $\fkq=(x_i-y_{i-1} \mid 1 \le i \le n)$. Note that $x_i-y_{i-1}$ for $1 \le i \le n$ is a regular sequence of $S$ and $R_2$. Hence, $S/(J + \fkq)$ is a $\GGL$ ring with respect to $(x_1, x_2, \dots, x_n, z_1, z_2, \dots, z_{d-n}){\cdot}S/(J+\fkq)$ by (1). This implies that $R_2$ is also a $\GGL$ ring with respect to $(x_1, x_2, \dots, x_n, y_1, y_2, \dots, y_n, z_1, z_2, \dots, z_{d-2n})R_2$ by Theorem \ref{a3.5}.
\end{proof}

The following is an application to Rees algebras. It is known that for a Gorenstein local ring $(S, \fkn)$, the Rees algebra of a parameter ideal $Q$ of $S$ is an almost Gorenstein graded ring if and only if $S$ is regular and $Q=\fkn$ (\cite[Theorem 8.3]{GTT}). The following asserts that all such Rees algebras are $\GGL$ rings after localizing the graded maximal ideals.

\begin{cor}\label{a6.5}
Let $(S, \fkn)$ be a Gorenstein local ring and $1\le n\le \dim S=d$.
Let $a_1, a_2, \dots, a_d$ be a system of parameters of $S$.
Set $Q=(a_1, a_2, \dots, a_n)$, and let
$$\calR (Q):=S[a_1t, a_2t, \dots, a_nt]\subseteq S[t]$$
denote the Rees algebra of $Q$, where $S[t]$ is the polynomial ring over $S$.
Then,  $\calR (Q)_M$ is a $\GGL$ ring, where $M=\fkn \calR(Q) + \calR(Q)_+$ is the unique graded maximal ideal.
\end{cor}

\begin{proof}
By \cite{B}, we have that $\calR(Q) \cong S[T_1, T_2, \dots, T_n]/I_2\left(\begin{smallmatrix}
T_{1}&T_{2} & \cdots &T_{n}\\
a_{1}&a_{2} & \cdots &a_{n}\\
\end{smallmatrix}
\right)$, where $S[T_1, T_2, \dots, T_n]$ is the polynomial ring over $S$. Hence, $\calR(Q)_M$ is a $\GGL$ ring by Corollary \ref{a6.4} (2).
\end{proof}

We note several examples of rings having the condition discussed in this paper.

\begin{prop}\label{a7.6}
Let $(S, \fkn)$ be a Gorenstein local ring and $3\le n=\dim S$.
Let $x_1, x_2, \dots, x_n \in \fkn$ be a system of parameters of $S$. Set
$$I=I_2\left(\begin{matrix}
x_{1} & x_{2} & \cdots &x_{n-1} &x_{n}\\
x_{2} & x_{3} & \cdots &x_{n} &x_1\\
\end{matrix}
\right) $$
and $R=S/I$. Then,  we have the following.
\begin{enumerate}[{\rm (1)}]
\item $R$ is a one-dimensional $\GGL$ ring with respect to $(x_1, x_2, \dots, x_n)R$.
\item $\tr_R(\rmK_R)=(x_1, x_2, \dots, x_n)R$ is an Ulrich ideal of $R$.
\end{enumerate}

\end{prop}

\begin{proof}
(1): This follows from Corollary \ref{a6.4}.

(2): By (1) and Corollary \ref{a7.1.1}, we have the equality $\tr_R(\rmK_R)=(x_1, x_2, \dots, x_n)R$. Set $J=\tr_R(\rmK_R)$. Then,  we have 
$$
J^2=\ol{x_1}J + (\ol{x_2}, \ol{x_3}, \dots, \ol{x_n})^2=\ol{x_1}J
$$
because $\ol{x_i}\ol{x_j}=\ol{x_{i-1}}\ol{x_{j+1}}$ for all $2 \le i \le j \le n$, where $\ol{x}$ denotes the image of $x\in S$ in R and $\ol{x_{n+1}}=\ol{x_1}$. Therefore, by Theorem \ref{a7.5} and Lemma \ref{a4.18}, $J$ is an Ulrich ideal of $R$.
\end{proof}

\begin{Example}\label{a7.10}
Let $(S, \fkn)$ be a Gorenstein local ring of dimension $4$, and let $x_1, x_2, x_3, x_4 \in \fkn$ be a system of parameters of $S$. Set
$$I=\rmI_2\left(\begin{matrix}
x_{1} & x_{2} & x_{3}\\
x_{2} & x_{3} & x_{4}\\
\end{matrix}
\right) $$
and $R=S/I$. Then $R$ is a $\GGL$ ring with respect to $\tr_R(\rmK_R)$ of $\dim R=2$. Furthermore, $\tr_R(\rmK_R)=(x_1, x_2, x_3, x_4)R$ is an Ulrich ideal of $R$.
\end{Example}

\begin{proof}
By \cite[Corollary 3.4]{HHS} and Hilbert-Burch's theorem, we have $\tr_R(\rmK_R)=(x_1, x_2, x_3, x_4)R$. Set $J=(x_1, x_2, x_3, x_4)R$, and let $\ol{*}$ denote the image of $*\in S$ in $R$. Then,  
$$
J^2=(\ol{x_1}, \ol{x_4})J + (\ol{x_2}, \ol{x_3})^2=(\ol{x_1}, \ol{x_4})J.
$$
Furthermore, by letting $S'=S/(x_1, x_4)$, we have
\begin{align*}
\rme_J^0 (R)=&\ell_R(R/(x_1, x_4)R)
=\ell_S (S/I+(x_1, x_4))
=\ell_S (S/[(x_2, x_3)^2+(x_1, x_4)]) \\
=&\ell_{S'} (S'/(x_2, x_3)^2S') 
=3{\cdot}\ell_{S'} (S'/(x_2, x_3)S')
=3{\cdot}\ell_{R} (R/JR),
\end{align*}
where the fifth equality follows from the fact that $(x_2, x_3)S'$ is a parameter ideal of $S'$.
Hence, $J$ is an Ulrich ideal of $R$ by \cite[Lemma 2.3]{GOTWY}.
\end{proof}


\begin{acknowledgments}
We started this project around 2018, inspired by the paper \cite{CGKM}. Discussions proceeded as the second author participated in the first author's seminar. The second author is extremely grateful to the first author. The authors would like to thank Naoyuki Matsuoka, Naoki Endo, and Ryotaro Isobe for their valuable advice and comments in the seminar. 
The authors are grateful to the anonymous referee for the careful reading and valuable comments.
\end{acknowledgments}



\end{document}